\documentclass[reqno, 11pt]{amsart}

\usepackage{amsmath,amssymb,amsthm,amsfonts, amscd}
\usepackage{appendix}
\usepackage{hyperref}
\usepackage[numeric]{amsrefs}

\usepackage[dvipsnames]{xcolor}
\usepackage{graphics}
\usepackage{graphicx}

\newcommand\alp{\alpha}

\newcommand\lam{\lambda}		\newcommand\Lam{\Lambda}

\newcommand\calK{{\mathcal{K}}}

\newcommand\calO{{\mathcal{O}}}

		\newcommand\bfA{{\mathbf A}}
		\newcommand\bfB{{\mathbf B}}

		\newcommand\bfG{{\mathbf G}}

		\newcommand\bfU{{\mathbf U}}

		\newcommand\bfX{{\mathbf X}}

\newcommand\ZZ{\mathbb{Z}}

\newcommand\CC{\mathbb{C}}

	\newcommand\grg{{\mathfrak{g}}}

\newcommand\sdp{\times \hskip -0.3em {\raise 0.3ex
\hbox{$\scriptscriptstyle |$}}} 

\newcommand\Aut{\operatorname{Aut}}

\newcommand\Hom{\operatorname {Hom}}

\newcommand\im{\operatorname {im}}

\newcommand\ord{\operatorname{ord}}

\newcommand\Supp{\operatorname{Supp}}

\renewcommand{\>}{\rangle}

\newcommand\aff{\operatorname{aff}}

\newcommand\kk{\mathbf{k}}
\renewcommand\O{\calO}

\renewcommand\ord{\operatorname{ord}}

\DeclareMathOperator{\Rep}{Rep}
\newcommand{\ve}{\mathbf{v}}

\newtheorem*{de}{Definition}
\newtheorem*{nthm}{Theorem}
\newtheorem*{nlem}{Lemma}
\newtheorem*{nprop}{Proposition}
\newtheorem*{ncor}{Corollary}
\newtheorem*{nclaim}{Claim}
\theoremstyle{remark}

\newcommand{\be}[1]{\begin{eqnarray} \label{#1}}

\newcommand{\ee}{\end{eqnarray}}

\newcommand{\tpoint}[1]{\subsubsection{#1}}

\newcommand{\spoint}{\subsubsection{}}

\numberwithin{equation}{section}

\newcommand{\ch}{H(G_+, I)}

\newcommand{\mps}{\mathbb{M}}

\newcommand{\dhec}{\mathbb{H}}

\newcommand{\chs}{H_{\leq}(G_+, K)}
\newcommand{\mhs}{M_{\leq}(G, K)}
\newcommand{\xv}{\xi^{\vee}}
\newcommand{\ac}{a^{\vee}}
\newcommand{\tv}{\tau^{\vee}}

\newcommand{\av}{a^{\vee}}

\newcommand{\aw}{\mathcal{W}}
\newcommand{\ar}{\mathcal{R}}
\newcommand{\dd}{\mathbf{d}}
\newcommand{\cc}{\mathbf{c}}

\newcommand{\I}{I}
\renewcommand{\im}{{I^{-}}}
\newcommand{\ifib}{F^{I^-}_{w, \lv}(\mv)}
\newcommand{\kfib}{F^{K}_{\lv}(\mv)}

\newcommand{\tpb}{\tp^{\bullet}}
\newcommand{\valu}{\mathsf{val}}
\renewcommand{\v}{\mathbf{v}}
\newcommand{\la}{\langle}
\newcommand{\ra}{\rangle}
\newcommand{\lv}{\lambda^{\vee}}
\newcommand{\mv}{\mu^{\vee}}
\newcommand{\res}{\mathbf{k}}
\newcommand{\mf}[1]{\mathfrak{#1}}
\newcommand{\f}[1]{\mathfrak{#1}}
\newcommand{\rr}{\rightarrow}
\newcommand{\mc}[1]{\mathcal{#1}}
\newcommand{\wt}[1]{\widetilde{#1}}
\newcommand{\tp}{\theta}
\newcommand{\lambdav}{\lambda^\vee}
\newcommand{\muv}{\mu^{\vee}}
\newcommand{\Lv}{\Lambda^{\vee}}

\newcommand{\w}{\wt{w}}

\newcommand{\K}{\mathcal{K}}
\newcommand{\zee} {\mathbb{Z}}
\newcommand{\C} {\mathbb{C}}

\renewcommand{\O}{\mathcal{O}}
\newcommand{\qn}{{\mathcal Q}_{v}}

\newcommand{\iw}{\mathfrak{T}_w}

\begin{document}

\title{Iwahori-Hecke Algebras for $p$-adic Loop Groups}
\author{A. Braverman, D. Kazhdan, M. Patnaik}

\maketitle

\begin{abstract} This paper is a continuation of \cite{bk} in which the first two authors have introduced the {\em spherical Hecke algebra} and the Satake isomorphism for an untwisted affine Kac-Moody group over a non-archimedian local field. In this paper we develop the theory of the {\em Iwahori-Hecke algebra} associated to these same groups. The resulting algebra is shown to be closely related to Cherednik's double affine Hecke algebra. Furthermore, using these results, we give an explicit description of the affine Satake isomorphism, generalizing Macdonald's formula for the spherical function in the finite-dimensional case. The results of this paper have been previously announced in \cite{bk-ecm}.

\end{abstract}
%


\setcounter{tocdepth}{1}
\begin{small}
\tableofcontents
\end{small}


\section{Introduction}

Let $\calK$ denote a local non-archimedian field with ring of integers $\calO$ and pick $
\pi \in \O$ a uniformizing element. We denote by $\kk$ the residue field $\kk:= \O/ \pi \O,$ which is a finite field of size $q.$

Usually we shall denote algebraic varieties over $\calK$ (or a subring of $\calK$) by boldface letters $\bfX, \bfG$ etc.; their
sets of $\calK$-points will then be denoted $X,G$ etc.

\subsection{Finite Dimensional Case}
We shall first describe the finite-dimensional case of which this paper is an affine generalization.

\subsubsection{Notations on Groups}

Let $\bfG_o$ be a split, simple, and simply connected algebraic group (defined over $\ZZ$) and let $\grg_o$ be its Lie
algebra. As agreed above, we set $G_o=\bfG_o(\calK)$. Let $\bfA_o\subset \bfG_o$ be a maximal split torus, which we assume is of rank $\ell$; we denote its character lattice by $\Lam_o$ and its
cocharacter lattice by $\Lam_o^{\vee}$; note that since we have assumed that
$\bfG_o$ is simply connected, $\Lam_o^{\vee}$ is also the coroot lattice of $\bfG_o$.  For any $x\in \calK^*,\lam^{\vee}\in \Lam_o^{\vee}$
we set $x^{\lam^{\vee}}=\lam^{\vee}(x)\in A_o$.

Let us choose a pair $\bfB_o,\bfB_o^-$ of opposite Borel subgroups such that
$\bfB_o\cap\bfB_o^-=\bfA_o$. We denote by $R_o$ the set of roots of $\bfG_o$ and by
$R_o^{\vee}$ the set of coroots. Similarly $R_{o,+}$ (resp. $R^{\vee}_{o,+}$) will denote the set
of positive roots (resp. of positive coroots), and $\Pi_o$ (resp. $\Pi_o^{\vee})$ the set of simple roots (resp. simple coroots). We shall also denote by $2\rho_o$ (resp. by
$2\rho_o^{\vee}$) the sum of all positive roots (resp. of all positive coroots). We denote by $\Lv_o$ (resp. $\Lv_{o, +}$)  the set of dominant coweights. Let $W_o$ be the Weyl group of $G,$ which is a finite Coxeter group with generators simple reflections $w_1, \ldots, w_{\ell}$ corresponding to the simple roots $\Pi_o.$

For a reductive group $\bf{H}$ over $\mc{K}$ we denote by $H^{\vee}$ the Langlands dual group defined over $\C$ defined by exchanging the root and coroot data of $\bf{H}.$

\subsubsection{Hecke algebras}Let $J_o \subset G_o$ be an open compact subgroup of $G_o$. Then one can consider the Hecke
algebra $H(G_o,J_o)$ of $J_o$-bi-invariant compactly supported functions with respect
to convolution. Studying the representation theory of $G_o$ is essentially equivalent to studying the representation theory of the algebras $H(G_o,J_o)$ for various $J_o.$

There are two
choices of open compact subgroups of $G_o$ that will be of interest to us. The first is that of $K_o=\bfG_o(\calO);$  the corresponding Hecke algebra $H(G_o,K_o)$ is called the
{\em spherical Hecke algebra}. The second is that of the Iwahori subgroup $I_o \subset K_o,$ which is
by definition the subgroup of $K_o$ equal to the preimage of $\bfB_o(\kk) \subset \bfG(\kk)$ under the natural projection map $K_o \to \bfG_o(\kk)$. The corresponding Hecke algebra $H(G_o, I_o)$is called the {\em Iwahori-Hecke algebra.}  Let us recall the description of the corresponding algebras in these two cases.

\subsubsection{Spherical Hecke Algebra and the Satake isomorphism}  \label{fin-sph}

The Cartan decomposition asserts that $G$ is the disjoint union of double cosets $K_o \; \pi^{\lam^{\vee}} K_o, \, \lam^{\vee} \in \Lam_{o, +}^{\vee};$ hence, $H(G_o, K_o)$  has a vector space basis corresponding to the characteristic functions of these double cosets  $h_{\lam^{\vee}}, \, \lv \in \Lv_{o, +}.$ As an algebra, $H(G_o, K_o)$ is commutative, associative, and unital, with unit $\mathbf{1}_{K_o}$ equal to the characteristic function $h_{\lv}$ with $\lv=0$ (i.e. the characteristic function of $K_o$).

Let $\CC[\Lv_o]$ denote the group algebra of $\Lv_o:$ it consists of finite $\C$-linear combinations in the symbols $e^{\lv}$ with $\lv \in \Lv_o$, where $e^{\lv} e^{\mv} = e^{\lv + \mv},$ for $\lv, \mv \in \Lv_o.$ The natural $W_o$-action on $\Lv_o$ lifts to $\CC[\Lv_o];$ for $f \in \CC[\Lv_o]$ and $w \in W$ we denote by $f^w$ the application of $w$ to $f.$ The Satake isomorphism $S_o$ makes clear the algebra structure of $H(G_o, K_o)$: it provides a canonical isomorphism (see \cite{sat}) \be{sat:intro} S_o: H(G_o, K_o) \stackrel{\cong}{\rr} \CC[\Lv_o]^{W_o}, \ee where $\CC[\Lv_o]^{W_o}$ is the ring of $W_o$-invariant elements in $\CC[\Lv_o].$  The algebra $\CC[\Lv_o]^W$ admits other interpretations: it is isomorphic to the complexified Grothendieck ring $K_0(\Rep(G_o^{\vee}))$
of finite-dimensional representations of $G_o^{\vee},$ the dual group of $G_o$; it is also isomorphic to the algebra $\CC(A_o^{\vee})^{W_o}$ of polynomial functions on the maximal torus
$A_o^{\vee}\subset G_o^{\vee}$ which invariant under $W_o$.

For many purposes, it is desirable to have an explicit formula for the elements $S_o(h_{\lv}).$ Such a formula was given by Macdonald (and independently by Langlands \cite[Chapter 3]{lan:ep} in a slightly weaker form), and we shall present the answer below. The arguments which we know for this formula (implicitly or explicitly) go through the Iwahori-Hecke algebra.

\subsubsection{The Iwahori-Hecke algebra}\label{fin-iwahori}
As follows from the work of Iwahori and Matsumoto \cite{im}, the group $G_o$ is the disjoint union of cosets $I_o$-double cosets indexed by $\mc{W}_o:= W_o \rtimes \Lam_o^{\vee},$ the {\em affine Weyl group associated to $W_o$}\footnote{recall that we have assumed that $\bfG_o$ is simply-connected, so that $\mc{W}_o:= W_o \rtimes Q_o^{\vee}$ where $Q_o^{\vee}$ is the coroot lattice}.
It is well-known (\cite{mac:aff}) that $\mc{W}_o$ is itself an infinite Coxeter group which has simple reflection generators $w_1, \ldots, w_{\ell+1}$ where $w_1, \ldots, w_{\ell}$ correspond to the previously introduced generators of $W_o.$ Denote by $\ell: \mc{W}_o \rr \zee$ the length function on $\mc{W}_o$ corresponding to this set of generators, and also let $T_x$ be the characteristic function of the double coset corresponding to $x \in \mc{W}_o.$  Then in \emph{loc. cit}, it was shown that the algebra $H(G_o, I_o)$ has the following simple presentation: it is generated by
 $\{ T_x \}_{x \in \mc{W}_o}$ and has relations \begin{enumerate} \label{IM-relation} \item[\textbf{IM 1}] $T_x T_y = T_{x y}$ for $x, y \in \mc{W}_o$ with $\ell(x y) = \ell(x) + \ell(y)$ \item[\textbf{IM 2}] $T_{w_i}^2 = q T_{1} + (q-1)T_{w_i} =0$ for $i=1, \ldots, \ell+1.$  \end{enumerate}

The algebra $H(G_o,I_o)$ has an important alternative  description, the \emph{Bernstein presentation}  : it is generated by elements $\Theta_{\lv}$ for $\lv\in \Lv_{o}$ and $T_w$ for $w\in W_o$, subject to the following
relations:
\begin{enumerate} \label{bern-relation}

\item[ \textbf{B 1}] $T_wT_{w'}=T_{ww'}$ for $w, w' \in W_o$ with $\ell(ww')=\ell(w)+\ell(w')$;

\item[\textbf{B 2}] $\Theta_{\lv} \Theta_{\mv}=\Theta_{\lv+\mv}$; in other words, the $\Theta_{\lv}$'s generate a (commutative) subalgebra $\CC[\Lv_o]$ inside $H(G_o,I_o)$;

\item[\textbf{B 3}] For any $f\in \CC[\Lv_o]$ and any simple reflection $w_i$ for $i=1, \ldots, \ell$ we have
\be{bernrel-intro}
f T_{w_i} -T_{w_i} f^{w_i} =(1-q)\frac{f- f^{w_i}}{1-\Theta_{-a_i^{\vee}}} \ee where $a_i^{\vee}$ are the simple coroots of $G_o$.  Note that the right hand side of the above equation is an element of $\CC[\Lv_o]$. \end{enumerate}

\subsubsection{Explicit description of the Satake isomorphism}\label{fin-mac}
For any subset $\Sigma_o \subset W_o$ we define \be{poin-poly} \Sigma_o(q^{-1}) = \sum_{w \in \Sigma_o} q^{ - \ell(w)}.\ee
For any $\lv \in \Lv_o$, let $W_{o, \lv}$ denote the stabilizer of $\lv$ in $W_o.$  The following result is due to  Macdonald.

\begin{nthm}\label{macdonald:finite} \cite{mac:mad}
For any $\lv \in \Lv_{o,+}$ we have
\begin{equation}\label{fin:mac}
S_o(h_{\lv})=\frac{q^{\la \rho_o, \lv  \ra}}{W_{o, \lv}(q^{-1})}\sum\limits_{w\in W_o}
w\Bigg(e^{\lv}\frac{\prod\limits_{\alp\in R_{o,+}} 1-q^{-1}e^{-\alp^{\vee}}}{\prod\limits_{\alp\in R_{o,+} } 1-e^{-\alp^{\vee}}}\Bigg),
\end{equation}
where recall that $\rho_o$ was defined as the half-sum of the positive roots of $G_o$.
\end{nthm}

Note that it is not immediately clear that the right hand side of (\ref{fin:mac}) belongs to  $\CC[\Lv_o]$.  Of course this follows from the theorem as the left hand side of (\ref{fin:mac}) is in $\CC[\Lv_o].$

Since $S_o$ is an algebra map, it sends the identity in the Hecke algebra $\mathbf{1}_{K_o}$ (i.e., the characteristic function $h_{\lv}$ with $\lv=0$) to the identity $1 \in \CC[\Lv_o]^{W_o}.$ Specializing (\ref{fin:mac}) to $\lv=0$ we obtain a non-trivial combinatorial identity \be{fin:mac:iden} S_o(\mathbf{1}_{K_o}) = 1 = \frac{1}{W_o(q^{-1})} \sum\limits_{w\in W_o}
w\Bigg(\frac{\prod\limits_{\alp\in R_{o,+}} 1-q^{-1}e^{-\alp^{\vee}}}{\prod\limits_{\alp\in R_{o,+} } 1-e^{-\alp^{\vee}}}\Bigg),\ee  (see \cite{mac:poin}). We emphasize this point, as the naive analogue of the above identity \emph{fails} in the affine setting.

\subsection{The Affine Case}

The main purpose of this paper is to extend the results described in $\S \S$ \ref{fin-iwahori}-- \ref{fin-mac} to the case of (untwisted) affine Kac-Moody groups, since the results in $\S$ \ref{fin-sph} having been generalized already to the affine setting by the first two authors in \cite{bk}.

\subsubsection{Notations on Loop Groups}

As before, we start with a split, simple, simply connected group  $\bfG_o$. Fix a symmetric, bilinear form $(\cdot, \cdot)$ on the coroot (or coweight) lattice of $\bfG_o$ (which is specified in \S \ref{finrtsys}). For this fixed choice of $(\cdot, \cdot)$ the  polynomial loop group $\bfG_o[t, t^{-1}]$ (i.e. the functor whose points over a ring $R$ are given by $\bfG_o(R[t, t^{-1}])$) admits a non-trivial central extension by $\mathbb{G}_m$ which we denote by $\wt{\bfG}.$ The full affine Kac-Moody group is then $ \bfG :=\mathbb{G}_m \ltimes \wt{\bfG}$ \footnote{The reader should be warned that in the main body of this paper a slightly different (but equivalent) construction is adopted coming from the general theory of Kac-Moody groups}, where $\mathbb{G}_m$ acts by rescaling the loop parameter $t.$ Denote by \be{eta-intro} \eta: \bfG \rr \mathbb{G}_m \ee the projection onto the rescaling parameter. We  choose a pair $\bfB ,\bfB_-$ of opposite Borel subgroups of $\bfG$ (see \S \ref{section-loopgroups} for precise definition of this notion) whose intersection $\bfA = \bfB \cap \bfB_-$ is  equal to the group $\bfA = \mathbb{G}_m \times \bfA_o \times \mathbb{G}_m$ where the first $\mathbb{G}_m$ corresponds to the central direction and the second to the rescaling parameter. Let $R$ (resp. $R^{\vee}$) denote the set of roots (resp. coroots) and $\Pi$ (resp. $\Pi^{\vee}$) the set of simple roots (resp. coroots) of $\bfA$. Recall that the roots (resp. coroots) come in two flavours: the real roots (resp. coroots) shall be denoted by $R_{re}$ (resp. $R_{re}^{\vee}$) and the imaginary ones by $R_{im}$ (resp. $R_{im}^{\vee}$). The minimal imaginary coroot shall be denoted by $\cc,$ and corresponds to the central extension.  Also let $\Lv$ denote the cocharacter lattice of $ \bfA$ instead of $A$ and $\Lv_+$ the set of dominant cocharacters. We have that \be{cowts-intro} \Lv \cong \zee \oplus \Lv_o \oplus \zee, \ee where again the first $\zee$-component corresponds to the center and the last to the loop rescaling. As before, let $\CC[\Lv]$ be the group algebra of $\Lv.$ We denote by $W$ the Weyl group of $G.$ The group $W$ is an infinite Coxeter group,  \footnote{The group $W$ also admits a description as an affine Weyl group associated to $W_o,$ namely $W \cong W_o \rtimes Q_o^{\vee},$ though we shall not use this description in what follows.} which leaves invariant the imaginary roots and coroots. We define the \emph{Tits cone} $X$ as the union $ \cup_{w \in W} w (\Lv_+).$ Unlike  the finite-dimensional case, $X \neq \Lv;$ in terms of the decomposition (\ref{cowts-intro}), the Tits cone is characterized as the subset of all elements $(a,\lv_o,k)\in \zee \oplus \Lv_o \oplus \zee$ such that either $k>0$ or $k=0$ and $\lv_o=0$.

For the groups $\bf{G}$ or $\bf{A}$ we denote by $G^{\vee}$ and $A^{\vee}$ the Langlands dual group (defined over $\C$) {[see \cite{bk} for more details]. If  $\bf{G}_o$ is simply-laced, then} the dual group to its affinization $\bf{G}$ is again an untwisted affine Kac-Moody group. But in  general  the dual group $G^{\vee}$ is a twisted affine Kac-Moody group. To avoid  technical complications that we assume that  $\bfG_o$ simply-laced.

\subsubsection{$p$-adic Loop Groups} \label{intro-Gplus}

The main object  of  interest is the group  $G=\bfG(\calK).$ It was observed in \cite{bk} that one should work with a certain semigroup $G_+ \subset G.$ The importance of  this semigroup was also recognized earlier by Garland \cite{gar:car}
who showed that $G_+ \subset G.$ is the subset of $G$ on which an analogue of the Cartan decomposition holds.

To describe $G_+,$ recall the map $\eta: \bfG \rr \mathbb{G}_m$ from (\ref{eta-intro}) which induces the map \be{eta:Kpts} | \eta | : G \rr \mc{K}^* \stackrel{\valu}{\rr} \mathbb{Z} \ee where the last map is the valuation map $\valu: \mc{K}^* \rr \zee.$ We define  $G_+ \subset G$ as the sub-semigroup of $G$ generated by the following three types of elements:

\begin{enumerate}
\item the central $\calK^*\subset T \subset G$;

\item the subgroup $\bfG(\calO) \subset G$;

\item All elements $g \in G$ such that $ |\eta(g)| >0$.
\end{enumerate}

We define an affine analogue of $K_0$ and $I_0$ as $K:= \bf{G}(\O) \subset G_+$ and $I:= \pi_{\kk}^{-1}(\bfB(\kk)) \subset K$ where $\pi_{\kk}: \bf{G}(\O) \rr \bf{G}(\kk)$ is the natural projection. The following result, proven in \S3, generalizes the Cartan and Iwahori-Matsumoto decompositions of the theory of $p$-adic groups (see e.g., \cite{mac:mad}),

\begin{nprop} \label{intro-doublecosets}There are bijective correspondences between the following sets
\begin{enumerate}
\item[(a)] $\Lv_+$ and the set of double cosets $K \setminus G_+ / K$
\item[(b)] $\mc{W}_X:= W \ltimes X$ and the set of double cosets $I \setminus G_+ / I$ (recall that $X$ was defined to be the Tits cone). \end{enumerate} \end{nprop}

Part (a) of this Proposition follows form the results in \S \ref{sec-cartan} and Part (b) from those of \S \ref{sec-im}.
\subsubsection{Spherical Hecke Algebras and the Satake Isomorphism}
Due to the infinite-dimensionality of $K$ and $I,$ one cannot resort to the usual techniques to define a convolution structure on the space of $I$ or $K$-double cosets on the group. However, it was shown in \cite{bk} an associative algebra structure can still be defined on a certain \emph{ completion} of the space of finite linear combinations of $K$-double cosets of $G_+$. We would like to emphasize that this claim is by no means trivial. The main point is that in order to define an algebra structure, one has to show that  fibers of a certain convolution diagrams are finite. In \emph{loc. cit} this is proved using an algebro-geometric interpretation; alternatively, a more elementary (though not particularly short) proof of this finiteness can be found in \cite{bgkp}; more recently, a generalization of this result to arbitrary symmetrizable Kac-Moody groups was achieved in \cite{gau:rou} by developing the notion of buildings (the so called \emph{hovels}) for these groups.

Denote by $H_{\leq}(G_+, K)$ the {\it spherical Hecke algebra} of $K$-double cosets of $G_+,$ where the subscript $\leq$ denotes that a certain completion (depending on the dominance order $\leq$ on $\Lv$) of the space of finitely supported functions is necessary. The precise definition is reviewed in $\S5.$ Note that the commutative algebra  $H_{\leq}(G_+,K)$ is  unital with $\mathbf{1}_K.$ (the characteristic function of $K.$) as the unit. The algebra $H_{\leq}(G_+, K)$ comes with two additional structures: a grading by non-negative integers which comes from the map $|\eta|$ of (\ref{eta:Kpts}) and a structure of an algebra over the field $\CC((\mathsf{t}))$ of Laurent power series in a variable $\mathsf{t}$ (not to be confused with the loop variable 't'), which comes from the center of $G$.

%

The statement of the Satake isomorphism for $G$ is very similar to that for $G_o$.
First, in \cite{bk} the natural analogue of the algebra $\CC[\Lv_o]^{W_o}$ was defined (it had also made its appearance in the literature earlier by Looijenga \cite{loo}). The definition again involves a certain completion $\C_{\leq}[\Lv]$ of group algebra of $\Lv.$ We shall denote the corresponding space of $W$-invariants here by $\CC_{\leq}[\Lv]^{W}.$

This algebra also has natural interpretations  in terms of dual groups: either as the subring  $\CC[\Lv]^W$ of $W$-invariants in a certain completion of the space of polynomial functions on $T^{\vee}$
or as a suitable category $\Rep(G^{\vee})$ of representations of $G^{\vee}$ (stable under tensor product) so that $K_o(\Rep(G^{\vee}) \cong \CC_{\leq}[\Lv]^W$ (see \S \ref{loo}).

The algebra $\CC_{\leq}[\Lv]^W \cong \C_{\leq}(T^{\vee})^W$ is aalso  finitely generated $\ZZ_{\geq 0}$-graded commutative algebra over the field $\CC((\mathsf{t}))$ of Laurent formal power series in the variable $\mathsf{t}$.
The affine Satake isomorphism asserts that
\begin{nthm} \label{aff-satake} \cite{bk}
There is a natural isomorphism of graded  $\CC((\mathsf{t}))$-algebras \be{aff-sat} S: H_{\leq}(G_+,K) \rr \CC_{\leq}[\Lv]^W \cong \CC_{\leq}(T^{\vee})^W. \ee
\end{nthm}

We shall present below a generalization of (\ref{fin-mac}) giving an explicit formula for $S$ on certain basis elements of $H_{\leq}(G_+, K).$

\subsubsection{The Iwahori-Hecke algebra for $G$} \label{aff-iwahori}

As we observed in Proposition \ref{intro-doublecosets}(b), $G_+$ can be written as a disjoint union of $I$-double cosets paramterized by the semigroup $\mc{W}_X:= W \rtimes X$ where $X:= \cup_{w \in W} w(\Lv_+) \subsetneq \Lv$ is the Tits cone of $\Lv.$ The semigroup $\mc{W}_X$ plays a role similar to $\mc{W}_o$ in the usual theory of $p$-adic groups; however it is not a Coxeter group. Although we largely circumvent a systematic study of the combinatorics of this group, but we do introduce here certain orders on this group and show how they arise naturally from a group theoretic point of view.

Denote by $H(G_+,I)$ the \emph{Iwahori-Hecke algebra} associated to $G,$ which is the space of $I$-bi-invariant functions on $G_+$ supported on a union of finitely many double cosets. It has a vector space basis $T_x$ for $x \in \mc{W}_X$ where $T_x$ is the characteristic function of the double coset $I x I, \, x \in \mc{W}_X.$ In this paper we show the following result, which follows from the finiteness theorems in \cite{bk} or \cite{bgkp}.

\begin{nthm} The space $H(G_+,I)$ can be naturally equipped with a convolution structure. \end{nthm}


Note the difference from the spherical case, where the convolution was only defined on a completion of some space of finitely-supported functions.

In the finite-dimensional case, one has two presentations for the algebra $H(G_o, I_o)$ as was described in \S \ref{fin-iwahori}, one in terms of the basis $\{ T_x \}_{x \in \mc{W}_o}$ and the relations $\textbf{IM 1, IM 2}$ , and the other in terms of a basis $\{ \Theta_{\lv}, T_w \}_{\lv \in \Lv_o, w \in W_o}$ subject to the relations \textbf{B1, B2, B3}. We do not know how to generalize the former, but $H(G_+, I)$ does admit a description similar to the latter. To describe this it, recall that we have from (\ref{cowts-intro}) that  $\Lv=\ZZ \oplus\Lv_o \oplus \ZZ.$ Let $\mathbb{H}$ denote the algebra generated by elements $\Theta_{\lv}$, $\lv \in \Lv$ and $T_w$ for $w\in W \subset \mc{W}_X$ with relations \textbf{B 1}, \textbf{B 2}, \textbf{B 3} as in \S \ref{fin-iwahori}. The algebra $\mathbb{H}$ is $\ZZ$-graded; this grading is defined by setting
\be{degs}
\deg T_w=0;\quad \deg \Theta_{(a,\lv,k)}=k \text{ where } (a, \lv, k) \in \Lv = \zee \oplus \Lv_o \oplus \zee.
\ee
We denote by $\mathbb{H}_{k}$ the space of all elements of degree $k$ in $\mathbb{H}$. Note that $\mathbb{H}_{0}$ is a subalgebra
of $\mathbb{H}$, which is isomorphic to Cherednik's double affine Hecke algebra.

On the other hand, set
$$
\mathbb{H}_+:=\CC\<T_w\>_{w\in W}\oplus\Bigg(\bigoplus\limits_{k>0}\mathbb{H}_{\aff,k}\Bigg).
$$
It is easy to see that $\mathbb{H}_+$ is just generated by elements $\Theta_{\lv}$, $\lv \in X$ and $T_w$, $w\in W$. The following is one of the two main results of this paper:
\begin{nthm}\label {aff-iw-theorem}
The algebra $H(G_+,I)$ is isomorphic to the algebra $\mathbb{H}_+.$
\end{nthm}
In particular, the algebra $H(G_+,I)$ is closely related to Cherednik's double affine Hecke algebra. We would like to note that another relation between the double affine Hecke algebra and the group $G$ was studied by Kapranov \cite{kap}. In \emph{op. cit}, the double affine Hecke algebra was constructed as an algebra of intertwining operators on some space built from $G.$ As such, it naturally comes equipped with a Bernstein type presentation. Our algebra $H(G_+, I)$  is first constructed as a convolution algebra of double cosets, and then shown to admit a Bernstein-type description by studying certain intertwining operators. Hence, by definition the algebra $H(G_+,I)$ is endowed with a natural basis corresponding to characteristic
functions of double cosets of $I$ on $G_+$. It is natural to conjecture that this basis has a purely algebraic or combinatorial description as well.

\subsubsection{Affine Macdonald formula}

As observed in Proposition \ref{intro-doublecosets} (see also Theorem \label{cartan-thmdefn}), the semi-group $G_+$ is equal to the disjoint union of $K$-double cosets $K \pi^{\lv} K, \, \lv \in \Lv_+.$ Denote again by $h_{\lv} \in H(G, K)$ the characteristic function of the corresponding double coset. Our second main result is an explicit formula for $S(h_{\lv})$ which generalizes (\ref{fin-mac}) above.

To state it, let $\CC_{v}:= \CC[[v]]$ denote the ring of Taylor series in the formal variable $v.$ Define $\CC_{v, \leq}[\Lv]:= \CC_{v} \otimes_{\CC} \CC_{\leq}[\Lv].$ The element \be{delta-intro} \Delta := \prod_{a \in R_+} \left( \frac{ 1 - v^{2} e^{-a^{\vee}} }{1 - e^{-a^{\vee} }} \right) ^{m(a^{\vee})} \ee where $m(a^{\vee})$ is the multiplicity of the coroot $a^{\vee}$ may be regarded as an element in $\CC_{v, \leq}[\Lv]$ by expanding each rational function in negative powers of the dimensions of the corresponding  coroot space. For $ w \in W$ we define $\Delta^w$ to be the expansion of the product $ \prod_{a \in R_+} \left( \frac{ 1 - v^{2} e^{-wa^{\vee}} }{1 - e^{-wa^{\vee} }} \right) ^{m(a^{\vee})}$ in negative powers of the coroots.  One can then show that the following element
\be{hlam}
H_{\lam^{\vee}}=\frac{v^{ -2\langle \rho, \lambda^{\vee} \rangle}}{W_{\lam^{\vee}}(v^{2}) }\sum_{w \in W} \Delta^w e^{w \lambda^{\vee}}
\ee
lies in $\CC_{v, \leq}[\Lv],$ where $\rho$ is the element of $\Lv$ such that $\la \rho,a _i^{\vee}\ra =1$
for any simple coroot $a_i^{\vee} \in \Pi^{\vee}$. The following is the other main new result of this paper,
\begin{nthm} \label{aff-mac} The element $\frac{H_{\lv}}{H_0}$ lies in the ring $\CC[v^{2}] \otimes_\C \C_{\leq}[\Lv],$ and its specialization at ${v^2=q^{-1}}$ is equal to $S(h_{\lv}).$
\end{nthm}

The proof is relies on  a study of the combinatorics of  the Iwahori-Hecke algebra (i.e., essentially combinatorics for the double affine Hecke algebra). In the finite-dimensional case the second equality of the identity (\ref{fin:mac:iden}) ensures that the analogue of $H_0$ is equal to $1.$ However, in the affine setting \be{fail} H_0 = \frac{1}{ W(v^2)} \sum_{w \in W} \Delta^w \neq 1, \ee although the analogue of the first equality in (\ref{fin:mac:iden}) must still hold. This explains the reason why we divide by $H_0$ in Theorem \ref{aff-mac}.  Observe that the function $H_0$ was studied by Macdonald in \cite{mac:formal} using the works of Cherednik. Macdonald has shown that $H_0$ has an infinite product decomposition. For example, when $G_o,$ the underlying finite-dimensional group of $G$ is of simply laced type, the Macdonald-Cherednik formula reads as follows:
\be{mac2003}
H_0=\prod\limits_{i=1}^{\ell}\prod\limits_{j=1}^{\infty} \frac{1-v^{2 m_i}e^{-j\cc}}{1-v^{2(m_i+1)}e^{-j\cc}},
\ee
where $\cc$ was the minimal positive imaginary coroot and $m_1,\cdots,m_{\ell}$ are the exponents of $G_o$ (defined  for example via (\ref{poin:fin}).  A similar product decomposition for $H_0$ exists for any $G_o,$ not necessarily simply-laced (see \cite{bfk}).

\subsection{Acknowledgements}A.B. was partially supported by the NSF grant DMS-1200807 and by Simons Foundation.
M.P. was supported by an NSF Postdoctoral Fellowship, DMS-0802940 and an University of Alberta startup grant while this work was being completed. D.K. was partially supported by the European Research Council grant 247049. We would like to thank I. Cherednik, P. Etingof, H. Garland, and E. Vasserot for useful discussions.

\section{Basic Notations on Groups and Algebras}

\subsection{Lie Algebras} \label{sec:loopalg}

\tpoint{Finite Dimensional Notations}  \label{finrtsys} Let $\mf{g}_o$ be a simple, simply-laced, split Lie algebra over $k$ of rank $\ell.$ In general the index $o$ will denote objects associated to a finite-dimensional root system. Choose a Cartan subalgebra $\mf{h}_o \subset \mf{g}_o$ and denote the set of roots with respect to $\mf{h}_o$ by $R_o.$ Choose a Borel subalgebra $\mf{b}_o$ and denote by  $\Pi_o = \{ \alpha_1, \ldots, \alpha_\ell \}$  the simple roots of $R_o.$ Let $\f{h}_o^*$ denote the algebraic dual of $\f{h}_o$ and denote the natural pairing $\f{h}_o^* \times \f{h}_o \rr k$ by $\la \cdot, \cdot \ra$. Let $\theta$ denote the highest root. Denote by $(\cdot, \cdot)$ the Killing form on $\f{h}_o,$ which induces an isomorphism $\psi: \f{h}_o \rr \f{h}_o^*.$ We continue to denote the induced form on $\f{h}_o^*$ by $(\cdot, \cdot),$  normalized so that $(\theta, \theta)=2.$ For each root $\alpha \in R_o$ denote the corresponding coroot by  $\alpha^{\vee}:= \frac{2}{(\alpha, \alpha)}\psi^{-1}(\alpha) \in \f{h}_o.$ For each $\alpha \in R_o$ we let $w_{\alpha}: \f{h}_o^* \rr \f{h}_o^*$ be the corresponding reflection, and denote by $W_o$ the Weyl group generated by the reflections $w_{\alpha_i}$ for $i=1, \ldots, \ell.$ Let $\Lambda_o \subset \f{h}_o^*$ denote the weight lattice of $\f{g}_o,$ defined as the set of $\lambda \in \f{h}_o^*$ such that $\la \lambda, \alpha_i^{\vee} \ra \in \zee$ for $i=1, \ldots, \ell.$  We define $\rho_o \in \Lambda$ by the condition that $\la \rho, \alpha_i \ra =1 $ for $i=1, \ldots, \ell.$ We let $Q_o \subset \mf{h}_o^*$ denote the root lattice, and observe that $Q_o \subset \Lambda_o.$ Dual to these notions, we denote by $\Lambda_o^{\vee}$ and $Q_o^{\vee}$ the coweight and coroot lattice of $\f{g}_o$ in the usual way.

\tpoint{Affine Lie Algebras} For a field $k,$ we  denote  by $\f{g}$ the affinization of the Lie algebra $\f{g}_o.$ As a vector space $\f{g}:= k \dd \oplus \mf{g}'$ where $\dd$ is the degree derivation and $\f{g}'$ is the one-dimensional central extension of the the loop algebra $ \f{g}_o \otimes_k k[t, t^{-1}]$ which is specified by the form  $(\cdot, \cdot)$ defined in \S  \ref{finrtsys}. Let $\mf{h} \subset \mf{g}$ denote a Cartan subalgebra containing the finite-dimensional Cartan $\mf{h}_o,$ the degree derivation $\dd,$ and the center of $\mf{g}.$  One has a direct sum decomposition \be{h:sum} \mf{h}:= \mf{h}_o \oplus \mf{h}_{cen} \oplus k \dd \ee where $\mf{h}_{cen}$ is the one dimensional $k$-vector space containing the center. One may equip $\mf{h}$ with a symmetric, non-degenerate bilinear form on $(\cdot | \cdot)$ as in \cite{kac}. Let $\f{h}^*$  be the algebraic dual of $\f{h}$. As before we denote by  \be{dualpair} \la \cdot, \cdot \ra \,: \f{h}^* \times \f{h} \rightarrow k\ee the natural pairing.

 Let $R$ be the set of roots of $\f{g},$ and $R^{\vee}$ the set of coroots. We denote the set of simple roots of $\mf{g}$  by  \be{simp:aff:roots} \Pi \ \ = \ \  \{ a_1, \ldots, a_{\ell+1} \}  \ \ \subset \ \  \f{h}^*\,. \ee Similarly, we write \be{simp:aff:coroots} \Pi^{\vee}  \ \ = \ \  \{ a_1^{\vee}, \ldots, a_{\ell+1}^{\vee} \}  \ \ \subset \ \ \f{h} \ee for the set of simple affine coroots.  Note that the simple roots $a_i: \f{h} \rr k$ satisfy the relations\begin{equation}\label{new2.16}
\aligned \la a_i , \dd \ra & \ \ = \ \   0  \ \ \text{ for } \ i \, =\, 1,\, \ldots, \,\ell \\ \text{and} \ \ \  \ \  \
\la a_{\ell+1} , \dd \ra &\ \ = \ \   1\,.
\endaligned
\end{equation}
Each root $\alpha \in R_o$ extends to an element of $\mf{h}^*$ which we denote by the same symbol $\alpha$ by requiring that $\la \alpha, X \ra =0,$ for $X \in \mf{h}_{cen} \oplus k \dd.$ Let $\delta \in \f{h}^{*}$ be the minimal positive imaginary root defined by the conditions, \be{delta:def}  \begin{array}{lcr} \la \delta, X \ra =0  \text{ for } X \in \f{h}_{cen} \oplus \f{h}_o & \text{ and } \la \delta ,\dd \ra =1. \end{array} \ee As is well known, we have $a_i =\alpha_i$ for $i=1, \ldots, l$ and $a_{\ell+1} = -\theta + \delta.$
%

We define the affine root lattice  as  \be{rootlat} Q  \ \ = \ \  \zee a_1  \ + \  \cdots  \ +  \ \zee a_{\ell+1} \ee and the affine coroot lattice as  \be{corootlat} Q^{\vee} \ \  = \ \  \zee a_1^{\vee}  \ + \  \cdots  \ + \  \zee a_{\ell+1}^{\vee} \,.\ee
We shall denote the subset of non-negative integral linear combinations of the affine simple roots  (respectively, affine simple coroots) as $Q_+$ (respectively,  $Q^{\vee}_+$).  The integral weight lattice is defined  by
\begin{equation}\label{integralweightlatticedef}
\Lambda  \ \  :=  \ \  \{ \lambda \in \f{h}^* \ | \   \langle \lambda, a_i^{\vee} \rangle \,  \in \,  \zee  \, \ \text{ for } \, i\,=\, 1, \ldots, \ell+1  \text{ and } \la \lambda, \dd \ra \in \zee \}\,.
\end{equation}

 The lattice $\Lambda$ is spanned by $\delta$ and the fundamental affine weights  $\Lambda_1, \ldots, \Lambda_{\ell+1},$ which are defined by the conditions that $\la \Lambda_i, \dd \ra =0 $ for $i=1, \ldots, \ell+1$ and  \be{aff:wts} \langle \Lambda_i, a_j^{\vee} \rangle \ \  = \ \ \left\{
                                                          \begin{array}{ll}
                                                            1\,, & i\,=\,j\, \\
                                                            0\,, & i\,\neq\,j\,
                                                          \end{array}
                                                        \right. \ \ \
 \text{ for ~} 1 \,\le \, i,j\, \le \, \ell+1 \,.\ee The dual space $\f{h}^*$ is spanned by $a_1, a_2, \cdots, a_{\ell+1}, \Lambda_{\ell+1}.$ We also let $\cc \in \mf{h}$ be the minimal imaginary coroot, characterized by the conditions, \be{c:char} \begin{array}{lcr} \la a , \cc \ra = 0, \,  a \in R &  \text{ and } &   \la \Lambda_{\ell+1}, \cc \ra = 1 \end{array}. \ee
%
%
%

We denote by  $\Lambda^{\vee}$  the coweight lattice  in  $\f{g}$ which is defined as \be{aff-cowts} \Lv:= \{ \lv \in \f{h}^* \mid \la a_i, \lv \ra \in \zee \text{ for } i=1, \ldots, \ell+1 \text{ and } \la \Lambda_{\ell+1}, \lv \ra \in \zee  \} .\ee In other words $\Lambda^{\vee} = \Hom_{\zee}(\Lambda, \zee).$

\spoint For each $i=1,\ldots, \ell+1$ we denote by $w_{a_i}: \f{h} \rr \f{h}$ the reflection through the hyperplane $H_i:= \{ h \in \f{h} | (h | a_i^{\vee} ) = 0\}$ and  denote by  $W \subset \Aut(\f{h})$  the group generated by the elements $w_i$ for $i=1, \ldots, \ell.$ It is a Coxeter group. We denote by $\leq_W$ the usual Bruhat order on the group $W.$ The group $W$ acts on $\mf{h}^*$ in the usual way, i.e. $w f (X) = f (w^{-1} X)$ for $f \in \mf{h}^*, X \in \mf{h}, w \in W.$
%
 A root $a \in R$ is called a \emph{real root} if there exists $w \in W$ such that $w a \in \Pi.$ The set of such roots is denoted as $R_{re}.$ Otherwise, $a$ is called an \emph{imaginary root}, and the set of all such imaginary roots is denoted $R_{im}.$ We have decompositions, \be{re:im} R = R_{re} \sqcup R_{im} \ee and we may define $R_{re, \pm}$ and $R_{im, \pm}$ accordingly.  The set of real roots admits the following description: \be{rr} R_{re} = \{ \alpha + m \delta \mid  \alpha \in R_o, m \in \zee \}.\ee The set of positive real roots is then \be{r+} R_{re, +}= \{ \alpha + m \delta \mid \alpha \in R_{o, +},\, m \geq 0 \} \cup \{ \alpha + m \delta \mid \alpha \in R_{o, -},\, m > 0 \}. \ee  The set of imaginary roots is equal to \be{rim} R_{im} = \{ m \delta \mid m \in \zee \setminus 0 \}.\ee  Let $a \in R_{re}$ and choose some $w \in W$ such that $w a = a_i \in \Pi.$ We define the corresponding coroot \be{cor} a^{\vee}:= w^{-1} a_i^{\vee}, \ee and note that this construction does not depend on the choice of $w$ (see \cite[\S 5.1]{kac}). Recalling that $\theta$ was the highest root of the underlying finite-dimensional root system, we now have \be{a:c} a_{\ell+1}^{\vee} = - \theta^{\vee} + \cc. \ee The element $\cc$ also spans the one-dimensional center of $\mf{g}$ and we have a decomposition \be{h:dec} \mf{h} = k \cc \oplus \mf{h}_o \oplus k \dd .\ee Parallel to the decomposition (\ref{re:im}) we have a decomposition of the coroots \be{coroots:reim} R^{\vee}= R^{\vee}_{re} \cup R^{\vee}_{im}\ee into real and imaginary coroots where \be{c:rr} \begin{array}{lcr} R^{\vee}_{re} = \{ \alpha^{\vee} + m \cc \mid  \alpha \in R_o, m \in \zee \} & \text{ and } & R^{\vee}_{im} = \{ m \cc \mid m \in \zee \setminus 0 \}. \end{array} \ee

\tpoint{The Tits Cone} \label{titscone} The group $W$ also has a simple description as \be{W:aff} W = W_o \rtimes Q_o^{\vee}. \ee Elements in this group are sometimes denoted as $\wt{w} t_H$ where $\wt{w} \in W_o$ and $H \in Q_o^{\vee}.$ The coweight lattice $\Lambda^{\vee}$ can be written (see (\ref{h:dec}))  as $\Lambda^{\vee}= \zee \cc \oplus \Lambda_o^{\vee} \oplus \zee \dd,$ where the elements in $\Lambda_o^{\vee}$ are regarded as elements of $\Lambda^{\vee}$ by defining their pairing with respect to $\delta$ and $\Lambda_{\ell+1}$ to be zero. One can check that for any  $H \in Q_o^{\vee}, m, r \in \zee$ and $\lambda^{\vee}_o \in \Lambda_o^{\vee},$ \be{ff} t_H (m \cc + \lv_o + r \dd) = ( m+ (\lv_o, H) - r\frac{(H, H)}{2} ) \cc + \lv_o - r H + r \dd. \ee Moreover, one can also see that if $w \in W_o, m, r \in \zee$ and $\lambda^{\vee}_o \in \Lambda_o^{\vee},$  then \be{fin-act} w( m \cc + \lv_o + r \dd) = m \cc + w (\lv_o) + r \dd. \ee Setting \be{Lv:r} \Lambda_r^{\vee}= \{ \lambda \in \Lambda^{\vee} | \langle \delta, \lambda^{\vee} \ra = r \}, \ee we see from the above two formulas that $\Lambda^{\vee}_r$ is  $W$ invariant and \be{Lv: sum} \Lambda^{\vee}= \oplus_{r \in \zee} \Lambda_r^{\vee}.\ee  The elements in $\Lambda_r^{\vee}$ are referred to as elements of level $r.$ Let $\Lambda^\vee_+$ be the set of dominant coweights, \be{lam:dom} \Lambda^{\vee}_+ := \{ \lv \in \Lambda^{\vee} \mid \la a_i , \lv \ra \geq 0 \text{ for } i=1, \ldots, \ell+1 \}. \ee
   The \emph{Tits cone} $X \subset \Lambda^{\vee}$ is defined as \be{tc} X:= \bigcup_{w \in W} w \Lambda^{\vee}_+. \ee One may then show that \cite[Proposition 1.3(b)]{kac-pet} \be{tc:2} X = \{ \lv \in \Lambda^{\vee} | \la a , \lv \ra < 0 \text{ for only finitely many } a \in R_{re, +} \}. \ee In terms of the description of $\Lambda^{\vee}$ given above, one has from \cite[Proposition 1.9(a)]{kac-pet},  \be{ti:exp} X = \{ \lv\in \Lambda^{\vee} | \la \delta , \lv \ra >0 \} \sqcup \zee \cc. \ee In other words, the Tits cone just consists of elements of positive level, and the multiples of the imaginary coroot $\cc$ (which are of level $0.$)

\tpoint{Ring of Affine Invariants} \label{sec-aff:inv} We next define a completion $\C_{\leq}[\Lv]$ of the group algebra of $\Lv$ which is used to describe the image of the Satake isomorphism.  This ring has also been introduced earlier by Looijenga \cite{loo} who calls it the  \emph{dual-weight algebra}.

To any  $\lv \in \Lambda^{\vee}$ we associate a  formal symbol $e^{\lv}$ and impose relations: $e^{\lv} e^{\mv} = e^{\lv + \mv}.$  We define  $\C_{\leq}[\Lv]$ as the set of (possibly infinite) linear combinations \be{f} f= \sum_{\lv \in X}  c_{\lv} e^{\lv} \ee such that there exists finitely many elements $\lv_1, \ldots, \lv_r \in \Lambda^{\vee}_+$ so that \be{supp} \Supp(f) := \{ \lambda^{\vee} | c_{\lv} \neq 0 \} \subset \cup_{i=1}^r \mf{c}(\lambda^{\vee}_i), \ee  where \be{c:lam} \mf{c} (\lv) = \{ \mu^{\vee} \in X |  \mv \leq \lv \}. \ee
One may easily verify that $\C_{\leq}[\Lv]$ is a unital, associative, commutative ring and that there is an action of $W$ on $\C_{\leq}[\Lv]$ such that  $e^{\lv}=e^{w \lv}.$

The ring $\C_{\leq}[\Lv]$ carries a natural grading, with graded pieces \be{Lr} \C_{\leq}[\Lv]_r:=  \{ f \in \C_{\leq}[\Lv] \mid \, \Supp(f) \subset \Lv_r \cap X \} .  \ee  as follows from  (\ref{ti:exp}) and (\ref{Lv:r}), we  have \be{L:grad} \C_{\leq}[\Lv] = \oplus_{r \geq 0} \C_{\leq}[\Lv]_r. \ee  It is easy to see that in fact $\C_{\leq}[\Lv]$ is a graded module over the commutative ring $\C_{\leq}[\Lv]_0.$ Moreover, the piece $\C_{\leq}[\Lv]_0$ is easy to describe explicitly.

\begin{nlem} \label{L:0}   Let  $\cc$ be the minimal imaginary positive coroot. Then $\C_{\leq}[\Lv]_0 = \C((e^{-\cc} ))$  \end{nlem}
\begin{proof} First note that if $f \in \C_{\leq}[\Lv]_0$ then its support  is contained in the set  of elements in the Tits cone of level $0$ and such elements are of the form  $\zee \cc..$ The Lemma will follow from the following claim, whose proof is given below.

\begin{nclaim} Let $m \in \zee$ and set $\mv = m \cc.$ If $\lv \in \Lambda_+^{\vee}$ is such that $\mv \leq \lv,$ then $\lv = n \cc,$ with $n \geq m.$ \end{nclaim}

Indeed, if $f \in \C_{\leq}[\Lv]_0$ it follows from the claim that then there exists finitely many integers $n_1, \ldots, n_r$ such that every $\mv \in \Supp(f)$ is of the form $\mu = m_i \cc$ with $m_i \leq n_i.$ The Lemma follows from this.  \end{proof}

\begin{proof}[Proof of Claim] First note that if $\mv$ has level $0$ and $\lv \geq \mv$ then necessarily $\lv$ also has level $0.$ But a dominant coweight of level $0$ must be of the form $n \cc:$ if not, then $\lv = n \cc + \lv_o$ where $\lv_o \in \Lambda^{\vee}_o.$ By the dominance property of $\lv$ we must have \be{lv:dom} \la a_i , \lv \ra = \la a_i, \lv_0 \ra \geq 0 \, \text{ for } i=1, \ldots, \ell+1. \ee Suppose there exists $i \in \{ 1, \ldots, \ell \}$ such that $\la a_i, \lv_0 \ra > 0.$ Then we must have \be{lv:dom:2} \la a_{\ell+1} , \lv_0 \ra = \la - \theta + \delta , \lv_0 \ra = - \la \theta, \lv_0 \ra < 0, \ee since $\theta$ is a positive sum of the elements $a_i$ for $i=1, \ldots, \ell.$ So $\la a_i , \lv_0 \ra =0$ for $i=1, \ldots, \ell$ and in fact $\lv = n \cc$ for some $n \in \zee.$ If in fact $\mv \leq \lv$ then $\lv - \mv = (n-m) \cc$ must be a positive sum of coroots, and so in fact $n \geq m.$
\end{proof}


\tpoint{Dual Root Systems} The affine Kac-Moody algebra may also be defined as a Lie algebra associated to  a  generalized Cartan matrix. For the construction, we refer to \cite{kac}. Suppose that a generalized Cartan matrix $M$ corresponds to the untwisted, affine Lie algebra $\mf{g}.$ Then the transpose ${}^{t}M$ is also a generalized Cartan matrix corresponding to the dual Lie algebra $\mf{g}^{\vee}.$ In general $\mf{g}^{\vee}$ is again an affine Lie algebra, but could be of twisted type. On the other hand, if $M$ is of simply-laced type (i.e., the underlying finite-type root system attached to $M$ is simply laced), then ${}^tM$ will again correspond to an untwisted affine Lie algebra. In fact, in this case, $\mf{g}^{\vee}$ is the untwisted affine Lie algebra attached to $\mf{g}_o^{\vee},$ the dual of the underlying finite-dimensional root system.

To avoid complications of twisted affine types, \emph{we shall throughout restrict to the case that $\mf{g}_o$ is of simply-laced type.}

\tpoint{Modules for Affine Lie Algebras} \label{loo}  A $\f{g}$ module is called $\f{h}$-diagonalizable if $V = \bigoplus_{\lambda \in \f{h}^* } V({\lambda}) $ where $V({\lambda})$ are the weight spaces $V({\lambda}) := \{ v \in V | h . v = \langle \lambda, h \rangle v \}.$  We define the set  $P(V)$ of \emph{weights} of $V$ by \be{wt:V} P(V):= \{ \lambda \in \f{h}^* | V({\lambda}) \neq 0 \}. \ee Given $\lambda, \mu \in \f{h}^*$ we define the dominance partial order on $P(V)$ by \be{dom:wts} \lambda \geq \mu \iff \lambda - \mu \in Q_+. \ee

We define $\Rep(\f{g})$ as the category   of $\f{g}$-modules $V$such that \begin{enumerate} \item $V$ is $\f{h}$-diagonalizable  \item $V({\lambda})$ are finite dimensional for each non-zero $\lambda \in P(V)$ \item there exist finitely many $\lambda_1, \cdots, \lambda_r \in \f{h}^*$ such that $P(V) \subset \cup_{i=1}^{r} \mf{c}(\lambda_i),$ where $\mf{c}(\lambda_i)$ is defined analogously to (\ref{c:lam}). \end{enumerate} One checks that $\Rep(\f{g})$ is an abelian category stable under tensor product. Therefore we may form its  complexified Grothendieck ring  $K_0(\f{g}),$ and then easily check that  the map being the one which sends a representation to its character defines an isomorphism,\be{ko:loo} K_0(\mf{g}^{\vee}) \stackrel{\cong}{\rr} \C_{\leq}[\Lv]^W. \ee

\subsection{Loop Groups} \label{section-loopgroups}

\renewcommand{\b}[1]{\mathbf{{#1}}}

\tpoint{The Tits Group Functor} We review in this part the construction of affine Kac-Moody group $\b{G}$ due to Tits. \cite{tits}.  A set $\Psi \subset R_{re} $ is \emph{pre-nilpotent} if there exists $w, w' \in W$ such that $w \Psi \subset R_{re, +}$ and $w' R_{re, -} \subset R_{re, -}.$ If such a set $\Psi$ also satisfies the condition: \be{nilp} \text{ if } a, b \in \Psi, a+b \in R_{re}, \text{ then } a + b \in \Psi, \ee we say that $\Psi$ is a nilpotent set.  For any  $a \in R_{re}$ we denote by $\b{U}_a$ a corresponding one-dimensional additive group scheme, and fix an isomorphism $x_a: \mathbb{G}_a \rr \b{U}_a.$ For any nilpotent set $\Psi$ of roots,  Tits has constructed (see \cite[Proposition 1]{tits} a group schemes $\b{U}_{\Psi}$ equipped with inclusions $\b{U}_a \hookrightarrow \b{U}_{\Psi}$ such that for any choice of an order on $\Psi$ the map, $\prod_{a \in \Psi} \b{U}_a \rr \b{U}_{\Psi}$ is an isomorphism of schemes.

Given any pre-nilpotent pair of roots $\{ a, b \}$ we set $\theta(a, b) = (\mathbb{N}a + \mathbb{N}b) \cap R_{re}.$
Tits has shown that for any  total order on $\theta(a, b) \setminus \{ a, b \},$ there exists a unique set  $k(a, b; c)$ of integers such that for any ring $S$ we have \be{stein:rel} ( x_a(u), x_b(u') ) = \prod_{c=ma+nb} x_c(k(a, b; c) u^m u'^n) \ee for all  $u, u' \in S$ and where $c=ma+nb$ varies over $\theta(a, b) \setminus \{ a, b\}$ and  $ ( x_a(u), x_b(u') )$ is the commutator.  We then define the Steinberg functor $\mathbf{St}$ to be the quotient of  the free product of the groups $\b{U}_a, a \in R_{re}$ by the  normal subgroup generated by the above relations (\ref{stein:rel}).

The Weyl group $W$ of $\f{g}$ acts on functors $\b{St}$ and $\b{A}.$ For $w \in W$ we denote by $w^*$ the corresponding action on either of these functors. For each  $i=1, \ldots, \ell+1$ we choose isomorphisms $x_{a_i}: \mathbb{G}_a \rr \b{U}_{a_i}$ and $x_{-a_i}: \mathbb{G}_a \rr \b{U}_{-a_i}$ and for  each invertible element $r \in S^*$ and $i=1, \ldots, \ell+1$  denote by $\wt{w_i}(r)$ the image of the product \be{w_i(r)}  x_{a_i}(r) x_{-a_i}(-r^{-1}) x_{a_i}(r)  \ee in $\b{St}(S).$ We set $\wt{w_i}:= \wt{w_i}(1).$


We let $\b{A}$ be the functor which sends a ring $S$ to \be{T:def} \b{A}(S)= \Hom_{\zee}(\Lambda, S). \ee For $u \in S^*$ and $\lv \in \Lambda^{\vee} = \Hom_{\zee}(\Lambda, \zee)  $ we write $s^{\lv}$ for the element of $\b{A}(S)$ map which sends each $\mu \in \Lambda$  to $s^{\langle \mu, \lv \rangle} \in S.$

The affine Kac-Moody group functor (Tits functor) is the functor $\b{G}$ which associates to a given ring $S$ the quotient of the free product $\b{St}(S) \star \b{A}(S)$ by the smallest normal subgroup containing the canonical images of the following relations, where $i=1, \ldots, \ell+1,$ $r \in S$, and $t \in \b{A}(S)$  \begin{eqnarray}
t \, x_{a_i}(r) \, t^{-1} &=& x_{a_i} ( t^{a_i^\vee} r ) \\
\wt{w_i} \, t \, \wt{w_i}^{-1} &=&  w_i^*(t) \\
\wt{w_i}(r) \wt{w_i}^{-1} &=& r^{a_i^{\vee}} \text{ for } r \in S^* \\
\wt{w_i} \, u \,  \wt{w_i}^{-1} &=& w_i^*(u) \text { for } u \in \b{U}_a(S), a \in R_{re}. \end{eqnarray}

Note the following important, but simple identity holds in the group $\b{G}(S),$ \be{sl2} x_{-a}(s^{-1}) = x_a(s) (-s)^{a^{\vee}} w_a(1) x_a(s), \ee where $s \in S^*$ and $a \in R_{re}.$

\tpoint{Bruhat Decompositions} \label{sec-bruhat} Now we  describe the structure of $G:= \b{G}(k)$ for any field $k$.  \footnote{In general, if $\b{X}$ is some functor, we shall denote by the roman letter $X$ the set of points $\b{X}(k)$ over some field $k.$ If there is some danger of confusion about which field we are working over, we shall write $X_k.$}
For each $a \in R_{re}$ we define $U_a = \mathbf{U}_a(k),$ and $A= \b{A}(k).$ Let $U$ denote the subgroup generated by $U_a$ for $a \in R_{re, +}$ and $U^-$ the subgroup generated by $U_a$ for $a \in R_{re, -}.$ Define now $B_a$ to be the subgroup of $G$ generated by $U_a$ and $T.$ Also, set $B$ and $B^-$ to be the subgroups generated by all the $B_a$ for $a \in R_{re, +}$ and $R_{re, -}$ respectively.  We have semi-direct products $B = A \rtimes U$ and $B^- = A \rtimes U^-.$ We let $N$ be the group generated by $T$ and the $\wt{w_i},$ where the elements $\wt{w_i}$ were defined above. There is a natural map \be{zeta:N} \zeta: N \rr W \ee which sends $\wt{w_i} \mapsto w_i$ and  which has kernel $A.$ This map is surjective, and induces an isomorphism $\zeta: N/A \rr W.$  For each $w \in W,$ we shall write $\dot{w}$ for any lift of $w$ by $\zeta.$ If $w \in W$ has a reduced decomposition $w = w_{a_{i_1}} \cdots w_{a_{i_r}},$ with the $a_k \in \Pi,$ we shall also sometimes write \be{tw} \wt{w}:= \wt{w}_{i_1} \cdots \wt{w}_{i_r} \ee for a specific lift of $w,$ where the $\wt{w}_i$ were defined after (\ref{w_i(r)}).
%

\begin{nprop} [\cite{tits}] \label{bruhat} One has the following Bruhat-type decompositions \be{bruhat} G &=& \sqcup_{w \in W}  B \, \dot{w} \, B \ \ = \ \  \sqcup_{w \in W} \, B^- \dot{w} \,  B^- \\
&=& \sqcup_{w \in W} B^-\,  \dot{w} \, B \ \ = \ \ \sqcup_{w \in W} B \,  \dot{w} \, B^-,\ee where $\dot{w}$ is any lift of $w \in W$ to $N$ under the map (\ref{zeta:N})
\footnote{Usually, for $w \in W$ we shall often just write $B w B$ for the coset $B \wt{w} B,$ as the choice of representative does not change the subset of the group.}.  \end{nprop}

Note that it is important here that we are working with the so-called minimal Kac-Moody group in order to have Bruhat decompositions with respect to $B$ and $B^-.$ We shall also need the following claim whose proof we suppress (see the remarks after \cite[Corollary 1.2]{deod}.)

\begin{nclaim} \label{bru-ord} Suppose that $w_1, w_2 \in W$ such that $B \w_1 B \cap B \w_2 B^- \neq \emptyset.$ Then we have $w_2 \leq w_1$ in the Bruhat order on $W.$ \end{nclaim}

\tpoint{Integrable Modules} \label{sec-repth} Fix a Chevalley form $\f{g}_{\zee} \subset \f{g}$ as in \cite{ga:la, tits}. For any $\omega \in \Lam_+$ we let $V^{\omega}$ denote the corresponding highest weight module for $\mf{g}.$ It is shown in \cite{ga:la} that $V^\omega$ can be equipped with a $\zee$-form $V^\omega_\zee$ which is stable under the elements of the Chevalley basis of $\mf{g}_{\zee}$ and their divided powers. Moreover, the action integrates to an action of $\b{G}(\zee).$ We shall write $V^{\omega}_R$ for $V_{\zee}^{\omega} \otimes R,$ which is naturally a $\b{G}(R)$-module.
%
%
%
%
%
%

Let now $\mc{K}$ be a local field, so that we then have a representation the group $G:= \b{G}(\K)$ on $V:= V^{\omega}_{\mc{K}}.$ Let $V_{\O}:= \O \otimes_{\zee} V_{\zee} \subset V,$ and denote by $v_{\omega}$ be a highest weight vector, i.e., $v_{\omega} \in V(\omega).$ For $v \in V$ we can set \be{ord} \ord_{\pi} (v) = \min \{ n \in \zee |  \pi^n v \in V_{\O} \} \ee and define \be{norm:v} || v || = q^{\ord_{\pi} (v) }, \ee where $q$ is the size of the residue field of $\res.$ If $v, w$ are in different weight spaces, then \be{triangle} || v + w || \geq ||v ||. \ee  Also, we can choose $v_{\omega}$ to be \emph{primitive}, i.e., $|| v_{\omega} || = 1.$  The elements of $K$ preserve the norm $|| \cdot ||.$ Note also that elements from $U$ stabilize the highest weight vector $v_{\omega}.$ Moreover, an element $s^{\lv} \in A$  with $s \in \K^*$ acts on an vector in the weight space $V(\mu)$ as the scalar $s^{\la \mu, \lv \ra}.$

\section{Basic Structure of $p$-adic Loop Groups} \label{section:groups}

Recall the conventions for local fields $\mc{K}$ from the start of this paper. The goal of this section is to study the basic properties of the group $G:= \b{G}(\K).$

\subsection{Subgroups of $\b{G}(\O)$}

In this part, we define various subgroups of the group of integral points $K:= \b{G}(\O) \subset G$ and establish some elementary properties of them.

\tpoint{The integral torus} \label{sec-A} Recall that we have set $A := \b{A}(\mc{K}) \cong \Hom_{\zee}(\Lambda, \mc{K}) \subset G.$ Let \be{A:o} A_{\O}:= \b{A}(\O), \ee and note that we have an identification $A / A_{\O} = \Lambda^{\vee}.$ Recall further that we have defined $N$ as the group generated by $A$ and the elements $\wt{w}_i$ in \S \ref{sec-bruhat}, and we have an isomorphism $\zeta: N/ A \rr W$  (\ref{zeta:N}). We define the "affine" Weyl group as \be{aw} \aw:= W \rtimes \Lambda^{\vee}, \ee and note that $\zeta$ can be naturally lifted to a homomorphism also denoted
\be{z:aw}
N \rr \aw.
\ee
From now on we shall denote {\it this} homomorphism by $\zeta$. The kernel of this map is $A_{\O},$ and we have
$N_{\O}:= N \cap K = \zeta^{-1}(W).$ If we write $\wt{w}$ for the representative of $w \in W$ as in (\ref{tw}), we have that
\be{N_O}
N_{\O} = \bigcup_{w \in W} A_{\O} \wt{w}.
\ee
Recall that for $\lv \in \Lv$, we have an element $\pi^{\lv} \in A.$ For each $x=(w, \lv) \in \aw,$ or $w \in W$ we shall abuse notation and denote by $w \pi^{\lv}$ the element $\wt{w} \pi^{\lv} \in N.$ Sometimes, we shall just write $w \pi^{\lv}$ for any element in the set $\zeta^{-1}(w \pi^{\lv})$, hoping no confusion will arise.

\tpoint{Iwahori Subgroups} \label{K:decomp} For each $a \in R_{re},$ the elements of the one-dimensional group $U_a:= \b{U}_a(\K)  \subset G$ will be written as $x_a(u)$ for $u \in \K.$ For $m \in \zee$ we set \be{U:m} U_{(a, m)}:= \{ x_a(u) | v (u) \geq m \} .\ee  As a shorthand, if $a \in R_{re}$ we write \be{Ua:defs} \begin{array}{lcr} U_{a, \O} := U_{(a, 0)} = \b{U}_a(\O) &
U_{a, \pi} := U_{(a, 1)} &
U_{a}[m] := U_{(a, m)} \setminus U_{(a, m-1)} \end{array} \ee Let us also set $U_{\pi}$ to be the group generated by $U_{a, \pi}$ with $a \in R_{re, +}$ and $U^-_{\O}$ the group generated by $U_{-a, \O}$ for $a \in R_{re, +}.$ Similarly, we may define the groups $U^+_\O$ and $U^-_\pi.$

The group $K = \b{G}(\O)$ is generated by the subgroups $U_{a, \O}$ as one can see by referring back to the definition of the functor $\b{G}.$ Let \be{kappa} \kappa: K \rr G_\res \ee  denote the map induced from the natural reduction $\O \rr \res$ We shall define the \emph{Iwahori subgroup} $I \subset K$ to as \be{iwa:+} I:= \{  x \in K | \kappa(x) \in B_\res \}. \ee Similarly we can define the \emph{opposite Iwahori subgroup} \be{iwa:-} I^-:= \{ x \in K \mid \kappa(x) \in B^-_{\res} \}. \ee  The proof of the following is entirely analogous to the classical situation (see \cite[\S 2]{im}), so we suppress the details.

\begin{nprop} \label{im:i} Keeping the notation above, we have the following decompositions \begin{enumerate} \item The groups $I$ and $I^-$ admit the following decompositions, \be{I+:im} I = U_{\O}U^-_{\pi}A_{\O}= U^-_{\pi} U_{\O} A_{\O} \ee and \be{I^-:im} I^- = U_{\pi}U^-_{\O}A_{\O}= U^-_{\O} U_{\pi} A_{\O} \ee

\item Choose representatives $\wt{w} \in K$ for $w \in W$ as in (\ref{tw}). Then there exist disjoint unions,
\be{bruhat:K} K &=& \sqcup_{w \in W}  I \, \wt{w} \, I \ \ = \ \  \sqcup_{w \in W} \, I^- \wt{w} \,  I^- \\
&=& \sqcup_{w \in W} I^-\,  \wt{w} \, I \ \ = \ \ \sqcup_{w \in W} I \,  \wt{w} \, I^-. \ee  \end{enumerate}

\end{nprop}

\subsection{Iwasawa Decompositions} \label{sec-iwa}

Recall that we ave defined subgroups $K, A, U \subset G,$ together with a subgroup $A_{\O} \subset A.$ Setting \be{A'} A':= \{ \pi^{\lv} : \lv \in \Lambda^{\vee} \}, \ee  we have a direct product decomposition, $A = A' \times A_{\O}$ which gives an identification $A/ A_{\O} = \Lv$ The Iwasawa decomposition in this context states,

\begin{nthm} \label{iw:unique} Every $g \in G$ has a decomposition $g = k \pi^{\lv} u$ where $k \in K$, $\pi^{\lv} \in A'$, and $u \in U,$ and $\lv$ is uniquely determined by $g.$ Furthermore, we also have an opposite Iwasawa decomposition: every $g \in G$ may be written as  $g = k \pi^{\lv} u^-$ where $k \in K$, $\pi^{\lv} \in A'$, and $u \in U^-,$ and $\lv$ is uniquely determined by $g.$ \end{nthm}

The existence of the Iwasawa decomposition can be deduced via a standard manner (see \cite[\S 16]{ga:ihes} and the references therein) from the Bruhat decomposition and a rank one computation. Note that in order to have both Iwasawa decompositions with respect to the groups $U$ and $U^-$ it is important that we are working in the minimal Kac-Moody group.

To address the uniqueness in Theorem \ref{iw:unique}, first note that it suffices to show uniqueness of $A' / A'_{cen}$ where $A'_{cen}$ is the central subgroup generated by $\pi^{n \cc}$ with $n \in \zee.$ Fix notations as in \S \ref{sec-repth}. If $g= k \pi^{\lambda^{\vee}} u$ then \be{g:rep} || g v_{\omega} || = || \pi^{\lambda^{\vee}} v_{\omega}|| = q^{\la \omega, \lv \ra}. \ee One then sees that $\lv$ is uniquely determined (modulo the center) from $g$ by varying $\omega$ over a set of fundamental weights of $\mf{g}.$

For the uniqueness statement with respect to $U^-,$ we argue slightly differently. Suppose that we have $K \pi^{\lv} U^- \cap  K \pi^{\mv} U^- \neq \emptyset$, i.e., $K \pi^{\lv} U^- = K \pi^{\mv} U^-.$ Then $K \pi^{\lv} U^- \cap K \pi^{\lv} U \neq \emptyset,$ and we also have $K \pi^{\mv} U^- \cap K \pi^{\lv} U \neq \emptyset,$ and hence from \cite[Theorem 1.9 (2)]{bgkp} , $\mv \geq \lv$ and $\lv \geq \mv,$ and so $\lv = \mv.$

\subsection{Cartan Decomposition and the semigroup $G_+$} \label{sec-cartan}

Unlike the finite-dimensional case, the Cartan decomposition does not hold for the group $G$, i.e., not every element $g \in G$ can be written as $g= k_1 \pi^{\lv} k_2$ with $k_1, k_2 \in K$ and $\lv \in \Lambda^{\vee}.$ On the other hand, this property does hold for the semi-group $G_+$ introduced in $\S$ \ref{intro-Gplus} (see \cite{bk, gar:car}). Recall that $G_+$ was defined to be the semigroup generated by $K$, the central $\mc{K}^*,$ and the elements $g \in G$ such that $| \eta(g) | > 0$ where the map $|\eta|$ was constructed in (\ref{eta:Kpts}).

\begin{nthm} [\cite{bk, gar:car}] \label{cartan-thmdefn} The semi-group $G_+$ can be written as a disjoint union, \be{semi:cartan} G_+ = \sqcup_{\lv \in \Lv_+} K \pi^{\lv} K \ee   \end{nthm}

In lieu of the above result, we shall often refer to $G_+$ as the \emph{Cartan semigroup}. The above result implies in particular that the right hand side of (\ref{semi:cartan}) is a semi-group, a non-trivial fact. In appendix \ref{car-app}, we shall give another proof of this theorem based in part on the argument from (\cite{gar:car}).

Let $G_+^{\textsf{Iw}},$ the \emph{Iwasawa semi-group}, be the subset of elements of the form $K \pi^{\mv} U$ with $\mv \in X.$ In Appendix \ref{car-app}, we shall also relate the Iwasawa and Cartan semigroups as follows,

\begin{ncor} \label{semigp-iwa}  We have an equality of semigroups, $G_+^{\textsf{Iw}} = G_+.$ \end{ncor}

The proof follows by combining Propositions \ref{iwa-rel}, \ref{cartan:reform}, and \ref{intro-semigp}.

\subsection{Iwahori-Matsumoto Decomposition}
\label{sec-im}
\subsubsection{"Affine" Weyl (semi)-group} \label{sec-affw}

We would now like to study another descriptions of $G_+,$ which is the analogue of the Iwahori-Matsumoto decomposition for a classical $p$-adic group into cosets indexed by the the affine Weyl group. The set indexing Iwahori double cosets of $G_+$ will be called the "affine" Weyl semi-group $\mc{W}_X$ \footnote{We use the term "affine" Weyl group to refer to what some other others call the double affine Weyl group. Our present notation is meant to emphasize the analogy with the usual theory of groups over a local field.} which is defined as follows: recall that $X \subset \Lv$ was defined to the Tits cone, which carries a natural action of $W$ the Weyl group of $G.$ We have already defined the "affine" Wey group $\mc{W} = W \rtimes \Lv$ in (\ref{aw}), and we now set \be{aw-X} \aw_X := W \rtimes X. \ee For an element $x = (w, \lv) \in \aw,$ with $w \in W$ and $\lv \in \Lv$ we shall abuse notation and just write $w \pi^{\lv}$ for the corresponding lift $\wt{w} \pi^{\lv}$ in $N,$ where $\wt{w} \in K$ was defined in (\ref{tw}). Sometimes we shall also just write $x = w \pi^{\lv} \in \aw$ to mean the pair $x=(w, \lv) \in \aw$ hoping this will not cause any confusion in the sequel.
%

\subsubsection{Iwahori-Matsumoto Semigroup}

Let us define, following Iwahori and Matsumoto \cite{im}, \be{iw:mat:G} G_+^{\textsf{IM}}:= \bigcup_{x \in \aw_X} I x I, \ee where for each $x \in \aw_X$ we denote by the same letter a corresponding lift to an element of $G_+.$

\begin{nprop} \label{im=car} We have an equality of semigroups $G_+^{\textsf{IM}} = G_+.$ \end{nprop}

\begin{proof} If $x = w \pi^{\lv} \in \mc{W}_X,$ with $w \in W$ and $\lv \in X$ then clearly $I w  \pi^{\lv} I \subset K \pi^{\lv} K \subset G_+.$ Conversely, since we also have a decomposition $K = \bigcup_{w \in W} I w I$ from (\ref{bruhat:K}), we obtain \begin{eqnarray} \label{im:sg2} K \pi^{\lv} K \subset \bigcup_{w, w' \in W} I w I \pi^{\lv} I w' I \subset  I W \pi^{\lv} W I \subset I \mc{W}_X I, \end{eqnarray}  where the second inclusion is a consequence of Lemma \ref{W:I} below.  So the two sets in question are equal, and by Theorem \ref{cartan-thmdefn} the semigroup property of $G_+^{\mathsf{IM}}$ follows. \end{proof}

In the proof of the previous Proposition, we used the following simple result,

\begin{nlem} \label{W:I} Let $\lv \in \Lv.$ Then, \begin{enumerate}
\item For each $w \in W$ we have $w I \pi^{\lv} I \subset \cup_{w' \in W} I w' \pi^{\lv} I.$ Symbolically, we write $W I \pi^{\lv} I \subset I W \pi^{\lv} I$
\item Similarly we have that $I \pi^{\lv} I  W \subset I \pi^{\lv} W I.$  \end{enumerate}  \end{nlem}

\begin{proof} Let us establish the first statement, the proof of the second being similar.  Given $w \in W,$ it is easy to see (see e.g. \cite[Proposition 2.6]{im} for an essentially similar argument) that \be{lem:sg:3} w I \subset I w U_{w}[0] \; \text{ where }\; U_{w}[0] = \prod_{a \in R_{+}: w \cdot a < 0} U_{a}[0], \ee where the notation $U_{a}[0]$ was introduced in (\ref{Ua:defs}).  In particular, if $w=w_a$ for $a \in \Pi$ a simple root, we have, \be{lem:sg:4} w_a I \pi^{\lv} I \subset I w_a U_{a}[0] \pi^{\lv} I ,\ee and thus we have two possibilities,
\begin{enumerate}
\item[a.] If $\langle \lv, a \rangle \leq 0$ then $\pi^{- \lv} U_{a}[0] \pi^{\lv} \subset I$ and so \be{lem:sg:4} w_a I \pi^{\lv} I \subset I w_a \pi^{\lv} I. \ee

\item[b.] If $\langle \lv, a \rangle > 0$ then a rank one computation as in (\ref{sl2}) shows that we can write \be{lem:sg:5} w_aU_{a}[0] \pi^{\lv} I \subset w_a U_{-a}[0] w_a U_{-a}[0] \pi^{\lv} I. \ee By our assumption, we have $\pi^{-\lv} U_{-a}[0] \pi^{\lv} \subset I$ and so \be{lem:sg:6} w_a U_{-a}[0] w_a U_{-a}[0] \pi^{\lv} I &\subset& w_a U_{-a}[0] w_a \pi^{\lv} I \\
&\subset&  U_{a}[0] \pi^{\lv} I \subset I \pi^{\lv} I .\ee
\end{enumerate}

\noindent So in both cases, we see that $w_a I \pi^{\lv} I \subset I W \pi^{\lv} I,$ namely \be{lem:sg:7} w_a I \pi^{\lv} I \subset \begin{cases} I \pi^{\lv} I & \text{ if } \la \lv, a \ra \geq 0 \\ I w_a \pi^{\lv} I & \text{ if } \la \lv, a \ra < 0. \end{cases} \ee  An easy induction on the length of $w$ concludes the proof of (1).
\end{proof}

\subsubsection{Disjointness of $I$-double cosets} We would like to show that the decomposition (\ref{iw:mat:G}) above is disjoint. In other words, we show,
\begin{nlem} \label{!:ixi} Let $x, y \in \aw_X.$ Then $I x I \cap I y I \neq \emptyset$ implies that $x=y.$ \end{nlem}
\begin{proof} Suppose that $x= w \pi^{\lv}$ and $y = v \pi^{\mu}$ with $w, v \in W$ and $\lv, \mv \in X.$ If the above intersection is non-empty, we have that in fact  $I w \pi^{\lv} I = I v \pi^{\mv} I.$ Hence we also have $K \pi^{\lv} I U = K  \pi^{\mu} I U. $ Now using that $I =A_{\O} U^-_{\pi} U_{\O} $ we obtain $K \pi^{\lv} A_{\O} U^-_{ \pi} U = K \pi^{\mu} A_{\O} U^-_{\pi} U.$ Thus $\pi^{\lv} \in K \pi^{\mv} U^-_{\pi} U$ and so $\pi^{\lv} U \in K \pi^{\mv} U^-.$ Using \cite[Theorem 1.9(1)]{bgkp}, we conclude that $\mv \leq \lv.$ A similar argument shows that $\lv \leq \mv.$ So $\mv = \lv.$

So we may assume now that $x= w \pi^{\lv}$ and $y = v \pi^{\lv}.$ If $I w \pi^{\lv} I  \cap I v \pi^{\lv} I \neq \emptyset,$ then again using the decomposition $I = U_{\O} U^-_{\pi} A_{\O}$ we obtain that \be{!:ixi:3} I w \pi^{\lv} U_{\O} \cap I v \pi^{\lv} U^-_{\pi} \neq \emptyset. \ee As $\pi^{\lv}$ normalizes both $U$ and $U^-$ we may conclude that  $I w U \cap I \sigma U^- \neq \emptyset.$ Suppose that we take some element in this intersection \be{!:ixi:4} i w U = i' v u^- \text{ where } i, i' \in I, \, u^{\pm} \in U^{\pm}.\ee  Then rearranging this last equation, we find that $u^- \in K U.$  By \cite[Lemma 5.2.1]{bgkp}, this implies that $u^- \in U^-_{\O}= U^- \cap K$ and so it follows form (\ref{!:ixi:4}) that we may assume that $u^{+} \in U_{\O}$ as well. Thus, we have produced an element which lies in the intersection of $I w I \cap I v I^-.$ From here, we may proceed by a simple induction to conclude that $w \geq v$ in the usual Bruhat order on $W$ (see Claim \ref{bru-ord}). Reversing the role of $v$ and $w$ in the above argument, we may also conclude that $v \geq w$ and so $v=w.$  \end{proof}

\subsubsection{An order on $\mc{W}$} \label{sec-affworder} An analysis of the argument above suggests the following order on $\aw.$ Note that we now choose to present elements $x \in \mc{W}$  as $x = \pi^{\lv} w$ with $\lv \in \Lv$ and $ w \in W.$

\begin{de} \label{aw:domord}  Let $x, y \in \aw,$ which we write as $x = \pi^{\lv} w$ and $y = \pi^{\mv} v $ with $w, v \in W$ and $\lv, \mv \in \Lambda^{\vee}.$ We say that $x \preceq y$ if either \begin{itemize} \item $\lv \leq \mv,$ where $\leq$ is the dominance order on $\Lambda^{\vee};$ \item $\lv=\mv$ and $w \leq v$ with respect to the Bruhat order on $W.$ \end{itemize} \end{de} It is easy to see that $\preceq$ is a partial order on $\aw.$ For us, this order will be important due to the following result, which translates into showing the the Iwahori-Hecke algebra $H(G_+, I)$  acts faithfully on its generic principal series $M(G, I).$

\begin{nprop} \label{domorder} Let $x \in \aw_X$ and $y \in \aw.$ If $I x I \cap U y I  \neq \emptyset,$ then $x \preceq y.$ \end{nprop}
\begin{proof} Write $x = \pi^{\lv} w$ and $y = \pi^{\mv} v $ with $w, v \in W$ and $\lv \in X,$ $\mv \in \Lv.$ Then if $U \pi^{\mv} v I \cap I \pi^{\lv} w I \neq \emptyset,$ we also have that \be{domorder:1} U \pi^{\mv} K \cap K \pi^{\lv} K \neq \emptyset. \ee It follows from this (see \cite[Theorem 1.9(1)]{bgkp}) that $\mv \geq \lv_+$ where $\lv_+$ is the dominant element in the $W$-orbit of $\lv.$ Hence also $\mv \geq \lv$ since $\lv_+ \geq \lv.$

Thus, we shall assume $\mv=\lv,$ and so $I \pi^{\lv} w I \cap U \pi^{\lv} v I \neq \emptyset.$ Using the decomposition $I= U_\O U^-_\pi A_\O$ we obtain that  $U^-_{\pi} \pi^{\lv} w I \cap U \pi^{\lv} v I \neq \emptyset,$ or in other words, $U^-_{\pi} \pi^{\lv} \cap U \pi^{\lv} v I w^{-1} \neq \emptyset.$ From here it follows that $U^- \cap U v I w^{-1} \neq \emptyset.$  From \cite{bgkp} (see also (\ref{b:1}) ), we may conclude that $U^- \cap U K \subset U^- \cap K,$ and so we may conclude that $U^-_{\O}  \cap U \sigma I w^{-1} \neq \emptyset.$ From this last statement, we conclude that $U^-_{\O} \cap U_{\O} v I w^{-1} \neq \emptyset.$ Thus we have produced an element in the intersection of $I^- w I \cap I v I,$ which implies that $w \leq v.$

\end{proof}

\subsubsection{A double coset decomposition}

In the sequel, we shall also need to understand the double cosets of $G$ under the left action of the group $A_{\O} U$ and under the right action of the groups $I$ or $I^-.$ The following result follows immediately from the Iwasawa decomposition and (\ref{bruhat:K}) and we suppress the proof.

\begin{nlem} \label{W:AO} The maps $\aw \rr A_\O U \setminus G / I$ and $\aw \rr A_\O U \setminus G / I^-$ which send $x \in \aw$ to $A_\O U x I$ and $A_\O U x I^-$ are bijections.  \end{nlem}


\section{Generalities on Convolution Algebras}

In this section, we describe some axiomatic patterns which our construction of Hecke algebras and their modules in the next two chapters will follow.  The notation is independent of the previous sections.  Throughout this section we shall fix the following notation: $\Gamma$ will be a group and $\Gamma_+ \subset \Gamma$ will be a sub-semigroup. All constructions will take this pair $(\Gamma, \Gamma_+)$ as data, though we shall often omit $\Gamma$ from our notation.

\subsection{Convolution of Finitely Supported Functions}

\tpoint{Basic Notations on Spaces of Double Cosets}  \label{conv:note} Let $L, R  \subset \Gamma_+$ be sub\emph{groups} of $\Gamma,$ and further assume that there exists a set $\Lam_{L, R}$ equipped with a bijection to the set of $(L, R)$ double cosets of $\Gamma_+,$ \be{X:lam} \Lam_{L, R}: X_{L, R} \rr L \setminus \Gamma_+ / R, \,\ \   \lambda \mapsto X^{\lam}  \ee When there is no danger of confusion, we shall often omit the subscripts $L$ and $R$ from our notation and just write $X: \Lam \rr L \setminus \Gamma_+ / R.$

Given a function $f: \Gamma_+ \rr \C$ which is left $L$-invariant and right $R$-invariant, there exists a subset $\Lambda_f \subset \Lambda$ defined as \be{supp:f} \Lambda_f:= \{ \lambda \in \Lam \mid f(x) \neq 0 \text{ for any } x \in X^{\lambda} \} \ee
We shall often write $f = \sum_{\mu \in \Lambda_f} c_{\mu} X^{\mu}$ where $c_{\mu} = f(x)$ for any $x \in X^{\mu}.$ We shall say that $f$ has \emph{finite support} if $\Lambda_f \subset \Lambda$ is finite. Denote by $\mc{F}(L \setminus \Gamma_+ / R) $ the set of finitely supported left $L$ and right $R$-invariant functions. It is clear that $\mc{F}(L \setminus \Gamma_+ / R) $ is a $\C$-vector space with basis indexed by the characteristic functions $X^{\lambda}$ for $\lambda \in X.$

\tpoint{Fiber Products} \label{fiberprod} Let $R \subset \Gamma_+$ be a subgroup. Given a right $R$-set (i.e., a set with a right $R$-action) $A,$ a left $R$-set $B,$ and any set $S,$ we say that a map $m: A \times B \rr S$ is \emph{$R$-linear} if $$m(ai, i^{-1} b) = m(a, b) \text{ for } a \in A, b \in B, i \in R.$$ Equivalently, if we endow the set $\Hom(B, S)$  with a left $R$-action via \be{lef:act} i \varphi(b) = \varphi(i^{-1}b) \text{ for } i \in R, \varphi \in \Hom(B, S), \ee then an $R$-linear map from $A \times B \rr S$ is a map $\psi: A \rr \Hom(B, S)$ such that \be{psi} \psi(a i^{-1}) = i.\psi(a) \text{ for } i \in R, a \in A. \ee Indeed, for each $m: A \times B \rr S$ the map $\psi_m: A \rr \Hom(B, S)$ given by $\psi_m(a) = m(a, \cdot)$ satisfies (\ref{psi}); and it is easy to see that all such maps come from $R$-linear maps $m.$ Using either of these two descriptions, it is easy to verify that the functor \be{F:A,B} F_{A, B} (S) = \{ R \text{-linear maps } A \times B \rr S \} \ee is corepresented by the quotient set \be{A:r:B} A \times_R B := A \times B / \equiv \ee where $(a, b) \equiv (ai, i^{-1}b)$ for $a\in A, b \in B, i \in R:$ i.e., $\Hom(A \times_R B, S) = F_{A, B}(S).$

A variant of the above construction is as follows: let $A_1$ be a right $R$-set, $A_2, \ldots, A_{r-1}$ be left-right $R$-sets, and $A_r$ a left $R$ set.  Then for any set $S$, an $R^{r-1}$-linear map $m: A_1 \times \cdots \times A_r \rr S$ is one such that \[ m(a_1 i_1, i_1^{-1} a_2 i_2, i_2^{-1}, \ldots, i_{r-1}^{-1} a_r) = m(a_1, \ldots, a_r) \text{ for } a_{k} \in A_{k}, i_j \in I. \] One can then check that the functor \[ F_{A_1, \ldots, A_r} (S) = \{ R^{r-1}- \text{linear maps } m: A_1 \times \cdots \times A_r \rr S \} \] is represented by \[ A_1 \times_R A_2 \times \cdots \times_R A_r:= A_1 \times A_2 \cdots \times A_r / \equiv \] where $(a_1, \ldots a_r) \equiv (a_1 i_1, i_1^{-1} a_2 i_2, i_2^{-1}, \ldots, i_{r-1}^{-1} a_r)$ for $(i_1, \ldots, i_{r-1}) \in I^{r-1}$ generates the equivalence relation. Suppose in addition that $A_1, A_2, A_3$ carry both a left and right $R$-action. Then $A_1 \times_R A_2$ inherits a left and right $R$-set structure, and similarly for $A_2 \times_R A_3.$ Hence we may form the sets $(A_1 \times_R A_2) \times_R A_3$ and $A_1 \times_R (A_2 \times_R A_3).$ It is then easy to verify the following associativity of functors, \be{F:ass} F_{A_1, A_2, A_3} = F_{A_1 \times_R A_2, A_3} = F_{A_1, A_2 \times_R A_3}, \ee which will be used below in the proof of Proposition \ref{fin:hec:alg} below.

\tpoint{Explicit Description of Fibers} \label{fiber:exp}  Let $R \subset \Gamma_+$ be a subgroup, and let $A$ be a right $R$-coset and $B$ a left $R$-coset of $\Gamma.$ Define the multiplication $m_{A, B}: A \times_R B \rr \Gamma_+$ which sends $(a, b) \mapsto ab.$ It is clearly $R$-linear. For any $x \in \Gamma_+$, we can easily verify that there is a bijection of sets  \be{fiber:int} m_{A, B}^{-1}(x)= R \setminus A^{-1}x \cap B.\ee

\emph{Variant 1:} Suppose that $A, B$ are $R$-double cosets of $\Gamma_+.$ Then we can describe the fibers $m^{-1}(x)$ in terms of left or right $R$-cosets: \be{fiber:int:2} m^{-1}(x) = R \setminus A^{-1} x \cap B = A \cap x B^{-1} / R. \ee In this case, we see that $m^{-1}(x)$ only depends on the coset of $x$ in $R \setminus \Gamma_+ / R.$

\emph{Variant 2:} Let $L, R, H$ are three subgroups of $\Gamma_+,$ and $A$ be a $(L, R)$ double coset and $B$ a $(R,H)$ double coset. Then we have \be{fiber:int:3} m^{-1}(x) = R \setminus A^{-1} x \cap B \ee and we see that $m^{-1}(x)$ only depends on the class of $x$ in $L \setminus G  /H.$

\tpoint{Finite Hecke Datum} \label{sec-finhecdat} Let $R \subset \Gamma_+$ be as above. Let us write $\Lam_+:=\Lam_{R, R},$ the set indexing $R$-double cosets of $\Gamma_+,$ and $X: \Lam \rr \Gamma_+.$ For $\lambda, \mu \in \Lam$ let $m_{\lambda, \mu}: X^{\lambda} \times_R X^{\mu} \rr \Gamma_+$ denote the multiplication map. Note that by (\ref{fiber:int:2}) the fibers $m^{-1}(x)$ for $x \in \Gamma_+$ only depend on the class of $x$ in $R \setminus \Gamma_+ / R = \Lambda.$ We thus write $m^{-1}(\mu)$ to denote $m^{-1}(x)$ for any $x \in X^{\mu}.$

\begin{de} \label{hec-dat}  We shall say that $(R, \Lam)$ (we drop the map $X$ from our notation) is a \emph{finite Hecke datum} for the semi-group $\Gamma_+ \subset \Gamma$ if it satisfies the following two conditions,

\begin{enumerate} \label{hm}
\item \textbf{(H1)}  For any $\lambda, \mu, \nu \in \Lam$ we have $| m^{-1}_{\lambda, \mu}(\nu) |$ is a finite set.

\item \textbf{(H2)} Given any $\lambda, \mu \in \Lam,$ the set of $\nu \in \Lam$ such that  $| m^{-1}_{\lambda, \mu}(\nu) | \neq 0$ is finite.
\end{enumerate}

\end{de}

\noindent Let $(R, \Lam_+)$ be a finite Hecke datum. Recall that we have the space of finitely supported $R$-binvariant functions which we shall just denote by $H(\Gamma_+, R):= \mc{F}(R \setminus \Gamma_+/ R).$ Given any $\lambda, \mu \in X$, the assumptions (H1) and (H2) above allow us to make sense of the sum \be{pre:conv} \sum_{\nu \in X} | m_{\lambda, \mu}^{-1}(\nu)| \; X^{\nu} \ee as an element of $H(\Gamma_+, R)$ and so we can define a map $\star:H(\Gamma_+, R) \times H(\Gamma_+, R) \rr H(\Gamma_+, R)$ by linearly extending the above map, \be{conv} X^{\lambda} \star X^{\mu} = \sum_{\nu \in X} |m_{\lambda, \mu}^{-1}(\nu)| \; X^{\nu}. \ee

\tpoint{Finite Convolution Hecke Algebras} \label{sec-finhecalg} The importance of the notion of finite Hecke datum is provided by the following result, which is certainly well-known. We sketch a few points of the proof as we shall need several variants of it in the sequel.

\begin{nprop} \label{fin:hec:alg} Let $(R, \Lambda_+)$ be a finite Hecke datum.  Then the map $\star$ in (\ref{conv}) equips $H(\Gamma_+, R)$ with the structure of an associative $\C$-algebra with a unit. The unit is given by convolution with the characteristic function of $R.$ \end{nprop}
\begin{proof} Let us show the associativity of $\star:$  let $\lam, \mu, \nu \in \Lam_+.$  Given $\xi \in \Lam_+,$ and any $x \in X^{\xi}$ we would like to see that \begin{eqnarray} \label{star:ass} (X^{\lambda} \star X^{\mu}) \star X^{\nu} (x) = X^{\lambda} \star (X^{\mu} \star X^{\nu}) (x). \end{eqnarray} Let $m_{\lam, \mu, \nu}: X^{\lam} \times_R X^{\mu} \times_R X^{\nu} \rr \Gamma_+$ be the map induced by multiplication. Consider the $R$-bilinear sets $X^\lam \times_R X^{\mu}$ and $X^{\mu} \times_R X^{\nu}.$ We then have from (\ref{F:ass}) that \be{m:ass} F_{(X^{\lam} \times_R X^{\mu}) \times_R X^{\nu} }(\Gamma_+) = F_{X^{\lam} \times_R X^{\mu} \times_R X^{\nu}}(\Gamma_+) = F_{ X^{\lam} \times_R (X^{\mu} \times_R X^{\nu})}(\Gamma_+). \ee The map $m_{\lambda, \mu, \nu} \in F_{X^{\lam} \times_R X^{\mu} \times_R X^{\nu}}(\Gamma_+)$ and hence defines corresponding maps which we denote by $m_{(\lambda, \mu), \nu} \in F_{(X^{\lam} \times_R X^{\mu}) \times_R X^{\nu} }(\Gamma_+)$ and $m_{\lambda, (\mu, \nu)} \in F_{X^{\lam} \times_R (X^{\mu} \times_R X^{\nu}) }(\Gamma_+).$ Then one can show that the left hand side of (\ref{star:ass}) is equal to $m_{(\lambda, \mu), \nu}^{-1}(x)$ and the right hand side is equal to $m_{\lambda, (\mu, \nu)}^{-1}(x);$ both of which are in turn equal to $m_{\lambda, \mu, \nu}^{-1}(x)$ which proves the associativity. We supress the proof of this fact, as well as that of the rest of the Proposition. \end{proof}

\tpoint{Finite Hecke Modules} \label{action:section} Let $L, R \subset \Gamma_+$ be as in \S  \ref{conv:note}. Let us write $\Lam_+$ for $\Lambda_{R, R},$ and continue to write $H(\Gamma_+, R)$ for $\mc{F}(R \setminus \Gamma_+ / R).$ Now suppose that $\Omega:= \Lam_{L, R}$ is a set of representatives for the $(L, R)$-double cosets of $\Gamma$ (\emph{note:} we want to really consider  $\Gamma$ here instead of $\Gamma_+,$ as this is the case which comes up in the sequel). For each $\mu \in \Omega$ we denote the corresponding double coset by $Y^{\mu}.$ Let us write $M(L, \Gamma, R)$ for the set $\mc{F}(L \setminus \Gamma / R).$ Given $\lambda \in \Omega, \mu \in \Lam_+$ we have a map \[ a_{\lambda, \mu}: Y^{\lambda} \times_R X^{\mu} \rr \Gamma \] defined by multiplication.

\begin{de} \label{hec-mod}  We say that the collection $(R, \Lambda_+; L, \Omega)$ is a \emph{finite Hecke module datum} if

\begin{enumerate}
\item  \textbf{(M0)} $(R, \Lam_+, X)$ is a finite Hecke datum

\item  \textbf{(M1)} For any $\lambda \in \Omega$ and $\mu \in \Lam_+$ the map $a_{\lambda, \mu}$ has finite fibers.

\item  \textbf{(M2)} For any $\lambda \in \Omega$ and $\mu \in \Lam_+,$ the there are only finitely many $\nu \in \Omega$ such that $a_{\lambda, \mu}^{-1}(\nu)$ is non-empty.

\end{enumerate}
\end{de}

Using the properties (M1) and (M2) we can make sense, for each $\lambda \in Y, \mu \in X$ of the sum, \be{fin:conv} \sum_{\nu \in \Omega} | a_{\lambda, \mu}^{-1}(\nu) | Y^{\nu} \ee as an element of $M(L, \Gamma, R).$ Thus, we can define a map $\star_a: M(L, \Gamma, R) \times H(\Gamma_+, R) \rr M(L, \Gamma, R)$ by linearly extending the following formula, \be{fin:act} Y^{\lambda} \star_a X^{\mu}  := \sum_{\nu \in \Omega} | a_{\lambda, \mu}^{-1}(\nu) | Y^{\nu}. \ee We omit the proof of the following since it is very similar to that of Proposition \ref{fin:hec:mod}.

\begin{nprop} \label{fin:hec:mod} Let $(R, \Lam_+; L, \Omega)$ be a finite Hecke module datum.  Then $\star_a$ defines an associative, unital right $H(\Gamma_+, R)$-module structure on $M(L, \Gamma, R),$ where the unit element of $H(\Gamma_+, R)$ acts as the identity map on $M(L, \Gamma, R).$   \end{nprop}

\tpoint{A Simplified Critrerion} \label{sec-finsimple} In what follows, we shall use the following variant of the above results, which reduces the number of conditions one needs to verify for $(R, \Lam_+; L, \Omega)$ to be a finite Hecke module datum. We leave the proof to the reader (see Proposition \ref{si:hecke:mod:simple} for more details in a slightly more complicated context).

\begin{nprop} \label{hecke:mod:simple} The quadruple $(R, \Lam_+; L, \Omega)$ is a Hecke module datum if $(R, \Lam_+)$ satisfy (H1) and $(R, \Lam_+; L, \Omega)$ satisfy (M1) and (M2). \end{nprop}

\subsection{Completions of Convolution Algebras} \label{completions}

\tpoint{Semi-infinite Support} Let $\Lambda$ be an abelian group whose underlying set is
equipped with a partial order $\leq$ which is compatible with the
group structure: i.e, if $\lambda \leq \mu$ then $\lambda + \nu \leq
\mu + \nu$ for $\lambda, \mu, \nu \in \Lambda.$  For elements $\lambda
\geq \mu$ we shall write \be{lam:mu:int}  [ \lambda, \mu ] := \{ \nu \in \Lambda |
  \lambda \leq \nu \leq \mu \} .\ee  For an element $\lambda \in \Lambda$ we shall set $$\mf{c}(\lambda) = \{ \mu \in \Lambda: \mu \leq \lambda \}.$$ A subset $\Xi \subset \Lambda$ is said to be \emph{semi-infinite} if there exists a finite set of elements $\lambda_1, \ldots,\lambda_r$ such that \be{xi:2} \Xi \subset \bigcup_{i=1}^r  \mf{c}(\lambda_i). \ee

If $\Lambda^+ \subset \Lambda$ is a sub-semigroup, it inherits an order which we shall continue to denote by $\leq$ and all of the constructions above can be repeated with $\Lambda^+$ in place of $\Lambda.$

\tpoint{Semi-infinite Convolution Hecke Algebras} \label{si:dat} Let $\Gamma_+ \subset \Gamma, R, X$ be as in \S \ref{conv:note}. As before set $\Lambda_+:= \Lambda_{R, R}.$ We assume that $\Lambda_+$ is an ordered abelian semigroup with order denoted by $\leq.$ Let $f: \Gamma_+ \rr \C$ be a function which is $R$-binvariant. We may define its support $\Lam_f \subset \Lam_+$ as in (\ref{supp:f}). We shall say that $f$ has semi-infinite support if $\Lambda_f \subset \Lambda_+$ is a semi-infinite subset. Denote the set of $R$-binvariant semi-infinitely supported functions on $\Gamma_+$ by $H_{\leq}(\Gamma_+, R).$  For $\lambda, \mu \in \Lam_+$ denote by $m_{\lambda, \mu}: X^{\lambda} \times_R X^{\mu} \rr \Gamma_+$ the natural multiplication map.

\begin{de} \label{de:sim} A collection $(R, \Lam_+, \leq)$ will be said to be a \emph{semi-infinite Hecke algebra datum} \label{si:hecke:alg} if the following conditions are met,
\begin{enumerate}

\item \textbf{(F)} For every pair $\lambda, \mu \in \Lam_+$ the set $[ \lambda, \mu ]$ is finite.

\item \textbf{(SH1)} For any $\lambda, \mu \in \Lam_+$ the fibers of the map $m_{\lambda, \mu}$ are finite.

\item \textbf{(SH2)} For any $\lambda, \mu \in \Lam_+,$ the composition $m_{\lambda, \mu}^{-1}(\nu) \neq 0$ implies that $\nu \leq \lambda + \mu.$
\end{enumerate}

\end{de}

Given any $\lambda, \mu \in X$, the assumptions (SH1) and (SH2) allow us to make sense of the sum \[ \sum_{\nu \in \Lam_+} | m_{\lambda, \mu}^{-1}(\nu)| \; X^{\nu}  \] as an element of $H_{\leq}(\Gamma_+, R).$  Setting \[ X^{\lambda} \star {X^{\mu}} = \sum_{\nu \in \Lam_+} | m_{X^{\lambda}, X^{\mu}}^{-1}(\nu)| \; {X^{\nu}}, \] the Proposition below tells us in particular that we may extend this multiplication to a map \be{star:si} \star:H_{\leq}(\Gamma_+, R) \times H_{\leq}(\Gamma_+, R) \rr H_{\leq}(\Gamma_+, R)\ee using the assumptions (Fin), (SH1), and (SH2).

\begin{nprop} \label{si:hec:alg} Let$(R, \Lam_+, \leq)$ be a semi-infinite Hecke datum.  Then the product $\star$ endows $H_{\leq}(\Gamma_+, R)$ with the structure of an associative, unital algebra with unit being given by the trivial $R$-double coset $R$. \end{nprop}
\begin{proof} The proof is very similar to that of Proposition \ref{fin:hec:alg}, and so let us just verify that the above formula extends to give a well-defined map as in (\ref{star:si}). Then let $f_1, f_2 \in
  H_{\leq}(\Gamma_+, R)$ which are of the form \[ \begin{array}{lcr} f_1 = \sum_{\mu
  \leq \lambda_1} a_{\mu} {X^{\mu}} & \text{ and } & f_2 =
  \sum_{\nu \leq \lambda_2} b_{\nu} {X^{\nu}} \end{array} \]
  Suppose that \[ {X^{\mu}} \star {X^{\nu}} = \sum_{\eta \leq \mu
  + \nu} c_{\eta}^{\mu, \nu} X^\eta .\] Then we write
\begin{eqnarray}
f_1  \star f_2 &=& \sum_{\mu \leq \lambda_1, \nu \leq \lambda_2} a_{\mu}
  b_{\nu} {X^{\mu}} \star {X^{\nu}} \\
&=& \sum_{\mu \leq \lambda_1, \nu \leq \lambda_2} a_{\mu} b_{\nu} (
  \sum_{\eta \leq \mu + \nu} c_{\eta}^{\mu, \nu} {X^{\nu}}) \\
&=& \sum_{\mu \leq \lambda_1, \nu \leq \lambda_2, \eta \leq \mu + \nu}
  a_{\mu} b_{\nu} c_{\eta}^{\mu, \nu} X^{\eta}  \end{eqnarray} For
  fixed $\eta$ we have that $\eta \leq \mu + \nu \leq \lambda_1 +
  \lambda_2$ and by property (F) there are only finitely many values
  which $\mu + \nu$ can take.  Again for fixed $\eta$, since $\mu \leq \lambda_1$ and $\nu
  \leq \lambda_2$ there are only finitely many terms in the sum, \[
  \sum_{\mu \leq \lambda_1, \nu \leq \lambda_2, \eta \leq \mu + \nu}
  a_{\mu} b_{\nu} c_{\eta}^{\mu, \nu} \] and hence the formula for
  $f_1 \star f_2$ is well-defined.  To show that $f_1 \star f_2$ has
  semi-infinite support, we note that \[ \Lambda_{f_1 \star f_2} \subset
  \bigcup_{\mu \leq \lambda_1, \nu \leq \lambda_2} \Lambda_{X^{\mu}
  \star X^{\nu}} \subset \bigcup_{\mu \leq \lambda_1, \nu \leq
  \lambda_2, \eta \leq \mu + \nu} \mf{c}({\eta}) \subset \bigcup_{\eta \leq
  \lambda_1 + \lambda_2} \mf{c}({\eta}) \] \end{proof}

\tpoint{Semi-infinite Hecke Modules} \label{sec-si:mod} Let $L \subset \Gamma_+$ be a subgroup of $\Gamma.$ Denote by $\Omega:= \Lambda_{L, R}$ the set parametrizing $(L, R)$-double cosets of $\Gamma.$ Assume that $\Omega$ is an ordered abelian group with order $\leq.$ Furthermore, we assume this construction is compatible with the constructions from the previous paragraph in the following sense: we are given a fixed embedding $\Lambda_+ \subset \Omega$ such that the order on $\Lambda_+$ is the restriction of the order on $\Omega.$ For a function $f: \Gamma \rr \C$ which is $(L, R)$-binvariant, we can define its support $\Omega_f \subset \Omega$ as in (\ref{supp:f}), and we say that $f$ is semi-infinitely supported if $\Omega_f \subset \Omega$ is a semi-infinite subset. Denote as in \S \ref{action:section} the map $Y: \Omega \rr \Gamma$ which sends $\nu \in \Omega$ to the corresponding coset $(L,R)$-double coset $Y^{\nu}.$  Given $\lambda \in \Omega, \mu \in \Lambda_+$ we have a map $a_{\lambda, \mu}: Y^{\lambda} \times_R X^{\mu} \rr \Gamma,$ defined via multiplication.

\begin{de} We say that the collection $(R, \Lambda_+; L, \Omega, \leq)$ as above is a \emph{semi-infinite Hecke module datum} \label{si:hecke:mod} if

\begin{enumerate}

\item \textbf{(M0)} $(R, \Lambda_+, \leq )$ is a semi-infinite Hecke module datum

\item \textbf{(MFin)} For $\lambda, \mu \in \Omega$ with $\lambda \leq \mu,$ the set $[\lambda, \mu]$ is finite.

\item \textbf{(SM1)} For any $\lambda \in \Omega$ and $\mu \in \Lam_+$ the map $a_{\lambda, \mu}$ has finite fibers.

\item \textbf{(SM2)} For any $\lambda \in \Omega$ and $\mu \in \Lam_+,$ the map if $a_{\lambda, \mu}^{-1}(\nu) \neq \emptyset$ then $\nu \leq \lambda + \mu.$

\end{enumerate}

\end{de}

Using the properties (SM1) and (SM2) we can make sense, for each $\lambda \in \Omega, \mu \in X$ of the sum, \[ \sum_{\nu \in \Omega} | a_{\lambda, \mu}^{-1}(\nu) | Y^{\nu} \] as an element of $M_{\leq}(\Gamma; L, R).$ Thus, we can define a map $\star$ via the formula \be{semiinf:act} Y^{\lambda} \star X^{\mu}  = \sum_{\nu \in \Omega} | a_{\lambda, \mu}^{-1}(\nu) | Y^{\nu}. \ee The following proposition, whose proof we omit, then shows that $\star_a$ extends to a map \be{semiinf:act:2}
M_{\leq}(\Gamma; L, R) \times H_{\leq}(\Gamma_+, R) \rr M_{\leq}(\Gamma; L, R). \ee

\begin{nprop} Let $(R, \Lambda_+; L, \Omega, \leq)$ be a semi-infinite Hecke module datum.  Then $\star_a$ defines a right $(H_{\leq}(\Gamma_+, R), \star)$-module structure on $M_{\leq}(\Gamma; L, R).$ \end{nprop}

\tpoint{A Simplified Criterion} \label{sec-si:simple} In practice, one has the following slight but useful strengthening of the above result.

\begin{nprop} \label{si:hecke:mod:simple} Let $(R, \Lambda_+; L, \Omega, \leq)$ satisfy all of the conditions of being a semi-infinite Hecke module datum except (SH1).  Then (SH1) follows and $(R, \Lambda_+, L, \Omega, \leq)$ is in fact a semi-infinite Hecke module datum. \end{nprop}

\begin{proof} Let $\lambda, \mu \in \Lam_+$ and $\nu \in \Omega$ and consider the map, \be{a:lmv} a_{\nu, \lambda, \mu}: Y^{\nu} \times_R X^{\lambda} \times_R X^{\mu} \rr \Gamma \ee induced by multiplication, i.e.,
  which sends $(x, y, z) \mapsto xy(z).$  It suffices to show that the
  fibers of this map are finite. Indeed, if (SH1) were not satisfied, then there would exist some $\lambda, \mu \in \Lam_+$ as above and $\xi \in \Lam_+$ such that $m_{\lambda, \mu}^{-1}(\xi)$ is infinite. Choose any element $z$ which lies in the image of the multiplication map $a_{\nu, \xi}: Y^{\nu} \times_R X^{\xi} \rr \Gamma.$ Then $a_{\nu, \lambda, \mu}^{-1}(z)$ will be infinite.

  By condition (SM2) we have that the image of $a_{\nu, \lambda}$ lies in the union $\bigcup_{\eta \leq \lambda + \nu} Y^{\eta}.$ For any such $\eta,$ we know that if  $Y^{\xi}$ lies in the image of $a_{\eta, \mu},$ we must have $\xi \leq \nu + \mu.$ Hence, for each such $\xi$ we have inequalities, \[ \xi \leq \eta + \mu \leq \lambda + \mu + \nu. \] From this it follows that there can be only finitely many $\eta$ for fixed $\xi, \lambda, \mu, \nu.$ The finiteness of the fibers of $a_{\nu, \lambda, \mu}$ follows from associativity using (SM1): for any element $z$ in the image of this map, we must have $z \in Y^{\xi}$ as above for fixed $\xi, \lambda, \mu, \nu$ as above.
\end{proof}

\section{Iwahori Theory I: "Affine" Hecke Algebras and Convolution Hecke algebras}

Fix the notations of $\S3$ in this section. We shall now apply the axiomatics developed in the previous section to construct a convolution algebra $H(G_+, I)$ on the space of $I$-double cosets of $G_+.$ This algebra may be identified with an "affine" Hecke algebra (a slight variant of Cherednik's DAHA) as we show in Theorem \ref{main:iwahori}. The proof of this theorem rests on the construction of a certain family of commuting elements $H(G_+, I)$ (see Proposition \ref{tp-const}) whose proof is deferred until \S \ref{iwahori-2}.

\subsection{"Affine" Hecke Algebras } \label{is-ahec}

We begin by defining the precise variant of Cherednik's DAHA which will arise from our group theoretic convolution algebra.

\spoint Fix the notations as in  \S \ref{sec:loopalg}. The Weyl group $W$ of our Kac-Moody Lie algebra $\f{g}$ is a Coxeter group with length function, $\ell: W \rr \zee$.  As such, we can associate a Hecke algebra to it as follows:  first, define the \emph{Braid group} $B_W$ as the group with generators $\mathbb{T}_w$ for $w \in W$ subject to the relations, \be{braid:rel} \mathbb{T}_{w_1} \mathbb{T}_{w_2} = \mathbb{T}_{w_1 w_2} \text{ if } \ell(w_1) + \ell(w_2) = \ell(w_1w_2). \ee Let $v$ be an indeterminate, and consider $F= \C(v),$ the field of rational functions in the intermediate $v.$ \footnote{ One could also consider the ring $\C[v, v^{-1}]$ in what follows}  The \emph{Hecke algebra} $\dhec_W$ associated to $W$ is then the quotient of the group algebra $F[B_W]$ by the ideal generated by the relations \be{hecke:rel} (\mathbb{T}_a + 1)(\mathbb{T}_a - v^{-2})=1 \; \text{ for } a \in \Pi^{\vee}. \ee

\renewcommand{\v}{\vee}

Recall that $\Lambda^{\vee}$ was the weight lattice of $\mf{g}.$ Now consider the group algebra $R = \C[ \Lambda^{\vee}].$ For $\lv \in \Lv$, denote by $\Theta_{\lv}$ the corresponding element of $R$ subject to the relations $\Theta_{\lv} \Theta_{\mv} = \Theta_{\lv + \mv}$ if $\lv, \mv \in \Lv.$ Following Garland and Grojnowski \cite{gg} we can then define $\dhec,$ the \emph{"affine" Hecke algebra associated to $W$} as the algebra generated by $\dhec_W$ and $\C[\Lambda^{\vee}]$ subject to the relations \be{bern:rel} \mathbb{T}_a \Theta_{\lambda^{\vee}} - \Theta_{w_a(\lambda^{\vee})} \mathbb{T}_a = (v^{-2} - 1)\frac{ \Theta_{\lambda^{\vee}} - \Theta_{w_a \lambda^{\vee}}}{1 - \Theta_{-a^{\vee}}} \ee Expanding the right hand side of (\ref{bern:rel}) as a series in $\Theta_{-a^{\vee}},$ the right hand side is seen to be an element in $F[\Lv],$ i.e.,
\be{bern:rel-2} (\ref{bern:rel} ) = \begin{cases} 0 &  \text{ if } \la a, \lv \ra = 0 \\
(v^{-2} - 1) ( \Theta_{\lv} + \Theta_{\lv - a^{\vee}} + \cdots + \Theta_{w_a \lv + a^{\vee} })  & \text{ if } \la a, \lv \ra > 0 \\  (1 - v^{-2}) ( \Theta_{w_a \lv} + \Theta_{w_a \lv - a^{\vee}} + \cdots + \Theta_{w_a \lv - ( \la a, w_a \lv \ra -1 ) a^{\vee})  } & \text{ if } \la a, \lv \ra <  0 \end{cases} \ee

%

 Recall from \S \ref{titscone} that $\Lam^\vee$ is equipped with a grading $\Lam^{\vee} = \oplus_{r \in \zee} \Lam_r^\vee.$ Hence, $F[\Lambda^{\vee}]$ is also equipped with a grading \be{grading:R} F[\Lambda^{\vee}] = \bigoplus_{r \in \zee} F[\Lambda_r^{\vee}] \ee where $F[ \Lambda_r^{\vee} ]$ is the $F$ span of $\Theta_{\lambda^{\vee}}$ for $\langle \delta, \lambda^{\vee} \rangle = r.$ The algebra $\dhec$ also inherits a $\zee$-grading in this way, and we denote by $\dhec_r$ the typical graded piece for $r \in \zee.$
As each graded piece $\Lam^\vee$ is $W$-invariant, so too are the pieces $F[\Lambda_r^\vee].$ Using (\ref{bern:rel}) we can easily deduce that for each $r \in \zee,$ the subspace $F[\Lambda_r^\vee]$ is a $\dhec_W$-module. In general $\dhec_r$ is not a subalgebra since $\Lam_r^\vee$ is not closed under addition. However, $\dhec_0$ is a subalgebra and is essentially the Double Affine Hecke Algebra (DAHA) of Cherednik . Let $\dhec'_0 \subset \dhec$ be the subalgebra generated by $\dhec_W$ and $\Theta_{n \cc}$ for $n \in \zee$ and $\cc$ the minimal positive imaginary coroot.  As $W$ fixes $\cc$ and $\zee \cc$ is closed under addition, the subspace $\dhec'_0$ is also a subalgebra of $\dhec.$ What will be important for us is that the following subspace of $\dhec,$
\be{dhec:p} \dhec_+:= \dhec'_0 \oplus \bigoplus_{r > 0} \dhec_r \ee is also a subalgebra. Indeed, this follows from the following simple result whose proof we omit

\begin{nlem} \label{dhecp:tits} The algebra $\dhec_+$ is the subalgebra of $\dhec$ generated by $\dhec_W$ and $F[X],$ the group algebra of the Tits cone $X \subset \Lam^\vee.$ \end{nlem}

\spoint  We would now like to construct a module for the algebra $\dhec.$ Let $\mps$ be the $F$-vector space with basis $\ve_x$ for $x \in \aw$ (denote by $\ve_1$ the basis element corresponding to the identity element of $\aw.$) Then we set \begin{eqnarray} \Theta_{\lv} \ve_1 &=& \ve_{\lv} \text{ for } \; \lv \in \Lambda^\vee \\ \mathbb{T}_w \ve_{\mu^\v} &=& \ve_{\pi^{\mu^\v} w} \text{ for } \; w \in W, \mu^\v \in \Lambda^\vee. \end{eqnarray}   It is easy to see that this defines a representation of $\dhec$ on $\mps.$ Let $\mps_+$ denote the $F$-subspace of $\mps$ which is generated by $\ve_x$ for $x \in \aw_X.$ Then one has
that $\mps_+ \subset \dhec_+ \cdot \ve_1.$ Indeed, if $x= \pi^{\mu^\v} w \in \aw_X$ then we have that \be{mps:1} \mathbb{T}_w \Theta_{\mu^\v} \ve_1 = \ve_{\pi^{\mu^\v} w}. \ee
We shall show below the reverse inclusion: $\dhec_+ \cdot \ve_1 \subset \mps_+.$

\subsection{Convolution Algebras of $I$-double cosets} \label{is-conv}

We would like to show that the set of finite linear combinations of $I$-double cosets on the semigroup $G_+$ can be equipped with the structure of a convolution algebra.  We shall moreover construct a natural action of this convolution algebra on the space of functions on $G$ with are right $I$-invariant and left invariant by the subgroup $A_{\O} U.$

\spoint From (\ref{iw:mat:G}) and Proposition \ref{im=car} , the space of $I$-double cosets of $G_+$ is parametrized by the semi-group $\aw_X$ (see (\ref{aw-X}). Denote by $T: \aw_X \rr G_+, \, x \mapsto T_x$ the map which assigns $x \in \aw_X$ to the corresponding $I$-double coset $T_x:= I x I$.  Consider the subgroup $A_\O U \subset G$ and note that from Lemma \ref{W:AO} the set of $(A_{\O} U, I)$-double cosets of $G$ are parametrized by the set $\aw.$ Denote by $\ve: \aw \rr G, \, y \mapsto \ve_y$ the map which assigns to each $x \in \aw$ the corresponding double coset in $\ve_y:= A_{\O} U x I \subset G.$ We then have the following,

\begin{nthm} \label{iw:finite} The data $(I, \aw_X)$ forms a finite Hecke datum for $G_+$ (in the sense of Definition \ref{hec-dat}), and the the triple $(I, \aw_X; A_{\O}U, \aw)$ forms a finite Hecke module datum for $G_+$ (in the sense of Definition \ref{hec-mod}).  \end{nthm}

We denote by $H(G_+, I)$ and $M(G, I):= M(A_{\O} U, G, I)$ the corresponding convolution algebra and module, noting that we have an action \be{iw:hm} \star: M(G, I) \times H(G_+, I) \rr M(G, I) \ee given by right-convolution as in (\ref{semiinf:act:2}) .  The remainder of \S \ref{is-conv} is devoted to the proof of Theorem \ref{iw:finite}.

\spoint The following result will be shown to imply Theorem \ref{iw:finite}.

\begin{nprop}  \label{prop:iw} We have the following,
\begin{enumerate}
\item Let $x, y \in \aw_X.$   Then there exists finitely many $z_i \in \aw_X, i=1, \ldots, n$ such that $$I x I y I \subset \bigcup_{i=1}^n I z_i I.$$

\item Let $x, y, z \in \aw_X.$ Then \be{iw:alg:fin} I \setminus I x^{-1} I y \cap I z I \ee is a finite set.

\item Let $x \in \aw_X$ and $y \in \aw.$ Then there exists finitely many $z_i \in \aw, i=1, \ldots, \ell$ such that \be{iw:mod:fin} A_{\O} U y I x I \subset \bigcup_{i=1}^{\ell} A_{\O} U z_i I \ee

\item Let $x \in \aw_X$ and $y, z \in \aw.$ Then \be{iw:mod:fin} I \setminus I y A_{\O} U z \cap I x I \ee is a finite set.
\end{enumerate}
\end{nprop}

\begin{proof}[Proof of Theorem \ref{iw:finite} from Proposition \ref{prop:iw}]  For $x, y \in \aw_X$ consider the multiplication map $m_{x, y}: T_x \times_I T_y \rr G_+.$ The image is clearly equal to $Ix I y I,$ and hence (H1) is equivalent to Part (1) of the above proposition. The fiber over $z \in \aw_X$ is equal to $I \setminus I y^{-1} I \cap I x I $ by (\ref{fiber:int}) and hence (H2) follows from part (2) of the above proposition. Similarly, one checks that Part (3) is equivalent to (M1) and Part (4) to condition (M2). \end{proof}

We are now reduced to showing Proposition \ref{prop:iw}, parts (1), (3), (4) , since as we have already seen from Proposition \ref{hecke:mod:simple} that part (2) will then follow automatically.

\tpoint{Proof of Proposition \ref{prop:iw}, part (1)} Recall that $I= U_{\O} U^-_{\pi} A_{\O}$ and so $$ I x I y I = I x U_{\O} U^-_{\pi} y I $$ since $\aw$ normalizes $A_{\O}.$ The following result is the group theoretic analogue of Proposition \ref{half-length}.

\begin{nlem} \label{half} For $x, y \in \aw_X$ the following spaces of cosets are finite, \be{half:1} I \setminus I x U_{\O} \text{ and } U^-_{\pi} y I / I. \ee \end{nlem}

\begin{proof}

Let us write $x = w \pi^{\mv}$ where $w \in W$ and $\mv \in X.$ Furthermore, as $\mv \in X$ we may write it as $\mv = \sigma^{-1} \lv$ with $\lambda^{\vee} \in \Lambda^{\vee}_+$ and $\sigma \in W.$

Decompose $U_{\O}$ into a semi-direct product  $U_{\O}= U_{\O}^{\sigma}U_{\sigma, \O},$ where as before \be{u:decomp} U_{\sigma, \O} = \{ u \in U_{\O} \mid \sigma u \sigma^{-1} \in U^- \} & \text{ and }  & U_{\O}^{\sigma} = \{ u \in U \mid \sigma u \sigma^{-1} \in U \}, \ee and note that $U_{\sigma, \O}$ is finite-dimensional. So we have that $I w \pi^{\mv} U_{\O} \subset I w \pi^{\mv} U^{\sigma}_{\O} U_{\sigma, \O}.$ On the other hand, we know that $\pi^{\mv} U^{\sigma}_{\O} \pi^{- \mv} \subset U^{\sigma}_{\O}$ since \be{v:1} \pi^{\mv} U^{\sigma}_{\O} \pi^{- \mv} = \sigma^{-1} \pi^{\lv} \sigma U^{\sigma}_{\O} \sigma^{-1} \pi^{- \lv} \sigma \subset  U^{\sigma}_{\O} \ee since $\pi^{\lv}$ with $\lv \in \Lv_+$ normalizes $U_{\O}.$ Hence $I w \pi^{\mv} U_{\O} \subset I w U_{\O} \pi^{\mv} U_{\sigma, \O}.$ It is easy to see that $I \setminus I w U_{\O}$ is finite, and so it follows that that $I \setminus I w U_{\O} \pi^{\mv} U_{\sigma, \O}$ is finite, by the finite-dimensionality of $U_{\sigma}.$

A similar argument shows that $ U^-_{\pi} y I / I$ is finite. \end{proof}

From the Lemma \ref{half}, there exist $u_1, \ldots, u_n \in U_{\O}$ and $u^-_1, \ldots, u^-_l \in U^-_{\pi}$ such that \be{iwprop:1.1} I x I y I \subset \bigcup_{i, j} I x u_i u^-_j y I. \ee Part (1) of Proposition \ref{prop:iw} then follows from the semi-group property of $G_+.$

\tpoint{Proof of Proposition \ref{prop:iw}, part (3)}  Given $x \in \aw_X$ and $y = w \pi^{\lv} \in \aw$ with $w \in W$ and $\lv \in \Lv,$ we have that \be{part3:1} A_\O U y I x I &=& A_\O U w \pi^{\lv} I x I = A_\O U \pi^{w \lv} w U_{\O} U^-_\pi x I \\
&=& A_{\O} U \pi^{w \lv} w U_{w, \O} U^-_{\pi} x I \\
&\subset& A_{\O} U \pi^{w \lv} I^- w x I \ee where $U_{w, \O}$ is as in (\ref{v:1}). Thus, we are reduced to showing that for any $\mv \in \Lambda^{\vee}$ and $x \in \aw_X$ the set $A_{\O} U \pi^{\mv} I^- x I $ is contained in a finite union of $(A_{\O} U, I)$-double cosets. To show this, note that since $I^- = U_{\pi} U^-_{\O} A_{\O}$ we have  \be{part3:2} A_{\O} U \pi^{\mv} I^- x I = A_{\O} U \pi^{\mv} U_{\pi} U^-_{\O} A_{\O} x I = A_{\O} U \pi^{\mv}  U^-_{\O} x I. \ee Arguing as in Lemma \ref{half} we see that $ U^-_{\O} x I/ I $ is finite, and part (3) of the Proposition follows.

\tpoint{Proof of Proposition \ref{prop:iw}, part (4)} first, we claim that we may choose $z \in \aw$ to lie in $W \subset \aw.$ Indeed, suppose that $z= w \pi^{\lv} = \pi^{w \lv} w$ with $w \in W$ and $\lv \in \Lv.$ Then \be{part4:1} I x I \cap I y A_{\O} U z = I x I \cap I y A_{\O} U \pi^{w \lv} w =  I x I \cap I y \pi^{w \lv} A_{\O} U w .\ee Replacing $y$ by $y \pi^{w \lv}$ we may assume that $z \in W.$

Next, we claim that in fact we may choose $z=1.$ Indeed, for $w \in W,$ we have that $I x I \cap I y A_{\O} U w $ consists of finitely many left $I$-cosets if and only if $I x I w^{-1} \cap I y A_{\O} U$ consists of finitely many left $I$-cosets. Since $ I w^{-1} I /I$ is finite,  an application of the semi-group property of $G_+$ shows that $I x I w^{-1} I $ is contained in finitely many $I$-double cosets. Thus, we are reduced to showing that for $x \in \aw_X$ and $y \in \aw$ that \be{part4:2} I \setminus I x I \cap I y A_{\O} U \ee is finite. Note that $I x I = I x U_{\O} U^-_{\pi},$ and by Lemma \ref{half} there exists finitely many $x_i$ for $i=1, \ldots, n$ such that $I x I \subset \bigcup_{i=1}^n I x_i U^-_{\pi}.$ Use the Iwasawa decomposition to write $I x_i$ as $I x_i' U^-$ with $x_i' \in \aw$ for each $i=1, \ldots, n.$ Noting now that $I x_i U^-_\pi \subset I x_i' u^-$ we are reduced to showing the following,

\begin{nlem} \label{part4:lem} Let $x, y \in \aw.$ Then the set $I \setminus I x U \cap I y U^-$ is finite. \end{nlem}
\begin{proof}[Proof of Lemma] For fixed $x, y \in \aw,$ consider the set \be{K:x:y} K_{x,y} = \{ k \in K \mid k (I x U \cap I y U^-) \subset ( I x U \cap I y U^-) \}. \ee Note that the set $K_{x, y}$ is invariant by left multiplication by $I.$ Moreover, if we can show that $I \setminus K_{x, y} $ has finitely many elements, the Lemma will follow since we know from \cite[Theorem 1.9(1)]{bgkp} that $K \setminus K (I x U \cap I y U^-) $ is finite.

To show that $I \setminus K_{x, y}$ is finite, it suffices to show that $K_{x,y}$ is contained in finitely many $I$-double cosets. This proof of the lemma is then concluded using the following result. \end{proof}

\begin{nclaim} \label{part4:lem:claim} Let $x \in \aw.$ There are only finitely many $w' \in W$ such that \be{part4:lem:claim:1} I w' I  \cdot I x U \cap I x U \neq \emptyset. \ee   \end{nclaim}

\begin{proof}[Proof of Claim] For $x =w \pi^{\lv}$ with $w \in W$ and $\lv \in \Lv.$ Then $I w' I \cdot I x U  = I w' I w \pi^{\lv} U$ so that if (\ref{part4:lem:claim:1}) holds, we must have $I w' I w U \cap I w U \neq \emptyset.$ Hence, we have $I w' I w \cap I w U \neq \emptyset.$ As $I w' I w \in K$ we may assume the intersection takes place in $K$, i.e.,  $I w' I w I \cap I w I \neq \emptyset.$ Over the residue field this implies \be{int:res} \bfB(\kk) w' \bfB(\kk) w \bfB(\kk) \cap \bfB(\kk) w \bfB(\kk) \neq \emptyset. \ee Let $\bfU_{w}(\kk):= \bfU(\kk) \cap w^{-1} \bfU^-(\kk) w$ and $\bfU_{-, w}(\kk) = w \bfU_w(\kk)w^{-1}.$ Since $\bfB(\kk) w \bfB(\kk) = \bfB(\kk) w \bfU_w(\kk)$ we conclude from (\ref{int:res}) that \be{int:res-2} \bfU_{-, w}(\kk) \cap \bfB(\kk) w' \bfB(\kk) \neq \emptyset. \ee For a fixed $w \in W$, there can be at most finitely many $w'$  which satisfy this condition. Indeed, the left hand side of (\ref{int:res-2}) is a finite set and the double cosets $\bfB(\kk) w' \bfB(\kk)$ are disjoint for varying $w' \in W.$ \end{proof}

\spoint Finally we record here an important property of the action of $H(G_+, I)$ on $M(G, I),$ which follows from the results of \S \ref{sec-affworder}.

\begin{nprop} \label{faith}
Let $h_1, h_2 \in H(G_+, I).$ If $\ve_1\star h_1 = \ve_1\star h_2$ then we have $h_1 = h_2.$
\end{nprop}
\begin{proof} Let $\ve_1 \in M(G, I)$ be the coset corresponding to $A_{\O}U  I $ and $T_y$ is the characteristic function of the double coset $I y I,$ with $y \in \aw_X.$ Then from the definition of convolution we have $\ve_1 \star T_y := \sum_{z \in \aw} | m_{1, y}^{-1}(z) | \ve_z$ where $m_{1, y}: A_\O U I \times_I I y I \rr G$ is the multiplication map. If $m_{1, y}^{-1}(z) \neq \emptyset,$ then by definition we have $A_\O U I y I \cap A_{\O} U z I \neq \emptyset,$ which implies that $U z I \cap I y I \neq \emptyset.$ By Proposition \ref{domorder} this implies that $y \preceq z.$ Moreover, it is easy to see that $m_{1, y}^{-1}(y) \neq \emptyset$ so that \be{triag:1} \ve_1 \star T_y = c_y \ve_y + \sum_{y \preceq z} c_z \ve_z, \, c_y \neq 0. \ee  As $\preceq$ is a partial order, a simple triangularity arguments yields the proposition, since every $h \in H(G_+, I)$ is a (finite) linear combination of $T_y$ with $y \in \aw_X.$  \end{proof}
%

\subsection{Convolution Iwahori-Hecke algebras as DAHAs}

We would now like to state our main result identifying $H(G_+, I)$ with $\dhec_+.$

\spoint We have constructed the convolution algebra of $H(G_+, I)$ as well as the module of $(A_{\O}U, I)$ invariant functions $M(G, I)$ over this algebra in \S \ref{is-conv}. Consider the subspace $H_W \subset H(G_+, I)$ spanned by the double cosets $T_w:= I w  I$ with $w \in W \subset \aw.$ In the same manner as in \cite{im}, we can show

\begin{nprop} \label{hec_w} The subspace $H_W \subset H(G_+, I)$ is a subalgebra. Moreover, it is spanned by the elements $T_w$, $w \in W$ together with the following relations, \be{hec_w:relns} T_{w_1} T_{w_2} &=& T_{w_1 w_2} \text{ if } \ell(w_1 w_2) = \ell(w_1) + \ell(w_2) \\ (T_a +1)(T_a - q) &=& 0 \text{ for } a \in \Pi .\ee \end{nprop}

\spoint The module $M(G, I)$ has a basis $\ve_x$ with $x \in \aw.$ In this notation $\ve_1$ is the characteristic function of $A_{\O} U  I.$ We then have the following simple properties,

\begin{nprop} \label{simple-mod} Let $w \in W$ and $\lv \in \Lv_+.$ Then we have \begin{enumerate}
\item $\ve_1 \star T_w = \ve_w $ for $ w \in W$
\item $\ve_1 \star T_{\pi^{\lv}} = \ve_{\pi^{\lv}}$ for any $\lv \in \Lv_+.$ \end{enumerate} \end{nprop}

\begin{proof} The first result is straightforward. Let us sketch a proof of the second. First, let us first determine the support of $\ve_1 \star T_{\lambdav}.$ To do this, we note that \be{support:dom} A_{\O} U I \pi^{\lambdav} I = A_{\O} U U_{\O} U^-_{\pi} A_{\O} \pi^{\lambdav} I = A_{\O} U U^-_{\pi} \pi^{\lambdav} I = A_{\O} U \pi^{\lambdav} I,  \ee where in the last line we have used the fact that $\pi^{- \lambdav}$ normalizes $U^-_{\pi}$ for $\lambdav \in \Lam^\vee_+.$ The above computation shows that $\ve_1\star T_{\lambdav}$ is just a multiple of $\ve_{\pi^{\lambdav}},$ and in fact this constant is equal by definition to the cardinality of the set of left $I$-cosets of the intersection, \be{mult:dom} I  \pi^{\lambdav} I \cap I \pi^{\lv} A_{\O} U = I  \pi^{\lambdav} U^-_{\pi} \cap I \pi^{\lambdav} U .\ee However, from \cite{bgkp} we have that $I  \pi^{\lambdav} U^-_{\pi} \cap I \pi^{\lambdav} U  = I \pi^{\lambdav}$ and so this cardinality is equal to $1.$ \end{proof}

\spoint In Proposition \ref{indep} we shall generalize part (2) of Proposition \ref{simple-mod} as follows.
\begin{nprop} \label{tp-const} For any $\mv \in X$ there exists an element $\tp_{\mv} \in H(G_+, I)$ which satisfies the condition \be{act:all-1} \ve_1 \star \tp_{\mv} = e^{\mv} \circ \ve_1. \ee  \end{nprop}

We may now state our main result on Iwahori-Hecke algebras.

\begin{nthm} \label{main:iwahori} There exists an isomorphism of algebras \be{phi:isom} \varphi: \dhec_+ \stackrel{\cong}{\rr} \ch,\ee where the parameter $v$ in $\dhec$ is specialized to $v =q^{-1/2}.$ \end{nthm}
\begin{proof}

\emph{Step 1:} Let $H' \subset \ch$ be the subalgebra of $\ch$ generated by $T_w$ for $w \in W$ and $\tp_{\mv}$ for $\mv \in X.$  Consider the map $\varphi: \dhec_+ \rr H'$ which sends $\mathbb{T}_w \mapsto T_w'$ for $w \in W$ and $\Theta_{\mv} \mapsto \tp_{\mv}$ for $\mv \in X.$ It is easy to see that this map is actually an algebra isomorphism.  We can see this for the restriction to $\dhec_W$ using Proposition \ref{hec_w}. Then using the fact that $\ve_1 \star \tp_{\mv}= e^{\mv} \circ \ve_1,$ the algebra generated by $\tp_{\mv}$ is commutative. Using (\ref{act:all-1}) and the basic properties of the intertwining operators from \S \ref{is-intw}, one verifies that the $\tp_{\mv}$ satisfy the Bernstein relations (\ref{bern:rel-2}). Hence the map $\varphi$ is an algebra map.  It is clearly surjective by the definition of $\dhec_+.$ It remains to see that $\varphi$ is injective. For this it suffices to show that for example the elements $\{ \theta_{\mv} T_w \}_{\mv \in X, w \in W}$ are linearly independent. However this follows by exactly the same argument as in \cite[Lemma 3.4]{lus}

\emph{Step 2:} It remains to show that $H' = H(G_+, I).$ This will follow if we show that $\ve_1 \star H' = \ve_1 \star \ch.$ It is easy to see that $\mathbb{M}_+ = M(G_+, I) \subset \ve_1 \star H'$ On the other hand, for $x \in \aw_X$ we write $\ve_1 \star T_x= \sum_{y \in \aw} c_y \ve_y.$ If $c_y \neq 0$ then we must have that $A_{\O} U I x I \cap A_\O U y I \neq \emptyset,$ or in other words, $A_{\O} U y I \cap I x I \neq \emptyset.$ By the equivalence of the Iwahori-Matsumoto and Iwasawa semigroups, this implies that $y \in \aw_X.$ So \be{H'} \ve_1 \star \ch \subset M(G_+, I) \subset \ve_1 \star H'. \ee Thus $H' = \ch$ and the Theorem is proven.
\end{proof}
\subsection{Variant: $H(G_+, I^-)$ }

In our computation of the affine spherical function, we shall need a 'negative' variant of the above results.

\spoint \label{basic-opp} Recall that $W$ naturally parametrizes the set of $I^-$ double cosets of $K$. For each $w \in W$ the corresponding $I^-$ double coset will be denoted by $T^-_w.$ One may show easily that $(I^-, W)$ forms a Hecke-module datum for the group $K.$ The corresponding algebra will be called $H^-_W:= H(K, I^-)$ and is spanned by the double cosets $T^-_w$ for $ w \in W.$ Similar to Proposition \ref{hec_w} above, we have the following presentation of $H^-_W$

\begin{nprop} \label{hec_w:min} The algebra $H^-_W$ has presentation as follows: it is spanned by the elements $T_w$, $w \in W,$ together with the following relations, \be{hec_w:relns} T_{w_1} T_{w_2} &=& T_{w_1 w_2} \text{ if } \ell(w_1 w_2) = \ell(w_1) + \ell(w_2) \\  (T_a +1)(T_a - q) &=& 0 \text{ for } a \in \Pi .\ee \end{nprop}

Consider the space $M(G, I^-)$ of functions which are $A_{\O}U$ left-invariant and $I^-$-right invariant. By Lemma \ref{W:AO}, the set of $(A_{\O}U, I^-)$-double cosets of $G$ is parametrized by $\mc{W}.$ Denote by $\ve^-_x$ the characteristic function of $A_{\O} U x I^-$ for each $x \in \mc{W},$ and so $\ve^-_1$ is the characteristic function of $A_{\O} U I^-.$ Thus $M(G, I^-)$ has a vector space basis $\ve^-_x$ with $ x \in \aw.$ It is easy to see that $H^-_W$ acts on $M(G, I^-)$ by the right-convolution defined in (\ref{fin:act}), and we denote again this convolution by \be{note:conv:min} M(G, I^-) \star H^-_W \rr M(G, I^-). \ee \begin{nprop} \label{simple-mod:min} In $M(G, I^-),$ we have the following relation, \be{v1:Tw:min} \ve^-_1 \star T^-_w = q^{\ell(w) } \ve^-_w. \ee  \end{nprop}
\begin{proof} Using Proposition \ref{hec_w:min}, we may immediately reduce to the case that $w = w_a$ is a simple reflection through $a \in \Pi.$ The Proposition can be checked easily in this case. \end{proof}

\section{Iwahori Theory II: Intertwiners and Construction of $\theta_{\mv}$} \label{iwahori-2}

\subsection{Intertwining Operators} \label{is-intw}

The aim of this subsection is to develop some basic properties of intertwining operators. The arguments are mostly analogous to the finite-dimensional setting, so we only sketch the proofs.

\spoint \label{int:conv:fns} Let $A_1, A_2, A_3 \subset G$ be three subgroups. Let $X, Y$ and $Z$ denote the sets which parametrize $(A_1, A_2)$, $(A_2, A_3)$ and $(A_1, A_3)$-double cosets of $G.$ For $x \in X$ we denote by $\ve_x$ the corresponding double coset $A_1 x A_2.$ Similarly, we define $\ve_y, \ve_z$ for $y \in Y$ and $z \in Z.$ Denote by $M(X), M(Y), M(Z)$ the spaces of \emph{all} functions on $G$ which are $(A_1, A_2)$, $(A_2, A_3)$, and $(A_1, A_3)$-binvariant. One can formally attempt to define a convolution structure, $M(X) \star M(Y) \rr M(Z)$ which is defined on characteristic functions by \be{conv:for} \ve_x \star \ve_y = \sum_{z \in Z} m_{x, y}^{-1}(z) \ve_z \ee where $m_{x, y}: \ve_x \times_{A_2} \ve_y \rr G$ is the map induced by multiplication.  To make sense of the formula (\ref{conv:for}), one needs to impose certain finiteness conditions of course.

\tpoint{Some Completions} \label{intertwiner-completions} Let $A_1 = A_{\O}U$ and let $A_2, A_3$ be arbitrary subgroups as above and $X, Y, Z$ also as above. Let $M_{fin}(X) \subset M(X)$ denote the space of all function $\phi \in M(X)$ which are supported on only finitely many double cosets. Let $R= \C[\Lv]$ be the group algebra of the coweight lattice of $\mf{g},$ which has generators $e^{\lv}$ with $\lv \in \Lv$ subject to the relation $e^{\lv} e^{\mv} = e^{ \lv + \mv}$ with $\lv, \mv \in \Lv.$ The algebra $R$ acts on $M(X)$ on the left via the formula, \be{leftR} e^{\lv} \phi (x) = q^{ - \la \rho, \lv \ra } \phi(\pi^{-\lv} x) \, \text{ for } \phi \in M_{fin}(X). \ee Indeed, it is easy to see that the above operation preserves right $A_2$-invariance.

We shall also need the following completions of $M_{fin}(X)$ which are defined with respect to the left $R$-action as follows. Let $J \subset R_{re}$ be a finite subset of real roots, and consider the subalgebra of $R$ defined as $B_J = \C [ e^{a^{\vee}} ]_{a \in J}.$  We can complete $B_J$ with respect to the maximal ideal spanned by $e^{a^{\vee}}$ for $a \in J,$ and we denote the corresponding completion by $\hat{B}_J.$  Let us set $R_J :=  \hat{B}_J \otimes_{B_J} R$ and define \be{MJ} M_J(X) := R_J \otimes_R M_{fin}(X).\ee   For each $w \in W$ we let $D_w = R_{re, +} \cap w R_{re, -}$ and we set \be{Rw:Mw} \begin{array}{lcr} R_w:= R_{D_w} \text{ and } M_w(X):= M_{D_w}(X) \end{array}. \ee

\spoint Keep the conventions of the previous paragraph. Fix $w\in W,$ and consider the group \be{uw:1} U_w:= \prod_{a \in D_w} U_a(\mc{K}) \ee It is a finite product of root groups, and carries a natural Haar measure $du_w$ which assigns to $U_w \cap K$ measure $1.$ We may consider the integral operator $\iw$ defined via the formula,  \begin{eqnarray} \label{Iw} \iw (\phi) (x) = \int_{U_{w}} \phi(w^{-1} u_w x) du_w. \end{eqnarray} where $\phi \in M$ is a function which is left $A_{\O}U$-invariant. Note that by  the way $\iw$ is defined, if $\phi \in M(X)$ then $\iw(\phi) \in M(X)$ as well. Of course, in order to use the above formula one needs to verify certain finiteness criterion. We shall return to this point below, after considering the following simple result.

\begin{nlem} \label{inter} Let $\phi_1 \in M(X)$, $\phi_2 \in M(Y)$, $w \in W,$ and $\lv \in \Lv.$ \begin{enumerate} \item Suppose that $\phi_1 \star \phi_2$, $\iw(\phi_1) \star \phi_2,$ and $\iw (\phi_1 \star \phi_2)$ are all well-defined elements in $M(Z).$ Then we have \be{inter:1} \iw (\phi_1 \star \phi_2) = \iw(\phi_1) \star \phi_2. \ee  \item Suppose that $\iw(e^{\lv} \circ \phi_1)$ and $\iw(\phi_1)$ are well-defined elements on $M(X)$, then \be{inter:2} \iw(e^{\lv} \circ \phi_1) = e^{w \lv} \circ \iw(\phi_1). \ee
\item Suppose that $w=w_{a_1} \cdots w_{a_r}$ is a reduced decomposition for $w \in W$ where $w_{a_i}$ for $i=1, \ldots, r$ are reflections through simple roots $a_i \in \Pi.$ If $\iw (\phi_1)$ and the composition $\mf{T}_{w_{a_1}} \circ \cdots \circ  \mf{T}_{w_{a_r}} (\phi_1)$ are well-defined, then \be{inter:3} \iw(\phi_1) = \mf{T}_{w_{a_1}} \circ \cdots \circ \mf{T}_{w_{a_r}} (\phi_1). \ee
\end{enumerate} \end{nlem}

The verification of the above Lemma is straightforward, and we suppress the details here.

\subsubsection{Intertwiners for $I$} \label{intw-I}

Let $A_1 = A_{\O}U$ as above, and $A_2 = I.$ Then from Lemma \ref{W:AO}, $X = \aw$ is the set parametrizing $(A_{\O}U, I)$ double cosets, and we write $M(G, I)$ for $M_{fin}(X)$ and $M_w(G, I)$ for $M_w(X)$, etc.

The following result is necessary to formally make sense of $\iw: M(G, I) \rr M(G, I),$

\begin{nlem} \label{I:intw:fin} Let $w \in W$ and $x, y \in \aw.$ Let $\mu$ denote the Haar measure on $U_w$ which assigns $U_w \cap K$ measure $1.$
Then $\mu( \{ u
  \in U_{w} | w^{-1} u y \in A_{\O} U x I \} ) < \infty.$ \end{nlem}
%

  Finally, one can verify the following simple formula for the action of $\mf{T}_a:= \mf{T}_{w_a}: M_{fin}(X) \rr M_{w_a}(X)$ for $a \in \Pi.$

\begin{nlem} \label{I:rk1}   The map $\mf{T}_a: M(G, I) \rr M(G, I)_{w_a}$ is given by the following formula on $\ve_1,$ \be{int:a:1} \mf{T}_a(\ve_1) = q^{-1} \ve_{w_a} + (1-q^{-1}) \sum_{j=1}^{\infty} e^{ j a^{\vee} } \circ \ve_1 = q^{-1} \ve_{w_a} + \frac{(1-q^{-1}) e^{a^{\vee}}}{1 - e^{a^{\vee}}} \circ \ve_1 , \ee where the fraction in the last expression is formally expanded in the completion $R_{w_a}.$ \end{nlem}

\subsubsection{Intertwiners for $I^-$} \label{int-neg}

One may repeat the same discussion as above where we replace $I$ by $I^-.$ Thus, we have maps $\iw^-: M(G, I^-) \rr M(G, I^-)_{w}$ for each $w\in W.$ Writing again $\ve^-_x$ to refer to the double coset $A_{\O}U x I^-$ for $x \in \aw,$ we may compute.

\begin{nlem} \label{I^-:rk1}   The map $\mf{T}^-_a: M(G, I^-) \rr M(G, I^-)_{w_a}$ is given by the following formula on $\ve_1,$\be{I^-:a:1} \mf{T}^-_a(\ve^-_1) = \ve^-_{w_a} + (1-q^{-1}) \sum_{j=0}^{\infty} e^{ j a^{\vee} } \circ \ve^-_1 = \ve^-_{w_a} + \frac{1-q^{-1}}{1 - e^{a^{\vee}}} \circ \ve^-_1 , \ee where the fraction in the last expression is formally expanded in the completion $R_{w_a}.$ \end{nlem}

\subsubsection{Intetwiners for $K$} \label{intw-k}

Finally we turn to the case that $A_2=K,$ so that $X = \Lv.$ Let us write $M(G, K)$ for $M(X).$ We again have maps $\iw: M(G, K)_w \rr M(G, K)_{w w'}$ for each $w, w' \in W.$ We write $\ve_{\lv}$ to refer to the double coset $A_{\O}U \pi^{\lv} K$ for $\lv \in \Lv,$ and if $\lv=0$ we write $\mathbf{1}_K:=\ve_0$ and call this element the \emph{spherical vector}.
%

\begin{nlem} \label{K:rk1} \cite[Lemma 1.13.1]{hkp}  The map $\mf{T}_a: M(G, K) \rr M(G, K)_{w_a}$ is given by the following formula on $\mathbf{1}_K$ , \be{K:a:1} \mf{T}_a(\mathbf{1}_K) = \frac{1 - q^{-1} e^{a^{\vee}}}{1- e^{a^{\vee}}} \mathbf{1}_K \ee where the fraction in the last expression is formally expanded in the completion $R_{w_a}.$ \end{nlem}

\spoint We shall find it convenient to renormalize $\mf{T}_a$ as follows. Define \be{Ka} \mf{K}_a:= \frac{1 - e^{a^{\vee}}}{1 - q^{-1} e^{a^{\vee}}} \mf{T}_a, \ee and one can again verify that $\mf{K}_a: M(G, K)_w \rr M(G, K)_{w w_a}$ for any $w \in W$ such that $\ell(w w_a) = \ell(w)+1.$ Moreover, by Lemma \ref{inter} above, we have that \be{Kw} \mf{K}_w = \mf{K}_{w_{a_1}} \circ \cdots \circ \mf{K}_{w_{a_r}} \ee for any reduced decomposition $w=w_{a_1} \cdots w_{a_r}$ of $w.$ This precise normalization is chosen so that we have \be{K:intw:sv} \mf{K}_w(\mathbf{1}_K) = \mathbf{1}_K. \ee

\subsection{A Construction of Elements in $H(G_+, I)$ } \label{is-daha}

We now begin the construction of elements $\tp_{\mv} \in H(G_+, I)$ for $\mv \in X$ which are specified in Proposition \ref{tp-const}. To each $\mv \in X$ we give an algorithm for producing certain elements $\tp^{\bullet}_{\mv} \in H(G_+, I)$ which depends on various choices (to be specified below). In the \S \ref{sec-ind}, we shall show that the outcome of our construction does not actually depend on the choices made.

\spoint For $\lv \in \Lv_+$, we set $\tp_{\lv}:= q^{ - \langle \rho, \lv \rangle } T_{\pi^{\lv} },$ and note from Proposition \ref{simple-mod} (2)  that \be{act:dominant} e^{\lv} \circ \ve_1 = \ve_1 \star \tp_{\lv}, \, \text{ for }  \lv \in \Lv_+. \ee We would like to extend this construction and define elements $\theta_{\muv}$ for any $\muv \in X.$ We refer to the elements $\mv \in w (\Lv_+)$ as the shifted $w$-chamber of $X,$ or just a shifted chamber for short.  The length $\ell(w)$ of the Weyl group element defining the shifted chamber will be called the length of the chamber. Note that within each shifted chamber, we may consider the elements with respect to the dominance order $\leq.$ Our construction below shall proceed based on both the length of the chamber and on this dominance order.

We begin with the case of chamber length $1,$ i.e., those $\mv \in X$ such that there exists a simple root $a \in \Pi$ such that $w_a \mv \in \Lv_+.$ We proceed by induction on the quantity $\la a, \mv \ra <0,$ and would like to use the formula \be{theta:ind}  `` \theta_{\muv} = T_a \theta_{w_a \mv} T^{-1}_a -  (q-1) (\tp_{w_a \mv} + \tp_{w_a \mv - a^\v} + \cdots + \tp_{w_a \mv + (\la a , \mv \ra +1) a^\v} ) T_a^{-1}"  \ee in a manner to be made precise below.

\tpoint{Step 1} If $\la a , \mv \ra = - 1$ then the right hand side of (\ref{theta:ind}) reduces to \be{j=-1}  T_a \theta_{w_a \muv} T_a^{-1} - (q-1) \tp_{w_a \mv} T_a^{-1} . \ee As $w_a \mv \in \Lambda_+$ the term $\theta_{w_a \muv}$ has already been defined as elements in $H(G_+, I)$. Moreover, the elements $T_a$ and $T_a^{-1}$ have also be defined in $H(G_+, I)$ . Thus we may take (\ref{j=-1}) to be the \emph{definition} of $\tp_{\mv}$ for any $\mv \in w_a(\Lv_+)$ with $\la a, \mv \ra = -1$ i.e.,  \be{} \tp_{\mv} = T_a \theta_{w_a \muv} T_a^{-1} - (q-1) \tp_{w_a \mv} T_a^{-1} \text{ if } \mv \in w_a(\Lv_+) \text{ and  } \la a, \mv \ra = -1. \ee

\tpoint{Step 2} We next observe the following simple result, which is useful for our inductive construction.

\begin{nlem} \label{tpc} Let $\mv \in \Lv$ be such that $w_a \mv \in \Lv_+.$ Set $d:= - \la a, \mv \ra > 0,$ and assume $d > 1.$
\begin{enumerate}
\item If $ b \in \Pi$ with $b \neq a,$ then \be{tpc:1} \la b, w_a \mv - j a^{\v} \ra \geq 0 \text{ for } j=1, 2, \ldots, d - 1, \ee i.e., the elements $w_a \mv - j a^{\v}$ for $j=1, 2, \ldots, d - 1$ lie in $w_a (\Lv_+) \cup \Lv_+$.
\item For $j=1, 2, \ldots, d - 1,$ we have \be{tpc:2}  - \la a, w_a \mv - j a^\v \ra <  d. \ee
 \end{enumerate}
\end{nlem}
\begin{proof} Part (1) follows immediately from the following two facts: (i) the inner product $\la b , a^{\vee} \ra \leq 0$ for $b$ a simple root not equal to $a;$ and (ii) $\la b, w_a \mv \ra \geq 0$ since $w_a \mv \in \Lv_+$ and $b$ is a positive root.

As for (2), we compute, \be{ind:check}  - \la a, w_a \mv - j a^\v \ra =  \la a, \mv \ra +  \la a, j a^\v \ra = - d + 2j. \ee However, $-d + 2 j < d $ since $j < d.$

\end{proof}

\tpoint{Step 3}  Fix $d > 1$, and suppose now that we have defined expressions $\theta_{\mv}$ for $\mv \in w_a(\Lv_+)$ such that $-\la a, \mv \ra < d.$ Choose now $\mv \in w_a(\Lv_+)$ with $- \la a, \mv \ra = d .$ Then the right hand side of the expression (\ref{theta:ind}) takes the form, \be{theta:ind:2} T_a \theta_{w_a \mv} T_a^{-1} -  (q-1) (\tp_{w_a \mv} + \tp_{w_a \mv - a^\v} + \cdots + \tp_{w_a \mv - (d-1) a^\v} ) T_a^{-1}. \ee From Lemma \ref{tpc} (1) we know that the elements $w_a \mv - j a^\v$ for $j=1, \ldots, d-1$ are all in $w_a (\Lv_+)$ or $\Lv_+.$ Let $\xv$ be one of these elements. If $\xv \in \Lv_+$ we know how to define $\tp_{\xv}.$ On the other hand, if $\xv \in w_a(\Lv_+)$ we know from Lemma \ref{tpc}(2) that $- \la a, \xv \ra < d,$ and so $\tp_{\xv}$ has been defined inductively. Continuing in this way, we can define $\tp_{\mv}$ for any $\mv \in w_a(\Lv_+).$

Proceeding again by induction on the length of the chamber, and then by a second induction based on dominance, we may construct elements $\tp_{\mv}$ for every $\mv \in X.$

\tpoint{Step 4} It is important to note that in this construction a number of choices have been made to define each $\tp_{\mv}.$ We denote by $\tp_{\mv}^{\bullet}$ \emph{any} element associated to a given $\mv \in X$ which can be constructed by the procedure described above. It will be shown below that the construction is independent of the choices made. i.e., that  $\tp_{\mv}^{\bullet}$ only depends on $\mv.$ Our strategy will be to show that for any of the elements $\tp^{\bullet}_{\mv} \in H(G_+, I)$ constructed above, we have a relation of the form (\ref{act:dominant}), i.e.,  $\ve_1 \star \tp^{\bullet}_{\mv} = e^{\mv} \circ \ve_1.$ Proposition \ref{faith} will then imply that $\tp^{\bullet}_{\mv}$ depends only on $\mv \in \Lv.$ Our proof proceeds in an inductive manner and rests on the following fact which is obvious from our construction.

\begin{nlem} \label{theta:a} Let $\theta_{\muv}^{\bullet}, \, \mu \in X$ constructed as above. If $\mv \notin \Lv_+$, there exists a simple root $a \in \Pi$ such that $w_a \mv > \mv$ and a sequence of elements $\tp_{w_a \mv}^{\bullet}, \tp_{w_a \mv - a^\v}^{\bullet}, \ldots, \tp_{w_a \mv - (\la a , w_a \mv \ra -1) a^\v}^{\bullet} $ such that \begin{eqnarray} \label{theta:dot:a} T_a \theta^{\bullet}_{w_a \muv} = \theta_{\muv}^{\bullet} T_a +  (q-1) (\tp_{w_a \mv}^{\bullet} + \tp_{w_a \mv - a^\v}^{\bullet} + \cdots + \tp_{w_a \mv - (\la a , w_a \mv \ra -1) a^\v}^{\bullet} ). \end{eqnarray} \end{nlem}

\subsection{Proof of Independence of Construction} \label{sec-ind}  We now show that the elements $\tp^{\bullet}_{\mv}$ defined in \S \ref{is-daha} do not depend on the various choices made in their construction.

\spoint Let $\mv \in X$ and let $\tpb_{\mv} \in H(G_+, I)$ be any of the elements constructed in \S \ref{is-daha}. Recall that we have defined the intertwining operators $\mf{T}_a,\, a \in \Pi$ for some completion of the $(R, H(G_+, I))$ bimodule $M(G, I)$ in \S \ref{intw-I}. Using these operators, we can immediately verify the following relation,

\begin{nlem} \label{Lem:t-int} Let $a \in \Pi$ and $\tpb_{\mv}$ as above. Then we have \label{tpb:int} \be{Ja-1} v_1 \star T_a \star \tpb_{w_a \mv} = q \, \mf{T}_a(v_1 \star \tpb_{w_a \mv} ) + \frac{q- 1}{1- e^{-a^\v}} \ve_1 \star \tpb_{w_a \mv}\ee \end{nlem}
\begin{proof} Note that the right hand side a priori only lives in some completion of $M(G, I).$ However, if we can prove the equality (\ref{Ja-1}) holds in some completion, it holds in $M(G,I)$ itself since the left hand side lies in $M(G, I).$ The proof of the Lemma is a simple computation, using the properties of interwining operators established earlier. Indeed, using the fact (see Lemma \ref{inter} (1)) that $\mf{T}_a$ commutes with the right convolution action and the explicit formula (\ref{int:a:1}), we find
 \begin{eqnarray} \mf{T}_a (v_1 \star \tpb_{w_a \mv} ) &=& \mf{T}_a(\ve_1) \star \tpb_{w_a \mv} \\
&=& ( q^{-1} \ve_{w_a} + \frac{ (1-q^{-1}) e^{\ac} }{1 - e^{\ac} } \ve_1 ) \star \tpb_{w_a\mv} \\
&=& q^{-1} v_1 \star T_a \star \tpb_{w_a\mv} +  \frac{ (1-q^{-1}) e^{\ac} }{1 - e^{\ac} } (\ve_1 \star \tpb_{w_a\mv} ), \end{eqnarray} where in the last line we have used the fact  (see Proposition \ref{simple-mod}) that $\ve_{w_a} = \ve_1 \star T_{w}.$
\end{proof}

\spoint Let us now consider the following

\begin{nprop} \label{indep} For any $\mv \in X$ and for any of the elements $\tpb_{\mv}$ constructed in \S \ref{is-daha}, we have \be{act:all} v_1 \star \tpb_{\mv}  = e^{\mv} \circ v_1 .\ee In particular, $\tpb_{\mv}$ only depends on $\mv \in X.$
\end{nprop}

\begin{proof} Consider the statement for each $\mv \in X,$

\begin{description} \item[$\mathbb{P}(\mv)$] For any $\tpb_{\mv}$ constructed as in \S \ref{is-daha},
we have $v_1 \star \tpb_{\mv}  = e^{\mv} \circ v_1 $  \end{description}
If $\mathbb{P}(\mv)$ is true, then we can define $\tpb_{\mv}$ unambiguously according to the faithfulness result, Proposition \ref{faith}. We shall just write $\tp_{\mv}$ in this case. We know that $\mathbb{P}(\mv)$ is true for $\mv \in \Lv_+.$ Given any $\mv \in X$, assume by induction that  $\mathbb{P}(\xv)$ is true for all $\xv  > \mv.$ Let us show that it holds for $\mv$ as well.

\emph{Step 1:} Given $\mv \in X$, from Lemma \ref{theta:a}, there exits a simple root $a \in \Pi$ such that $w_a \mv > \mv$ and such that we have a relation of the form \be{1w} T_a \theta^{\bullet}_{w_a \muv} = \theta_{\muv}^{\bullet} T_a +  (q-1) (\tp_{w_a \mv}^{\bullet} + \tp_{w_a \mv - a^\v}^{\bullet} + \cdots + \tp_{w_a \mv - (\la a , w_a \mv \ra -1) a^\v}^{\bullet} ). \ee As $w_a \mv > \mv$ we must have $\la a, \mv \ra = - \la a, w_a \mv \ra   < 0.$ Let us next note that
\be{wa:dom} w_a \mv - j a^\vee > \mv \text{ for } j=0, \ldots, \la a , w_a \mv \ra -1. \ee Indeed for these values of $j$ we find \be{wa:dom-1} w_a \mv - j a^\vee - \mv = (- \la a, \mv \ra - j ) a^{\vee} = (\la a, w_a \mv \ra - j ) a^{\vee} \geq 0 \ee By the inductive hypothesis, we now have that $\mathbb{P}(w_a \mv - j a^\vee)$ is true for $j$ as above, and in particular the elements $\tp_{w_a \mv - j a^{\vee}}$ are well-defined. So (\ref{1w}) may be rewritten as
\be{2w} T_a \theta_{w_a \muv} = \theta_{\muv}^{\bullet} T_a +  (q-1) (\tp_{w_a \mv} + \tp_{w_a \mv - a^\v} + \cdots + \tp_{w_a \mv - (\la a , w_a \mv \ra -1) a^\v} ), \ee where only the element $\tpb_{\mv}$ may in fact depend on its construction (and not only on $\mv$).

\emph{Step 2:} From Lemma \ref{Lem:t-int} we have \be{Ja-1} v_1 \star T_a \star \tp_{w_a \mv} = q \, \mf{T}_a(v_1 \star \tp_{w_a \mv} ) + \frac{q- 1}{1- e^{-a^\v}} \ve_1 \star \tp_{w_a \mv} , \ee which may be rewritten using the assumption $\mathbb{P}(w_a \mv)$ and Lemma \ref{inter} and (\ref{int:a:1}) as follows, \be{Ja-2} v_1 \star T_a \star \tp_{w_a \mv} &=&  q \, \mf{T}_a(e^{w_a \mv} \circ \ve_1 ) + \frac{q- 1}{1- e^{-a^\v}}  e^{w_a \mv} \circ \ve_1 \\
&=&  q \, e^{\mv} \mf{T}_a(\ve_1 ) + \frac{q- 1}{1- e^{-a^\v}} e^{w_a \mv} \circ \ve_1 \\
&=& q e^{\mv} (q^{-1} \ve_{w_a} + \frac{ (1-q^{-1}) e^{\ac} }{1 - e^{\ac} } \ve_1 ) + \frac{q- 1}{1- e^{-a^\v}} e^{w_a \mv} \circ \ve_1 \\
&=& e^{\mv} \ve_{w_a} + (q-1) \frac{ e^{\mv} - e^{w_a \mv} }{1 - e^{-a^{\vee}}} \circ \ve_1 \\
&=&  \label{Ja-2b} e^{\mv} \ve_{1} \star T_a \\ &+& (q-1) ( e^{w_a \mv} + e^{w_a \mv - a^{\vee} } + \cdots + e^{w_a \mv - (\la a , w_a \mv \ra -1) a^\v} ) \circ \ve_1 . \ee

\emph{Step 3:} We may also compute $v_1 \star T_a \star \tp_{w_a \mv}$ in another way using the expression (\ref{2w}):  \be{3w} &&  \ve_1 \star  \theta_{\muv}^{\bullet} T_a +  \ve_1 \star (q-1) (\tp_{w_a \mv} + \tp_{w_a \mv - a^\v} + \cdots + \tp_{w_a \mv - (\la a , w_a \mv \ra -1) a^\v} ) \\
&=&  \ve_1 \star \theta_{\muv}^{\bullet} T_a +  (q-1) (e^{w_a \mv} + e^{w_a \mv - a^\v} + \cdots + e^{w_a \mv - (\la a , w_a \mv \ra -1) a^\v} ) \circ \ve_1 \ee  where in the second line we have used the fact $\mathbb{P}(w_a \mv - j a^{\vee})$ for $j =0, \ldots, \la a , w_a \mv \ra -1.$
Comparing with (\ref{Ja-2b}) we immediately conclude that \be{conc} e^{\mv} \circ  \ve_{1} \star T_a  = \ve_1 \star \theta_{\muv}^{\bullet} \star T_a \ee and hence the claim $\mathbb{P}(\mv)$ follows from Proposition \ref{faith} since $T_a$ is invertible.

\end{proof}

\section{Spherical Theory}

In this section, we shall first review the construction by the first two authors \cite{bk} of the \emph{spherical Hecke algebra}, i.e., the convolution algebra of certain infinite collections of $K$-double cosets. The main technical step in the construction is the verification of certain finiteness properties of the fibers of convolution. This was achieved in \emph{op. cit} by interpreting these fibers geometrically. Here we sketch an alternative construction which uses in an essential way the main finiteness result of \cite{bgkp}.  While this work was in preparation, there has appeared yet another approach to proving these finiteness results by S. Gaussent and G. Rousseau \cite{gau:rou}, which  works in the setting of general Kac-Moody groups.

The main (new) result in this section is Theorem \ref{sph:mac}, which gives an explicit formula for the image under the Satake isomorphism of the characteristic function of a $K$-double coset. This generalizes the formula of  Macdonald \cite{mac:mad} (see also  Langlands \cite{lan:ep}) in the finite-dimensional setting. Its generalization to the general Kac-Moody setting is not known to us; more precisely, although it seems that large portions of the proof of Theorem \ref{sph:mac} will hold in the general Kac-Moody setting, we do not know what is the correct analogue of the formula (\ref{corr}) for the quantity defined by (\ref{def:corr}).

We fix the notation of \S \ref{section:groups} in this chapter: so $G$ will be an affine Kac-Moody group over a local field $\mc{K},$ which has subgroups $I, K, U$ etc.

\subsection{Spherical Hecke Algebras and the Satake Isomorphism} \label{sec-sat}

\spoint Recall from Theorem \ref{cartan-thmdefn} that $\Lv_+,$  was demonstrated to be in bijective correspondence with the set of $K$-double cosets of the semi-group $G_+.$ Furthermore, from Theorem \ref{iw:unique}, we see that  $\Lv$ is in bijective correspondence with the space of $(A_{\O}U,K)$ double cosets of $G$. The abelian group $\Lv$ equipped with the dominance order $\leq$ becomes an ordered abelian group in the sense of \S \ref{completions}, and $\Lv_+$ inherits the dominance order to become an ordered abelian semi-group.

\begin{nlem} \label{Mfin} The abelian group $\Lv$ equipped with the dominance order $\leq$ satisfies the following condition: for each $\lv, \mv \in \Lv$ the set \be{interval} [ \lv, \mv ] = \{ \xi^{\v} \in \Lv | \lv \leq \xi^{\v} \leq \mv \} \ee is finite.
 \end{nlem}
\begin{proof} Let $\lv, \mv \in \Lam,$ and assume that $\lv \leq \mv$ so that $\mv - \lv = \sum_{i=1}^{\ell+1} n_i a_i^{\vee}$ with $n_i \geq 0.$ Then every $\xi\in [\lv, \mv]$ is of the form $\xi = \lv + \sum_{i=1}^{\ell+1} m_i a_i^{\vee}$ with $0 \leq m_i \leq n_i.$ The finiteness of $[\lv, \mv]$ follows. \end{proof}

\spoint Recall the setup of \S \ref{completions}. Our aim in the remainder of \S \ref{sec-sat} is to show the following result which was first shown in (\cite{bk}) with slightly different terminology.

\begin{nthm} [\cite{bk}] \label{sphhecke:const} The quadruple $(K, \Lv_+; A_{\O}U, \Lambda, \leq)$ is a semi-infinite Hecke module datum in the sense of Definition \ref{de:sim}. \end{nthm}

Thus, as in \S \ref{si:dat}, we may define a convolution algebra structure on the space of $K$-double cosets $H_{\leq}(\Gamma_+, K),$ as well as an action (on the right) of this algebra on the module structure on \be{sph:prin} M_{\leq}(G; A_{\O} U, K):= M_{\leq}(G, K). \ee

\begin{proof} To prove the theorem, we shall utilize Proposition \ref{si:hecke:mod:simple}. Thus we need to verify that $(K, \Lv_+, \leq)$ satisfies conditions (Fin), (SH2), and $(K, \Lv_+; A_{\O} U, \Lv, \leq)$ satisfies (Mfin), (SM1), and (SM2). Conditions (Fin) and (MFin) follow from Lemma \ref{Mfin} so we just focus on (SH2), (SM1), and (SM2).

Let \be{v-k} \ve: \Lv \rr A_{\O} U \setminus G  / K, \ \ \lv \mapsto \ve_{\lv} \\ \label{h-k} h: \Lv_+ \rr K \setminus G_+ / K, \ \ \lv \mapsto h_{\lv} \ee be the bijections induced respectively from the Iwasawa and Cartan decompositions. The properties (SM1) and (SM2) are expressed in terms of the fiber of the map $m_{\lv, \mv}: \ve_{\lv} \times h_{\mv} \rr G,$ for $\lv, \mv \in \Lv_+.$  Using  (\ref{fiber:int}), we note that the fibers of $m_{\lv, \mv}$ are in bijective correspondence with the following intersections \be{fiber:mod:K} m_{\lv, \mv}^{-1}(\xi^{\vee}) = K \setminus K \pi^{\lv} K \cap K \pi^{\mu} U \pi^{\xi^\v} \text{ where } \lv \in \Lv_+ \text{ and }  \mv, \xi^\v \in \Lv. \ee  Now the results (SM1) and (SM2) follow respectively from the next two facts: let $\lv \in \Lambda_+$ and $\mv \in \Lambda,$ then we have the following results from \cite[Theorem 1.9(2)]{bgkp}  \be{sm1:sm2}
 | K \setminus K \pi^{\lv} K \cap K \pi^{\mv} U  | &< &\infty; \text{ and } \\
\text{ if } K \pi^{\lv} K  \cap K \pi^{\mv} U &\neq& \emptyset \text{ then } \mv \leq \lv. \ee

Finally, we address the property (SH2), which involves the fibers of the map $m_{\lv, \mv}: h_{\lv} \times h_{\mv} \rr G_+$ where $\lv, \mv \in \Lv_+.$ If $m_{\lv, \mv}^{-1}(\xv) \neq \emptyset,$ then it follows from the definitions that \be{sh2:0} K \pi^{\lv} K \pi^{\mv} K \cap K \pi^{\xv} K \neq \emptyset. \ee Using the decomposition (\ref{bruhat:K}) we may write $K=I W I,$ for the Iwahori $I$ and the Weyl group $W.$ Thus \be{sh2:1} K \pi^{\lv} K \pi^{\mv} K = \bigcup_{w \in W}  K \pi^{\lv} I w I \pi^{\lv} K = \bigcup_{w \in W} K \pi^{w \lv} U_{\O} \pi^{\mv} K \ee where we have used (\ref{I+:im}) to write $I= U_\O U^-_\pi A_\O$ as well as the dominance condition on $\lv, \mv$ which implies that $\pi^{\lv} U_{\O} \pi^{-\lv} \subset K$ and $\pi^{-\mv} U^-_{\pi} \pi^{\mv} \subset K.$ So if (\ref{sh2:0}) is satisfied, one obtains from (\ref{sh2:1}) that for some $w \in W,$ \be{sh2:2} K \pi^{w \lv} U_{\O} \pi^{\mv} K \cap K \pi^{\xv} K \neq \emptyset. \ee Trivially we thus obtain that $K \pi^{w \lv+ \mv} U \cap K \pi^{\xv} K \neq \emptyset.$ From (\ref{sm1:sm2}), part (2), we thus obtain that $\xv \leq w \lv + \mv \leq \lv + \mv$ where we have $w \lv \leq \lv$ since $\lv$ was assumed dominant. \end{proof}

\spoint Describing the structure of $\chs$ along the lines of Satake (\cite{sat}) is the goal of the remainder of this paper. Recall the construction of Looijenga's coweight algebra $\C_{\leq}[\Lv]$ from \S \ref{sec-aff:inv} \footnote{We may also regard $\C_{\leq}[\Lv]$ as the (semi-infinite)Hecke algebra associated to the datum $(A_{\K}, A_{\O}, \Lv, \leq).$}. We can define an action of the elements $e^{\lv} \in \C_{\leq}[\Lv]$ on $\ve_{\mu} \in \mhs$ via the formula: \be{loo:act} e^{\lv} \circ \ve_{\mv} = q^{ - \langle \rho, \lv \rangle} \ve_{\mv + \lv}. \ee This action extends to give an action of the completion $\C_{\leq}[\Lv]$ on $\mhs,$ and we can easily verify the following result using the Iwasawa decomposition and the definition of the completions involved.

\begin{nlem} \label{loo:conv}  As a $\C_{\leq}[\Lv]$-module, $\mhs$ is free of rank one with generator the \emph{spherical vector} $\mathbf{1}_K := \ve_0,$ i.e., the characteristic function of the subset $A_{\O} U K.$ Furthermore, the action (\ref{loo:act}) is a right $\chs$-module map, i.e., \be{loo:conv:1} e^{\lv} \circ (\ve_{\mv} \star h) = ( e^{\lv} \circ \ve_{\mv} ) \star h, \ee where $h \in \chs$ and $\star$ denotes the convolution action of $\chs$ on $\mhs$ as in (\ref{semiinf:act}). \end{nlem}

\spoint Using Lemma \ref{loo:conv}, we obtain a map, \be{sat:defn} S: H_{\leq}(G_+, K) \rr \C_{\leq}[\Lv] , \, h \mapsto S(h) \ee defined by the expression \be{sat:defn:2} \mathbf{1}_K \star h = S(h) \circ \mathbf{1}_K. \ee Explicitly, if $h_{\lv}$ is as in (\ref{h-k}), then \be{sat:defn:3} S ( h_{\lv}) = \sum_{\mv \in \Lv} | K \setminus K \pi^{\mv} U \cap K \pi^{\lv} K | e^{\mv} q^{ \la \rho,  \mv \ra }. \ee  The above map is called \emph{the Satake homorphism:} it is a homomorphism of algebras since for $h_1, h_2 \in H_{\leq}(G_+, K),$ we have \be{sat:hom} \mathbf{1}_K \star ( h_1 \star h_2 ) &=& (S(h_1) \circ \mathbf{1}_K) \star h_2  =  S(h_2) \circ (S(h_1) \circ \mathbf{1}_K) \\ &=& (S(h_2) \, S(h_1)) \circ \mathbf{1}_K,  \ee  In fact, we have the following analogue of the classical Satake isomorphism,

\begin{nthm} \label{sat:isom} The map $S$ is an isomorphism of algebras, \be{sat:isom:2} S: H_{\leq}(G_+,K) \stackrel{\cong}{\rr} \C_{\leq}[\Lv]^W \ee where $\C_{\leq}[\Lv]^W$ is the ring of $W$-invariant elements of $\C_{\leq}[\Lv].$ \end{nthm}

The proof will be divided into three parts

\tpoint{Proof of Theorem \ref{sat:isom}, Part 1: $W$-invariance of the image of $S$.} The fact that the image of the Satake map lies in $\C_{\leq}[\Lv]^W$ follows as in the classical case from the properties of certain intertwining maps.  Recall that have described the basic properties of spherical intertwiners in \S \ref{intw-k}, and we freely use now same notation introduced there . In particular, we have defined the rings $R_w$ and the spaces $M_w:= R_w \otimes_R M(G, K)$ in (\ref{Rw:Mw}) as well as the normalized intertwiners $\mf{K}_w: M(G, K) \rr M(G, K)_w$ for each $w \in W$ in (\ref{Ka}-\ref{Kw}). We can extend these definitions easily to define the spaces $M_{\leq}(G, K)_w:= R_w \otimes_R M_{\leq}(G, K)$ and corresponding maps $\mf{K}_w: M_{\leq}(G, K) \rr M_{\leq}(G, K)_w.$
Moreover, these maps $\mf{K}_w$ are compatible with the right $H_{\leq}(G_+, K)$ action as follows,

\begin{nprop} \label{Ktable} In the notation above, we have
\begin{enumerate}
\item[(a)] $\mf{K}_w$ is a right $H(G_+, K)$-module maps for any $w \in W,$ i.e., \be{Kw:right} \mf{K}_w (\phi \star h) = \mf{K}_w(\phi) \star h \ee for any $\phi \in \mhs_w$ and $h \in \chs.$
\item[(b)] With respect to the $\circ$ action of $\C_{\leq}[\Lv]$ on $M_{\leq}(G, K)_w$ defined as in (\ref{loo:act}), the maps $\mf{K}_w$ satisfy  \be{K_w:w} \mf{K}_w \circ e^{\lv} = e^{w \lv} \circ \mf{K}_w \text{ for } w \in W, \,  \lv \in \Lv. \ee
\item[(c)] The maps $\mf{K}_w$ fix the spherical vector, i.e., $\mf{K}_w (\mathbf{1}_K) = \mathbf{1}_K$ for any $ w \in W.$ \end{enumerate} \end{nprop}

From the previous Proposition \ref{Ktable}, we find that for any $w \in W,$ \be{sat:W:1}  S(h_{\lv}) \circ \mathbf{1}_K &=& \mathbf{1}_K \star h_{\lv}  =  \mf{K}_w(\mathbf{1}_K) \star h_{\lv}  = \mf{K}_w (\mathbf{1}_K \star h_{\lv}  ) \\ &=& \mf{K}_w ( S(h_{\lv}) \circ \mathbf{1}_K ) = S(h_{\lv})^w \circ \mathbf{1}_K. \ee It follows that $S(h_{\lv})^w= S(h_{\lv})$ for any $w \in W.$

\tpoint{Proof of Theorem \ref{sat:isom} Part 2: Injectivity of $S$} To show that $S$ is injective, we shall verify that \be{sat:inj:1} S( h_{\lv} ) = q^{ \la \rho, \lv \ra  } e^{\lv} + \sum_{\mv \leq \lv} c_{\mv} e^{\mv}, \ee where $\mv < \lv$ means that $\mv$ is strictly less than $\lv$ in the dominance order, and $c_{\mv} \in \zee_{\geq 0}.$ Indeed, this follows from \ref{sm1:sm2}(2) and the following,

\begin{nlem} \label{easy:conv} Let $\lv \in \Lv_+.$ Then \be{easy:conv:1} K \pi^{\lv} U \cap K \pi^{\lv} K = K \pi^{\lv} \ee \end{nlem}
\begin{proof}[Proof of Lemma] Indeed, from decompositions (\ref{bruhat:K}) and (\ref{I+:im}), we obtain \be{easy:conv:2} K \pi^{\lv} K = \bigcup_{w \in W} K \pi^{w \lv} U^-_{\pi} U_{\O}. \ee Assume $K \pi^{\lv} K \cap K \pi^{\lv} U \neq \emptyset.$ Then from (\ref{easy:conv:2}) there exist $w \in W$ such that  $K \pi^{w \lv} U^- \cap K \pi^{\lv} U \neq \emptyset.$ By (\ref{sm1:sm2}) we must have $\lv \leq w \lv.$ But since $\lv \in \Lv_+,$ it follows that $w\lv = \lv$ for any such $w.$  Hence, we conclude that \be{easy:conv:3} K \pi^{\lv} K \cap K \pi^{\lv} U \subset K \pi^{\lv} I  \cap K \pi^{\lv} U \subset K \pi^{\lv} U^-_{\pi} \cap K \pi^{\lv} U  . \ee Let us now show that this latter intersection is $K \pi^{\lv}.$ Indeed, $\pi^{\lv} u \in K \pi^{\lv} U^-$ implies that $\pi^{\lv} u \pi^{ - \lv} \in K U^{-}.$  So by \cite[Proposition 3.3]{bgkp} we have that $\pi^{\lv} u \pi^{ - \lv} \in U_{\O}.$ \end{proof}

\tpoint{Proof of Theorem \ref{sat:isom} Part 3: Surjectivity of $S$} Finally we need to show that the map $S$ is surjective. The proof (in a slightly different context) is essentially contained in \cite[p.25]{loo}.  Let \be{x} x= \sum_{\xv \in \Lv} c_{\xv} e^{\xv} \in \C_{\leq}[\Lv]^W \ee  and let $\Xi \subset \Lv_+$ be a finite set of elements such that if $\lv \in \Supp(x)$ then $\lv \leq \mv$ for some $\mv \in \Xi.$ For any integer $n$ we let $\Upsilon(n) \subset \Lv_+$ be the set of elements of the form $\mv - n_1 a_1^\v - \cdots  - n_{\ell+1} a_{\ell+1}^{\vee}$ where $\mv \in \Xi,$  $n_i \geq 0$ for $i=1, \ldots, \ell+1,$ and $n_1 + \cdots + n_{\ell+1} \geq n.$ A subset $\Sigma \subset \Lv$ is said to be dominated by a subset $\Sigma' \subset \Lv_+$ if for every $\tv \in \Sigma$ there exists $\lv \in \Sigma'$ such that $\tv \leq \lv.$ We now construct a family of elements $h_n \in \chs$ for $n \in \zee_{\geq 0}$ such that: (a) $\Supp(h_n)$ is dominated by $\Xi;$ and (b) $\Supp( x - S(h_n) )$ is dominated by $ \Upsilon(n).$ Indeed, let $h_0=0;$ suppose that \be{h:n} h_n= \sum_{\xv \in \Lv} c_{n, \xv} h_{\xv}, \ee we then define \be{h:n+1} h_{n+1} = h_n + a_n; \text{ where } a_n =  \sum_{\mv \in \Upsilon(n) \setminus  \Upsilon(n+1) }  (c_{\mv} - q^{ - \la \rho, \xv \ra } c_{n, \xv})h_{\xv}. \ee Condition (a) is immediately verified since $ \Upsilon(n)$ is dominated by $\Xi$;  and condition (b) follows from (\ref{sat:inj:1}). Note that we have $h_n = a_1 + \cdots + a_{n-1}.$ Thus, setting $ h_{\infty} := \sum_{n \geq 0} a_n,$ we see from the definition of the completion that $h_{\infty} \in \chs.$ Furthermore, from (b), we see that $ x- S(h_{\infty})$ is supported on $\cap_{n=0}^{\infty}  \Upsilon(n),$ which is the empty set.

\subsection{Explicit Formula for the Satake Isomorphism} \label{sph:section}

We now present an explicit formula (\ref{sph:aff}) in $\C_{\leq}[W]^W$ for the image of $h_{\lv}$ with $\lv \in \Lv_+$ under the Satake isomorphism $S$ of Theorem \ref{sat:isom}.

\spoint Let $v$ be a formal variable, and consider the ring of Taylor series $\C_{v}:= \C[[v]].$ Let $\C_{\leq, v}[\Lv]$ be the ring of collections as in (\ref{f}) where the coefficients are now taken to lie in $\C_{v}$.  Recall that $Q^\v_-$ was the coroot lattice of $\mf{g},$ which has a $\zee$-basis $\{ a_1^\v, \ldots, a_{\ell+1}^\v \}.$ Denote by \be{qn} \mc{Q}_{v}= \C_{v}[[ Q_{-}^{\vee} ]], \ee the completion of the group algebra of $Q_{-}^{\vee},$ i.e., the ring of formal power series in the variables $e^{-a_i^{\vee}}$ with coefficients in $\C_{v},$ where as usual \be{gp:exp} e^{-a_i^{\vee}} e^{-a_j^{\vee}} = e^{-a_i^{\vee}-a_j^{\vee}} \text{ for } i, j=1, \ldots, \ell+1.\ee We have that $\mc{Q}_{v} \subset \C_{\leq, v}[\Lv],$ since for each $x \in \mc{Q}_{v}$ we have that $\Supp(x) \subset \mf{c}(0),$ where $0 \in \Lv_+$ denotes the zero coweight and the notation is as in (\ref{supp}) and (\ref{c:lam}).

\begin{de} \label{v-fin} The elements $\C_{\leq}[v^{-1}][\Lv] \subset \C_{v, \leq}[\Lv]$ are said to be $v$-finite. Similarly, an element of $\mc{Q}_{v}$ is said to be $v$-finite if it is $v$-finite as an element of $\C_{v, \leq}[\Lv],$ i.e., it lies in $\C[v^{-1}][[Q_{-}^{\vee}]].$ \end{de}

\emph{Remark:} Of course not all $f \in \C_{v, \leq}[\Lv]$ may be ''specialized'' at $v^2=q^{-1}$ but this is certainly possible for $v$-finite elements.

\spoint For each $a \in R$, we set \be{c:a} c(a^{\vee}) = \left ( \frac{1 - v^{2} e^{-a^{\vee}}}{1-e^{-a^{\vee}} } \right )  ,\ee  which we regard as an element of $\mc{Q}_{v}$ by formally expanding the rational function as a series in $e^{-a^{\vee}}.$ Letting $m(a^{\vee})$ denote the multiplicity of the coroot $a^{\vee},$ we now set \be{delta} \Delta := \prod_{a \in R_{+}} c(a^{\vee})^{m(a^\vee)} =  \prod_{a \in R_{re,+}} c(a^{\vee}) \prod_{n \in \zee, n \neq 0 } c(n \cc )^{\ell},
\ee which is again regarded as an element of $\mc{Q}_v$ by formally expanding the above series. In fact, it is easy to see (see \cite{mac:aff}) that the element $\Delta$ is invertible in $\mc{Q}_v.$

For $a \in R$ and $w \in W$, we shall write $c(w a^{\v})$ to denote the expansion of (\ref{c:a}) in $\qn.$ For example, if $a$ is a real root and $w =w_a$ then \be{c:-a} c(w_a a ) = c(-a) =  \frac{1 - v^{2} e^{a^{\vee}}}{1-e^{a^{\vee}} }= \frac{e^{-a^\vee} - v^{2} }{e^{-a^\vee}-1} = 1 + (v^2 - 1) e^{-a^\vee} + (1-v^2) e^{-2 a^\vee} + \cdots . \ee Now, for each $w \in W$, we may consider the element \be{Del:w} \Delta^w:= \prod_{a \in R_+} c(w a^{\vee})^{m(a^\vee)}  \in \mc{Q}_v. \ee The same argument as in \cite[p.199]{mac:formal} shows that the sum $\sum_{w \in W} \Delta^w$ is a well-defined element in $\mc{Q}_{v}$ which again is invertible.

For any subset $\Sigma \subset W$ we write its Poincare polynomial as \be{poin:poly} \Sigma(v) = \sum_{w \in \Sigma} v^{\ell(w)}. \ee Consider now the following expression, which takes its values in $\qn,$ \be{def:corr}  H_0:= \frac{\sum_{w \in W} \Delta^w}{W(v^2)}.\ee In fact, it is not hard to see (\cite[(3.8)]{mac:formal}) that $H_0 \in \C_{v}[[e^{-\cc}]]$ where $\cc$ is the minimal imaginary coroot of $\mf{g}.$ One may give a formula for $H_0$ as an infinite product of expressions in the variables $v^2$ and $e^{\cc}$ (see \cite[Theorem 1.7]{bfk}). In the case when $\mf{g}$ is simply-laced, the formula takes the following form, which is the result of \cite[(3.8)]{mac:formal} and the work of Cherednik \cite{cher:ct} on Macdonald's Constant Term Conjecture.

\begin{nthm} \label{mac:cher:fla}
Let $\mf{g}$ be a simply-laced untwisted affine Kac-Moody algebra. Then we have that  \begin{eqnarray} \label{corr} H_0 = \prod_{j=1}^\ell \prod_{i=1}^{\infty} \frac{1 - v^{2 m_j} e^{-i \cc} }{1 - v^{2(m_j+1)}e^{-i \cc}}, \end{eqnarray} where the integers $m_j$ for $j=1, \ldots, \ell$ are the exponents of $\mf{g}_o$ defined by the relation that \be{poin:fin} W_o(v^2) = \prod_{i=1}^{\ell} \frac{1 - v^{2(m_i+1)}}{1-v^2} \ee where $W_o \subset W$ is the finite Weyl group.  In particular, we see that $H_0$ is $v$-finite (and so can be evaluated at $v^2 = q^{-1}$).
\end{nthm}

\tpoint{Statement of Main Formula} We would now like to state precisely the formula for $S(h_{\lv})$ with $\lv \in \Lv_+.$ For any such $\lv \in \Lv_+,$ define \be{stab} W_{\lv}= \{ w \in W | w \lv = \lv \}  \subset W. \ee

\begin{nthm} \label{sph:mac} Let $\lambda \in \Lambda_+^{\vee}.$ The ratio $\frac{H_{\lv}}{H_0} \in \C_{v, \leq}[\Lv]$ is $v$-finite, and its value at at $v^2=q^{-1}$ is equal to $S(h_{\lv}).$ \end{nthm}

Informally, we shall write the above theorem as follows: \be{sph:aff} S( h_{\lv}) = \frac{1}{H_0} \frac{q^{-\langle \rho, \lv \rangle }}{W_{\lv}(q^{-1})}  \sum_{w \in W} \Delta^w e^{w \lv}, \ee where $\Delta$ is as in (\ref{delta}) and $H_0$ is defined in (\ref{def:corr}). The proof will occupy \S \ref{sph:section-proof} and will be broken up into four parts.

\begin{description}

\item[Step 1: Disassembly] We first break up the computation of $S(h_{\lv})$ into pieces indexed by a set of minimal length representatives in $W/W_{\lv}.$ This is done via a passage to the Iwahori subgroup $I^-$.

\item[Step 2: Recursion] We show a certain recursion relation between the pieces of the previous step corresponding to $w$ and $w'$ where $w$ and $w'$ differ by a simple reflection.

\item[Step 3: Algebraic Identities] We recall some purely formal algerbaic identities from \cite{cher:ma} involving the affine symmetrizers and polynomial representation of Cherednik.

\item[Step 4: Rephrasal and Reassembly] Reinterpreting the recursion of Step (2) using the polynomial representation of Step (3), we can rewrite the disassembly from Step (1) using the affine symmetrizers of Cherednik. The argument is then concluded by applying the an algebraic  proportionality principle from Step (3).

\end{description}


\subsection{Proof of Theorem \ref{sph:mac}} \label{sph:section-proof}

\subsection*{Step 1: Disassembly} \label{disass}

\spoint Recall that we have defined the stabilizer $W_{\lv}$ for each $\lv \in \Lv$ in (\ref{stab}). If $\lv = 0$, then $W_{\lv} = W.$ In the opposite extreme, if $W_{\lv}=1$ we say that $\lv$ is \emph{regular.} By a set of \emph{minimal length representatives} for the set $W/ W_{\lv}$ we shall mean a set of coset representatives $W^{\lv} \subset W$  for $W/W_{\lv}$ such that \be{stab-cond} \ell(w w_{\lv}) = \ell(w) + \ell(w_{\lv}) \text{ for } w \in W^{\lv}, w_{\lv} \in W_{\lv}. \ee

Recall that $I^-$ was the opposite Iwahori subgroup, which is equipped with a decomposition $I^- = U^-_{\O} U_{\pi} A_{\O}$ by (\ref{I^-:im}). Let us define for $w \in W$ and $\lv, \mv \in \Lv$ with $\lv \in \Lv_+$ the following two sets, \begin{eqnarray} \ifib &:=& \{ (a, c) \in U I^- w I^-  \times_{I^-} I^- \pi^{\lv} K | a c = \pi^{\mv} \} \\
\kfib &:=& \{ (x, y) \in U K \times_K K \pi^{\lv} K | x y = \pi^{\mv} \} \end{eqnarray} Equivalently if we consider the multiplication maps \be{mmaps} m_{w, \lv} &:& U\im w \im \times_{\im} \im \pi^{\lv} K \rr G \\ m_{\lv} &:& U K \times_K K \pi^{\lv} K \rr G ,\ee then we have that $\ifib=m_{w, \lv}^{-1}(\pi^{\mv})$ and $\kfib=m_{\lv}^{-1}(\pi^{\mv}).$

Consider the map $\varphi_w: F^{I^-}_{w, \lv}(\mv) \rr F^K_{\lv, \mv}$ which sends $(a, c) \mapsto (a, c)$ where $a \in U \im w \im$ and $c \in \im \pi^{\lv} K.$ It clearly induces a well-defined map: i.e., if $ j \in \im$ then $(a, c)$ and $(aj, j^{-1}c)$ are equivalent in $F^K_{\lv}(\mv).$ Let us now set, \be{f-lv} F^{I^-}_{\lv}(\mv) :=  \sqcup_{w \in W^{\lv}} F^{I^-}_{w, \lv}( \mv), \ee and define $\varphi$ to be the map from $F^{I^-}_{\lv}(\mv) \rr F^K_{\lv}(\mv)$ which restricts to $\varphi_w$ on $F^{I^-}_{w, \lv}( \mv).$

\begin{nlem}In the notation above, the map $\varphi: F^{I^-}_{\lv}(\mv) \rr F^K_{\lv}(\mv)$ is bijective.  \end{nlem}

\begin{proof}  Let us first see that $\varphi$ is surjective: a given $(x, y) \in F^K_{\lv, \mv}$ has representative of the form $x \in UK$ and $y \in \pi^{\lv}K.$ Write $x = n k$ with $n \in U$ and $k\in K,$ and further decompose $k=i_1 w i_2$ with $i_{1} , i_2 \in \im.$ Then $x= n i_1 w i_2,$ and surjectivity clearly follows.

To show injectivity, we begin with the following result.

\begin{nclaim} \label{inj:stab} Let $(a, c), (a',c') \in \sqcup_{w \in W^{\lv}} \ifib,$ and suppose that $\varphi(a, c) = \varphi(a', c').$ Then $v= v c'$ with $v \in \im w \im$ and $w \in W_{\lv}.$ \end{nclaim}

\begin{proof}[Proof of Claim] Given $(a, c)$ and $(a', c')$ as above such that $\varphi(a, c) = \varphi(a', c'),$ there must exist by definition an element $v \in K$ such that  $c = v c'.$ Suppose $c = i \pi^{\lv}k$ and $c'= i' \pi^{\lv} k'.$ Then we have a relation of the form \be{inj:stab:1} i \pi^{\lv} k = v i' \pi^{\lv} k. \ee We may suppose that $v \in \im w \im$ with $w \in W.$ Thus, we have found an element in the intersection of  \be{inj:stab:2} \im \pi^{\lv} K \cap I^- w I^- \pi^{\lv} K. \ee On the other hand, since $\lv$ is dominant, we have that \be{inj:stab:3} I^- w I^- \pi^{\lv} K  = I^- w \pi^{\lv} K .\ee However, this is only possible if $w  \in W_{\lv}:$ indeed, if the intersection was non-empty, then we would actually have that \be{inj:stab:4} I^- w \pi^{\lv} K = I^- \pi^{\lv} K. \ee Consider the map multiplication maps \begin{eqnarray} m: UI^- \times_{I^-} I^- \pi^{\lv} K &\rr& G \\
m: UI^- \times_{I^-} I^- w \pi^{\lv} K &\rr& G. \end{eqnarray} The image of the former is supported on $U \pi^{\lv} K$ whereas the image of the latter contains the element $\pi^{w \lv}.$ This is a contradiction unless $w \in W_{\lv}.$  \end{proof}

We return to the injectivity of $\varphi:$ we need to show two things, \begin{description} \item[(i)] the images of $\varphi_{w'}$ and $\varphi_{w'}$ are disjoint for distinct element $w, w' \in W^{\lv};$ and \item[(ii)] restricted to each $w \in W^{\lv}$ the map $\varphi_w$ is injective. \end{description}

First, we turn to (i): from Claim \ref{inj:stab}, if $\varphi(a, c) = \varphi(a', c')$  with $(a, c)$ and $(a', c')$ chosen to be elements in $F^{I^-}_{\lv}(\mv),$ we have $c= v c'$ with $v \in \im w_{\lv} \im$ where $w_{\lv} \in W_{\lv}.$ Since $\varphi(a, c) = \varphi(a', c')$ we must also have \be{ab:v} a v^{-1} = a', \ \ v \in \im W_{\lv} \im. \ee Assume now that $a \in U w_1 \im$ and $a' \in \im w_2 \im,$ with $w_1, w_2 \in W^{\lv}.$  The relation (\ref{ab:v}) will then lead to a contradiction of the minimality of the coset representatives $w_1, w_2$  unless $w_{\lv}=1.$ So in fact $ v \in \im$ and the two elements $(a, c)$ and $(a', c')$ were in fact both in $F_{w, \mv}(\lv)$ for some fixed $w \in W^{\lv}.$

Next, we turn to (ii): if $(a, c), (a', c') \in \ifib$ have the same image by $\varphi_w$, then by the analysis of the previous paragraph, we must have that $c' \in \im c.$ Hence $(a, c)$ and $(a', c')$ are actually equivalent in $F_{w, \mv}(\lv)$ if they have the same image under $\varphi_w.$

\end{proof}

\spoint To summarize, we have shown that if we define, for each $w \in W^{\lv}$ the sum \be{Jw} J_w(\lv) := \sum_{\mv \in \Lv} | \ifib | q^{ \langle \rho, \mv \rangle} e^{\mv}, \ee then we have \be{S:w} S(h_{\lv}) = \sum_{w \in W^{\lv}} J_w(\lv), \ee which is the required disassembly of $S(h_{\lv}).$

\subsection*{Step 2: Recursion}

\spoint Now we verify a recursion relation for $J_w(\lv)$ introduced above in (\ref{Jw}). The result is an analogue of the recursion relations obtained by Macdonald in \cite[Theorem 4.4.5]{mac:mad}.
\newcommand{\jb}{J^{\flat}}

We begin by reinterpreting the elements $J_{w}(\lv)$ as follows. Recall first the construction of the algebra $H^-_W$ and the module $M(G, I^-)$ from \S \ref{basic-opp}. In particular, we have defined elements $T^-_w$ for $w \in W$ and $\ve^-_x$ for $x \in \aw$ there.  Next, note that the set of $(I^-, K)$-double cosets of $G$ are parametrized by $\Lv.$ Let $M(I^-, K)$ denote the set of functions on $G$ which are left $I^-$-invariant and right $K$-invariant, and which have finite support. Letting $\theta_{\lv, K}$ be the characteristic function of $I^- \pi^{\lv} K,$ we know that $M(I^-, K)$ has a basis $\{ \theta_{\lv, K} \}$ as $\lv$ ranges over  $\Lv.$ Using the map \be{conv:min} m_{x, \lv}: \ve^-_x \times_I \theta_{\lv, K} \rr G, \ee we define convolution as \be{conv:min:2} \ve^-_x \star \theta_{\lv, K} = \sum_{\mv \in \Lv}| m_{x, \lv}^{-1} (\pi^{\mv}) | \theta_{\mv, K}. \ee In the case that $x \in W \subset \mc{W}$ the same argument as in the proof of parts (1) and (2) of the Proposition below shows that the right hand side of (\ref{conv:min:2}) actually lies in $M(G, K),$ the space of finitely supported functions on $A_{\O}U \setminus G / K.$ Now, by the definition of the various convolutions involved, we can immediately see that the following relations hold in $M(G, K):$ for $w \in W^{\lv}$ we have, \be{Jw:conv} J_w(\lv) \circ \mathbf{1}_K  &=& \ve^-_w \star \theta_{\lambda, K} \\
&=& q^{ - \ell(w)} \ve^-_1 \star T^-_w \star \theta_{\lv, K} . \ee Let us define now for \emph{any} $w \in W$ the elements $\jb_w(\lv)$ by the relation \be{Jbis} J^{\flat}_w(\lv) \circ \mathbf{1}_K =  \ve_1 \star T^-_w \star \theta_{\lv, K}. \ee For $w \in W^{\lv},$ we then have \be{Jbis:J} J_w(\lv) = q^{ - \ell(w) } \jb_w(\lv). \ee We shall use the right hand side of (\ref{Jbis:J}) as the definition of $J_w(\lv)$ for any $w \in W$ (not just $w \in W^{\lv}$ as in (\ref{Jw})). Having thus defined $J_w(\lv)$ for each $w \in W$ and $\lv \in \Lv$ we can now state,

\begin{nprop} \label{recursion}  Let $\lv \in \Lv_+, w \in W.$
\begin{enumerate}
\item The expression $J_w(\lv)$ actually lies in $ \C[\Lv].$ In other words, \be{Jw:sum} J_w(\lv) = \sum_{\mv \in \Lv} \Phi_{w, \mv} e^{\mv} \ee with $\Phi_{w, \mv} \in \mathbb{N}$ non-zero for only finitely many $\mv \in \Lv.$ (Note: $\Phi_{w, \mv}$ depend on $q,$ and so this finite set may in fact depend on $q$).

\item If $w \in W_{\lv}$, then we have \be{J:fix} J_w(\lv) = q^{ \la \rho, \lv \ra } q^{-\ell(w) } e^{\lv}= q^{ -\ell(w)} J_1(\lv). \ee

\item Let $w \in W$ and choose $a \in \Pi$ such that $w'=w_a  w$ satisfies $\ell(w') < \ell(w).$ Then we may write
 \be{J:w:w'}  J_w(\lv) = \frac{1-q^{-1}e^{a^\v}}{1- e^{a^\v}} \cdot J_{w'}(\lv)^{w_a} + \frac{q^{-1}-1}{1-e^{a^\v} } \cdot J_{w'}(\lv) , \ee where by $J_{w'}(\lv)^{w_a}$ we mean the termwise application of $w_a$ to the expression $J_{w'}(\lv),$ and the rational functions which appear on the right hand side are expanded in the ring $R_{w_a}$ of (\ref{Rw:Mw}).

\item There exist polynomials \footnote{In Proposition \ref{alg:padic} we connect the polynomials $\Phi_{w, \mv}(v^2)$ with the polynomial representation of the DAHA. }  $\Phi_{w, \mv}(v^2) \in \C[v^2]$ such that \be{j-univ-1} J_w(\lv) = \sum_{\mv} \Phi_{w, \mv}(q^{-1}) e^{\mv} \ee for any power of a prime $q.$

\end{enumerate}
 \end{nprop}

\tpoint{Remarks on Part (3)} \label{example-remark} From the way in which statement (3) is written, it is not clear that the expression $J_w(\lv)$ is actually a finite sum. However, there is a cancellation which occurs when one expands the rational functions appearing in (\ref{J:w:w'}). Let us illustrate this with a simple example: suppose $\lv \in \Lv_+$ and $a \in \Pi$ is such that $w_a \notin W_{\lv}.$ Then we may write $w_a \lv = \lv - k a^\vee$ for $k:= \la \lv, a \ra > 0.$ From (2) we find that $J_1(\lv) = q^{ \la \rho, \lv \ra} e^{\lv}.$ Hence, from (3) we find that \be{example-1} J_{w_a}(\lv) &=& \frac{1-q^{-1}e^{a^\v}}{1- e^{a^\v}} \cdot q^{\la \rho, \lv \ra} e^{w_a \lv}  + \frac{q^{-1}-1}{1-e^{a^\v} } \cdot q^{\la \rho, \lv \ra} e^{\lv}\\
&=& q^{ \la \rho, \lv \ra } \{ ( 1 + (1-q^{-1}) e^{\av}+ (1- q^{-1})e^{2 \av} + \cdots ) e^{ \lv - k \av} \\ &+& (  (q^{-1} -1) + (q^{-1} - 1)e^{\av} + \cdots ) e^{\lv} \}  \\
&=& q^{ \la \rho, \lv \ra } ( e^{w_a \lv} + (1-q^{-1}) e^{w_a \lv + \av} + \cdots + (1-q^{-1}) e^{\lv - \av} ) .\ee Hence, one can see directly in this this example that $J_{w_a}(\lv)$ actually lies in $\C[\Lv].$ Moreover, note that the sum $J_1(\lv) + J_{w_a}(\lv)$ is easily seen to agree with the usual spherical function for $SL(2).$

\tpoint{Proof of Proposition \ref{recursion}}  The proofs of (1) and (2) are quite straightforward, and we just sketch the argument here.
From the definition of $J_w(\lv)$ the support (i.e., the set of $\mv \in \Lv$ such that $F_{w, \lv}(\mv) \neq \emptyset$) consists of those $\mv$ such that \be{supp:Jw} U I^- w I^- \pi^{\lv} K \cap U \pi^{\mv} K \neq \emptyset. \ee
 However, using the decomposition $I^- = U^-_{\O} U_\pi A_{\O}$ together with the dominance of $\lv \in \Lv_+$ we can immediately deduce that \be{supp:Jw-2} U I^- w I^- \pi^{\lv} K \cap U \pi^{\mv} K  \neq \emptyset \iff U_{-w, \O} w \pi^{\lv} K \cap U \pi^{\mv} K \neq \emptyset \ee where $U_{-w, \O}:= \{ u \in U^-_{\O} \mid w^{-1} u w \in U \}$ is a finite-dimensional group. From the finite-dimensionality of $U_{-w, \O}$ we may deduce that only finitely many $\mv$ can occur within the support of $J_w(\lv).$ Moreover, from  (\ref{supp:Jw-2}), we see that if $w \in W_{\lv}$ the support of $J_w(\lv)$ consists of exactly one element. We leave the rest of (2) to the reader,

As for (3), it suffices in lieu of (\ref{Jbis:J}) to show the following recursion for the elements $\jb(\lv)$ introduced in (\ref{Jbis}): for every $w = w_a w'$ with $\ell(w) = 1 + \ell(w')$ and $a \in \Pi$ we have,\be{Jbis:w:w'}  \jb_w(\lv) = \frac{q-e^{a^\v}}{1- e^{a^\v}} \cdot \jb_{w'}(\lv)^{w_a} + \frac{1-q}{1-e^{a^\v} } \cdot \jb_{w'}(\lv). \ee To demonstrate this, we shall use the properties of the intertwining operators $\mf{T}_{w}^-$ introduced in \S \ref{int-neg}. For $a \in \Pi$ we have defined the operators $\mf{T}_a^-: M(G, I^-) \rr M(G, I^-)_{w_a},$ and from Lemma (\ref{I^-:rk1}) we have that \be{recur:1} \mf{T}^-_a(\ve_1) = \ve^-_{w_a} + \frac{1 - q^{-1}}{1- e^{\av} } \circ \ve_1. \ee We may also extend $\mf{T}_a^-$ to functions which are right $K$-invariant and left $A_{\O}U$-invariant using the same formula (\ref{Iw}), and we shall continue to denote this map by $\mf{T}_a^-.$ One can easily verify the following properties, \be{Imin:prop} \mf{T}_{a}^-(\mathbf{1}_K) &=& \mf{T}_a(\mathbf{1}_K) = \frac{1 - q^{-1} e^{\av}}{1- e^{\av}} \circ \mathbf{1}_K \\ \label{imin:K}
\mf{T}_a^-(\ve_{w} \star \theta_{\lv, K}) &=& \mf{T}_a^-(\ve_w) \star \theta_{\lv, K} \text{ for } w \in W, \lv \in \Lv \ee where recall that $\theta_{\lv, K}$ is the characteristic function of $I^- \pi^{\lv} K.$

We now compute the action of $\mf{T}_a^-(\jb_w(\lv) \circ \mathbf{1}_K )$ in two different ways.  On the one hand, \begin{eqnarray} \label{mac:recur1} \mf{T}^-_{a} (\jb_{w'}(\lv) \circ \mathbf{1}_K ) = \mf{T}^-_{a} (\ve^-_{1} \star T^-_{w'} \star \theta_{\lambda, K}) &=& \mf{T}^-_{a} (\ve^-_1) \star T^-_{w'} \star \theta_{\lambda, K} \ee using the fact that $\mf{T}_a^-$ is a right $H^-_W$-map. Using (\ref{recur:1}) we find that the previous expression is equal to
\be{mac:recur:2}  (\ve^-_{w_a} + \frac{1-q^{-1}}{1 - e^{a^{\vee}}} \circ \ve^-_1)  \star T^-_{w'} \star \theta_{\lambda, K}
&=& (q^{-1}  \ve^-_{1} \star T^-_{w_a} + \frac{1-q^{-1}}{1 - e^{a^{\vee}}} \circ \ve^-_1)  \star T^-_{w'} \star \theta_{\lambda, K} \\
&=& q^{-1} \jb_{w}(\lv) \circ \mathbf{1}_K + \frac{ (1- q^{-1})}{1- e^{a^\v}} \jb_{w'}(\lv) \circ \mathbf{1}_K. \end{eqnarray}

On the other hand, from (\ref{inter:2}) we have \be{mac:recur:2} \mf{T}^-_{a} (\jb_{w'}(\lv) \circ \mathbf{1}_K ) = (\jb_{w'})^{w_a} \mf{T}_a^-( \mathbf{1}_K) = \frac{1-q^{-1} e^{a^{\v}}}{1- e^{a^{\v}}} \jb_{w'}(\lv)^{w_a} \circ \mathbf{1}_K, \ee where we have used (\ref{imin:K}) to compute the action of $\mf{T}_a^-$ on $\mathbf{1}_K.$ Part (3) of the Proposition follows.

As for Part (4): if $w=w_a$ this follows from the remarks \S \ref{example-remark}. The general case follows by induction on the length of $w$ using the recursion relation \ref{J:w:w'}.

\renewcommand{\c}{\mathbf{c}}
\renewcommand{\b}{\mathbf{b}}

\subsection*{Step 3: Algebraic Identities}

\tpoint{The Polynomial Representation of Cherednik} In this part, we shall establish a purely algebraic identity (\ref{macdonald:operator}) which will be used in the subsequent part to establish the formula (\ref{sph:aff}). Let $v$ and $X$ formal variables, and set \be{c:b:X} \begin{array}{lcr} \c(X)= \frac{vX-v^{-1}}{X-1} \text{ and } \b(X) = \frac{v -v^{-1}}{1- X}. \end{array} \ee The following identities are easy to verify, see \cite[p. 58, 4.2.3(i) -(iv)]{mac:aff} \be{c:b:iden} \c(X) &=& v - \b(X) = v^{-1} + \b(X^{-1}) \\ \label{c:b:2} \c(X) + \c(X^{-1}) &=& v + v^{-1} \\ \label{c:b:3} \c(X) \c(X^{-1}) &=&1 + \b(X) \b(X^{-1}) . \ee

Recall the ring $\mc{Q}_{v}$ from (\ref{qn}) anove. For each $a \in R$ we may consider the elements on $\qn,$ $c(a):= \c(e^{\av})$ and $b(a):= \b(e^{\av})$ obtained by substituting $e^{\av}$ for $X$ in (\ref{c:b:X}) and then formally expanding the fraction as an element in $\mc{Q}_v$.
\newcommand{\T}{\wt{T}}

Denote by $\qn[W]$ the group algebra of $W$ over the ring $\qn,$ where a typical element $f \in \qn[W]$ is written as $f = \sum_{i=1}^n c_i [\sigma_i]$ where $c_i \in \qn$ and $\sigma_i \in W.$ For each $a \in \Pi$, we consider elements \be{Ta} \T_a = v [w_a] + b(a)(1-[w_a]) = c(a) [w_a] + b(a) [1] \in \qn[W], \ee where $[1]$ denotes the element corresponding to the identity of $W.$ One checks immediately that \be{vTa} v\T_a = \frac{1 - v^2 e^{\av}}{1- e^{\av}} [w_a] + \frac{v^2 -1}{1-e^{\av}} [1]  \ee For $w \in W$ with a reduced decomposition $w= w_{a_1} \cdots w_{a_n}$ where $a_1, \ldots, a_n \in \Pi$ we set \be{Tw} \T_w:= \T_{a_1} \cdots \T_{a_n}. \ee We view $\T_w \in \qn[W]$ in the natural way: first substitute for each $\T_{a_i}$ the expression $c(a_i)[w_{a_i}] + b(a_i)[1]$ from (\ref{Ta}), move all the rational functions to the left so that we may write \be{tw:rat} \T_w = \sum_{\sigma \in W} A_{\sigma}(w)[\sigma], \ee where $A_\sigma(w)$ is some sum of products of rational functions in $c(\cdot)$ and $b(\cdot).$ It is easy to verify that $A_\sigma(w) = 0$ unless $\sigma \leq w$ in the Bruhat order. After now expanding the rational functions involved in $A_\sigma(w)$ in $\qn,$ we may view $\T_w \in \qn[W].$ One may further verify that the above definition of $\T_w$ does not depend on the reduced decomposition of $w.$ Moreover, one can check the following Hecke identity using (\ref{c:b:iden}, \ref{c:b:2}, \ref{c:b:3}), \begin{eqnarray} \label{hecke}  (\T_a + v^{-1})(\T_a - v) &=& 0 \text{ for } a \in \Pi. \ee We may summarize the above by writing, \be{hecke:2}  \T_w \T_a &=& \T_{w w_a} \text{ if } \ell(ww_a) > \ell(w) \\ \label{hecke:3} \T_w \T_{w_a} &=& \T_{w w_a} + (v-v^{-1}) \T_w \text{ if } \ell(w w_a) < \ell(w). \end{eqnarray}

\tpoint{The expression $\mc{P}_{v}$}

 In order to consider certain infinite sums of elements from $\qn[W]$ we introduce the formal dual $\qn[W]^\v$ as the set of all possibly formal infinite sums $F=\sum_{w \in W} f_w [w]$ where $f_w \in \qn.$

Consider now the formal expression \be{P:nu} \mc{P}_{v} := \sum_{w \in W} v^{\ell(w)} \T_w, \ee and let us now argue that it lies in $\qn[W]^\v.$ Using (\ref{tw:rat}), we write \be{Pv:r} \mc{P}_v = \sum_{w \in W} v^{\ell(w)} \T_w = \sum_{w \in W} \sum_{\sigma \leq w}  A_\sigma(w) [\sigma] = \sum_{\tau \in W} C_{\tau} [\tau] ,\ee where $C_\tau:= \sum_{w \in W} A_\tau(w).$ We may then attempt to expand each coefficient $C_{\tau}$ (which is an infinite sum) in the domain $\qn.$ It is \emph{not} the case that, for a fixed $\tau \in W$ only finitely many $A_\tau(w)$ will be non-zero. However, the following Lemma of Cherednik ensures that $C_{\tau}$ is not just well-defined as an element on $\qn$ but that is a $v$-finite quantity.

\begin{nlem} \cite[Lemma 2.19(e)]{cher:ma} \label{P:defn} The element $\mc{P}_{v}$ lies in $\qn[W]^\v$ and has $v$-finite coefficients.  In other words, for each $\tau \in W$, the expansion of the sum $C_\tau:= \sum_{w \in W} v^{\ell(w)} A_\tau(w)$ in is a well-defined, $v$-finite element in $\qn.$ \end{nlem}

We refer the reader to \cite{cher:ma} for the proof, but just remark here that the essence of the argument is to show the following: for any fixed $\tau \in W,$  the contributions $A_\tau(w)$ will all arise with a factor $e^{n \cc}$ with $n \rr \infty$ as $\ell(w) \rr \infty.$ Hence, if one fixes a value $e^{b^{\vee}}$ with $b^{\vee} \in Q^{\vee}$ and $\tau \in W,$ there are only finitely many $w \in W$ such that $e^{b^{\vee}}$ will occur in $A_\tau(w).$

\tpoint{Cherednik's Identity} For each $w \in W,$ we have the element $\Delta^w \in \qn$ as defined in (\ref{Del:w}). Thus we have the element $\sum_{w \in W} \Delta^w [w] \in \qn[W]^\v.$ We then have the following algebraic identity due to Cherednik in $\qn[W]^\v$ (see also (\cite[\S5.5]{mac:aff} for a finite-dimensional analogue).

\begin{nprop}[\cite{cher:ma}] \label{proportional} As elements of $\qn[W]^\v$ we have an equality, \begin{eqnarray} \label{macdonald:operator} \mc{P}_v = \sum_{w \in W} v^{\ell(w)} \T_w = \mf{m} \sum_{w \in W} \Delta^w [w], \end{eqnarray} where $\mf{m} \in \qn$ is some $W$-invariant factor.
\end{nprop} \begin{proof} \emph{Step 1:} First we would like to explain how to make sense, for any $a \in \Pi$, of the expressions $\T_a \, \mc{P}_{v}$ and $\mc{P}_{v} \, \T_a$ as elements in $\qn[W]^\v.$ To make sense of $\T_a \, \mc{P}_{v},$ we may proceed in two equivalent ways: (a) we compute $\T_a \mc{P}_v$ using the Hecke relations (\ref{hecke:2}, \ref{hecke:3}); (b) we can write $\T_a = c(a)[w_a] + b(a)[1]$ as in (\ref{Ta}), and $\mc{P}_{v} = \sum_{\tau \in W} C_{\tau} [\tau]$ as in (\ref{Pv:r}), with the $C_{\tau}$ some (infinite) sum of rational functions. Then $\T_a \, \mc{P}_{v}$ is defined to be the expansion in $\qn[W]^\v$ of \be{TaP} \sum_{\tau \in W} c(a) C_{\tau}^{w_a} [ w_a \tau] + b(a) C_{\tau}[\tau], \ee where $C_\tau^{w_a}$ is the application of $w_a$ to $C_{\tau}.$ It is easy to see that the procedure (a) gives the following relation (see \cite[5.5.9, p. 113]{mac:aff}),  \be{P:leftdiv} \T_a \, \mc{P}_v&=&v \mc{P}_{v}, \ee which shows that the expansion of (\ref{TaP}) is well-defined (one could also proceed as in the proof of Lemma \ref{P:defn} to show that this expansion was well-defined).

To define $\mc{P}_{v} \, \T_a$ we proceed similarly. Either we can using the Hecke relations (\ref{hecke:2} \ref{hecke:3}), or we may expand $\T_a$ as in (\ref{Ta}) and $\mc{P}_{v}$ as in (\ref{Pv:r}). Then we define $\mc{P}_{v} \, \T_a$ as the expansion in $\qn[W]^\vee$ of the expression (which is well-defined using Lemma \ref{P:defn}) \be{PTa} \sum_{\tau \in W} C_{\tau} \, c(\tau a) [ \tau w_a] + C_{\tau} \,  b(\tau a) [\tau]. \ee One may check, using the same argument as in \emph{op. cit} that \be{P:rightdiv} \mc{P}_{v} \, \T_a &=&  v\mc{P}_v. \ee The remainder of the proof is drawing conclusions between, on the one hand (\ref{P:rightdiv}) and (\ref{PTa}), and on the other (\ref{P:leftdiv}) and (\ref{TaP}).
%

\emph{Step 2:} First, let us say that an element of $\qn[W]$ is $W$-invariant if it is of the form
  $$
  \sum f^w [w]
  $$
  for some $f\in \qn$. We want to show that $\mc{P}_{v}$ is $W$-invariant. Indeed, combining (\ref{P:leftdiv}) and (\ref{TaP}) we have the following equality in $\qn[W]^\vee$, \be{Ta:P} \sum_{\tau \in W} c(a) C_{\tau}^{w_a} [ w_a \tau] + \sum_{\tau \in W} b(a) C_{\tau}[\tau] = \sum_{\tau \in W} v C_{\tau}[\tau]. \ee From (\ref{c:b:iden}), we have that $c(a) = v - b(a),$ and hence we may conclude that \be{Ta:P-1}  \sum_{\tau \in W} c(a) C_{\tau}^{w_a} [ w_a \tau] = \sum_{\tau \in W} c(a) C_{\tau}[\tau], \ee and hence \be{Ta:P-2}  \sum_{\tau \in W} C_{\tau}^{w_a} [ w_a \tau] = \sum_{\tau \in W} C_{\tau}[\tau]. \ee Letting $\Gamma = C_1$ (regarded as an infinite sum of rational functions) we conclude that $C_{\tau} = \Gamma^\tau$ for any $\tau \in W.$ We write this formally as \be{P:Gam} \mc{P}_{v} = \sum_{w \in W} [w] \, \Gamma, \ee bearing in mind as usual the above is to be interpreted as $ \sum_{w \in W} \Gamma^w [w] $ where each $\Gamma^w$ is expanded as an expression in $\qn[W]^\v.$

\emph{Step 3:} Now we further apply the conditions (\ref{P:rightdiv}) via (\ref{PTa}).  From the definition of $\T_a$ we have an equality (in $\qn[W]$) $v\T_a + 1 = ([w_a]+1 ) v c(-a).$ Hence, using (\ref{P:Gam}) we may write \be{rightdiv:2}
\mc{P}_{v} \, (v\T_a +1) &=& (\sum_w [w] \, \Gamma ) \, (v\T_a +1) =  (\sum_{w\in W} [w] \, \Gamma ) \, ([w_a]+1 ) \, v c(-a) \\
 &=& (\sum_{w\in W} [w] \, \Gamma) \, v c(-a) + (\sum_{w \in W} [w] \, \Gamma) \, [w_a] \, v c(-a) \\
 &=& \label{rightdiv:2b} (\sum_{w \in W} [w] \, \Gamma) \, v c(-a) + (\sum_w [w] [w_a] \, \Gamma^{w_a})  \, v c(-a) \ee On the other hand, \be{rightdiv:3} \mc{P}_{v}(v^2+1) &=& (\sum_{w \in W} [w] \, \Gamma )  (v^2+ 1 ) \\
 &=& (\sum_{w \in W} [w] \, \Gamma)  \, v (c(-a) + c(a) ) \\
 &=&  (\sum_{w \in W}  [w] \, \Gamma) \, v c(-a) + (\sum_{w \in W} [w] \, \Gamma) \, v c(a) \\
 &=& \label{rightdiv:3b}( \sum_{w \in W} [w] \, \Gamma ) \, v c(-a) + (\sum_{w \in W} [w] [w_a] \, \Gamma ) \, v c(a). \ee From (\ref{P:rightdiv}), we see that (\ref{rightdiv:2b}) is equal to (\ref{rightdiv:3b}), and so we conclude that \be{rightdiv:3} \Gamma^{w_a}  c(-a) = \Gamma c(a) \in \qn. \ee On the other hand, from the definition (\ref{delta}), we have that \be{rightdiv:4} \frac{\Delta^{w_a}}{\Delta} = \frac{c(a)} {c(-a)}.\ee And so we obtain that \be{rightdiv:5} \frac{\Gamma^{w_a}}{\Gamma} = \frac{\Delta^{w_a}}{\Delta}. \ee An easy induction then gives that for any $w \in W$ we have \be{rightdiv:6} \frac{\Gamma^{w}}{\Gamma} = \frac{\Delta^{w}}{\Delta}, \ee or in other words the element $\Gamma \Delta^{-1} \in \qn$ is $W$-invariant. Now, we may write \be{rightdiv:7} \mc{P}_{v} \,  \Delta^{-1} = ( \sum_{w \in W} [w] \, \Gamma ) \,  \Delta^{-1} = \sum_{w \in W} w \, (\Gamma \Delta^{-1}). \ee As $\Gamma  \, \Delta^{-1}$ is $W$-invariant, we obtain that \be{rightdiv:7} \mc{P}_{v} \,  \Delta^{-1} = \mf{m} \, \sum_w [w] \ee where $\mf{m} \in \qn$ is some $W$-invariant factor. \end{proof}

\tpoint{Factorization of $\mc{P}_{v}$} \label{P:fact}  For each $\lv \in \Lv_+$ we have defined $W_{\lv} \subset W$ to be the stabilizer of the $\lv.$ It is well-known that $W_{\lv}$ is generated by the set of simple reflections which it contains. Hence, by the classification of parabolic subgroups of an affine Lie algebra (see \cite{kac}), there are two possibilities for $W_{\lv}:$ either it is finite, or it is equal to all of $W.$

For $\lv \in \Lv_+$, let $W^{\lv} \subset W$ be a set of minimal coset representatives as in  \S \ref{disass}. Consider now, in analogy to the element $\mc{P}_{v},$ the following two elements, \be{p:two} \begin{array}{lcr}  \mc{P}_{v, \lv} = \sum_{w \in W_{\lv}} v^{\ell(w)} \T_w & \text{ and } &
\mc{P}_{v}^{\lv} = \sum_{w \in W^{\lv}} v^{\ell(w)} \T_w. \end{array} \ee

If $W_{\lv}$ is infinite, then in fact it is equal to $W$ and so $\mc{P}_{v, \lv}= \mc{P}_{v}, $ and $\mc{P}^{\lv}_{v}=1$ in this case. On the other hand, if $W_{\lv}$ is finite, then clearly $\mc{P}_{v, \lv} \in \qn[W],$ and $\mc{P}^{\lv}_{v},$ which is an infinite sum, can be easily see to be a $v$-finite element of $\qn[W]^\v$ using Lemma \ref{P:defn}. Using the defining property of $W^{\lv}$ and (\ref{hecke:2})  we obtain factorizations in $\qn[W]^{\vee}$ \be{p:plam} \mc{P}_{v}  =  \mc{P}_{v, \lv} \, \mc{P}_{v}^{\lv} \ee where if $W_{\lv}$ is infinite the above equality is a tautology. In the case that $W_{\lv}$ is finite, one may verify that $\mc{P}_{v, \lv}$ is an invertible element of $\qn[W]$ , and so we may rewrite (\ref{p:plam}) \be{p:plam:2} \mc{P}_{v} \, \mc{P}_{v, \lv}^{-1}    =  \mc{P}_{v}^{\lv}. \ee In the case that $W_{\lv} = W,$ we loosely interpret the above statement as the tautology $1=1.$

\subsection*{Step 4: Rephrasal and Reassambly}

\tpoint{$p$-adic connection} Using the natural action of $W$ on $\Lv$ we may use formula (\ref{Ta}) to define an action of $\T_a$ , $a \in \Pi$ (and hence also $\T_w$ for $w \in W$ by induction) on $\C_{v}[\Lv].$ For example, when $a \in \Pi$ we have that \be{Ta:lv} \T_a(e^{\lv}) = c(a) e^{w \lv} + b(a) e^{\lv}. \ee Using the above, we see easily that if $w \in W_{\lv}$ then\be{Tw:elam} \T_w(e^{\lv}) = v^{\ell(w)} e^{\lv}. \ee  The following result is the key to linking the algebraic and $p$-adic treatments of the spherical function.

\begin{nprop} \label{alg:padic} Let $\lv \in \Lv_+.$ For $w \in W$ the specialization of the elements $v^{-2 \la \rho, \lv \ra } v^{\ell(w)} \T_w(e^{\lv})$ at $v^2=q^{-1}$ is equal to $J_w(\lv)$ (as defined in \ref{Jw}, \ref{Jbis:J}), i.e., \be{alg:padic:q} q^{\la \rho, \lv \ra } q^{-\ell(w)/2} \T_w(e^{\lv}) = J_w(\lv). \ee \end{nprop}
\begin{proof} The proof will be done by induction on $\ell(w).$ If $\ell(w)=1$ so that $w=w_a$ for $a \in \Pi,$ we are aiming to show that $q^{ -\la \rho, \lv \ra } J_a(\lv)$ is the specialization of $v \T_a(e^{\lv})$ at $v = q^{-1/2}.$ Noting that $J_1(\lv) = q^{ \la \rho, \lv \ra } e^{\lv},$ we find from (\ref{J:w:w'}) that
 \be{Ja} q^{ - \la \rho, \lv \ra }  J_a(\lv) = \frac{1 - q^{-1}e^{a^\v} }{1- e^{a^\v}} e^{w_a \lv} + \frac{q^{-1} -1}{1- e^{a^\v} } e^{\lv}  \ee which is exactly the specialization of (\ref{vTa}) at $v = q^{-1/2}.$ The induction, which follows by comparing (\ref{J:w:w'}) with the similar recursion relations for $\T_w,$ is left to the reader. \end{proof}

\spoint Consider the formal application of the element $\mc{P}^{\lv}_v$ to $e^{\lv},$
\be{P:app}
\mc{P}_{v}^{\lv}(e^{\lv}) := \sum_{w \in W^{\lv} } v^{\ell(w) } \T_w (e^{\lv}).
\ee We would first like to argue that the resulting expression is a $v$-finite expression in $\C_{v, \leq}[\Lv].$ In particular, the coefficient of each $e^{\mv}$ with $\mv \in \Lv$ is a polynomial in $v^2.$ If $W_{\lv}=W$ this is obvious. So suppose that $W_{\lv}$ is finite. We have already argued that we may write $\mc{P}_v^{\lv} = \sum_{\sigma \in W} C_{\sigma} [\sigma]$ with $C_{\sigma} \in \mc{Q}_v$ a $v$-finite element. Applying this to $e^{\lv}$ we obtain \be{P:app-2} \mc{P}_v^{\lv}(e^{\lv}) = \sum_{\sigma \in W } C_{\sigma} e^{\sigma \lv}. \ee For any fixed $\mv,$  since $\lv$ is dominant and $W_{\lv}$ is finite, if $\ell(\sigma)$ is sufficiently large we have $\sigma \lv < \mv.$ Since each $C_{\sigma}$ is an expansion in negative powers of the coroots, it follows that only finitely many terms in the above sum can contribute to each $e^{\mv},$ and the $v$-finiteness of (\ref{P:app}) follows from that of $C_{\sigma}.$

Let us rephrase this observation slightly. Writing \be{v:T} v^{ - 2 \la \rho, \lv \ra} v^{\ell(w)} \T_w(e^{\lv}) = \sum_{\mv} b_{w, \mv}(v^2) e^{\mv}, \ee it is easy to see that $b_{w, \mv}(v^2) \in \C[v^2].$ For fixed $\mv,$ set \be{v:t-3a} b_{\mv}(v^2) := \sum_{w \in W^{\lv} } b_{w, \mv}(v^2). \ee The $v$-finiteness of (\ref{P:app-2}) show that $b_{\mv}(v^2) \in \C[v^2][[Q_-^{\vee}]].$ In fact (see the comments after Lemma \ref{P:defn}) for any fixed $\mv$, only finitely many $b_{w, \mv}(v^2) \neq 0.$ In sum, we write \be{P:1} v^{ - 2 \la \rho, \lv \ra} \mc{P}_v^{\lv}(e^{\lv})  = \sum_{w \in W^{\lv}}  v^{ - 2 \la \rho, \lv \ra} v^{\ell(w)} \T_w(e^{\lv}) = \sum_{\mv \in \Lv} b_{\mv}(v^2) e^{\mv}. \ee

\spoint \label{fact:P:2} Now we distinguish two cases:
\begin{description}
\item[Case (a)] $W_{\lv}$ is finite;
\item[Case (b)] if $W_{\lv}$ is infinite, or in other words $W_{\lv}=W.$
\end{description}
	
	In case (a), we may apply (\ref{p:plam:2}) to conclude that as elements of $\C_{v, \leq}[\Lv]$ we have that \be{case1} \mc{P}_{v}^{\lv}(e^{\lv}) = \mc{P}_{v} \, \mc{P}_{v, \lv}^{-1} (e^{\lv}). \ee However, from (\ref{Tw:elam}) we have  \be{plam:poin} \mc{P}_{v, \lv}(e^{\lv}) = \sum_{w \in W_{\lv}} v^{2 \ell(w)}  \, e^{\lv} = W_{\lv}(v^2) e^{\lv}. \ee Thus \be{P:elam} \mc{P}_{v}^{\lv}(e^{\lv}) = \frac{1}{W_{\lv}(v^2)} \mc{P}_{v}(e^{\lv}). \ee In case (b), we note that $W^{\lv}$ is the identity element, and we again have the equality (quite tautologically now) \be{case:b} \mc{P}^{\lv}_{v} (e^{\lv}) = \frac{1}{W_{\lv}(v^2)} \mc{P}_{v}(e^{\lv}).\ee
		
\tpoint{Relating $\mc{P}_v^{\lv}(e^{\lv}) $ and $S(h_{\lv})$}	 Recall the $b_{\mv}(v^2)$ defined in (\ref{v:t-3a}) and (\ref{P:1}). We would like to argue that $b_{\mv}(q^{-1})$ is finite and that it equals the $e^{\mv}$-coefficient of $S(h_{\lv}).$

We begin by revisiting the polynomials $\Phi_{w, \mv}(\cdot)$ of Proposition \ref{recursion} (4).  From Proposition \ref{alg:padic}, we know that \be{v:t-1} b_{w, \mv}(q^{-1}) = \Phi_{w, \mv} \ee for any $q,$ where the right hand side is defined in Proposition \ref{recursion}(1). Hence from Proposition \ref{recursion} (4), we may conclude that \be{v:t-2} b_{w, \mv}(q^{-1}) = \Phi_{w, \mv}(q^{-1}) \ee holds for every $q.$ Thus as polynomials, we have \be{v:t-3} b_{w, \mv}(v^2) = \Phi_{w, \mv}(v^2) ,\ee and so as elements in $\C[[v^2]]$ we must have \be{b} b_{\mv}(v^2) =  \sum_{w \in W^{\lv} } \Phi_{w, \mv}(v^2). \ee But in fact, the above is an equality in $\C[v^2]$ since the left hand side lies in this smaller ring. For later use, let us set \be{phi} \Phi_{\mv}(v^2) = \sum_{w \in W^{\lv} } \Phi_{w, \mv}(v^2) \in \C[v^2].\ee

From Step 2 (see (\ref{S:w}) )  we have written $S(h_{\lv}) = \sum_{w \in W^{\lv}} J_w(\lv)$ as elements in $\C_{\leq}[\Lv],$ which can be further written using Proposition \ref{recursion} (4) as  \be{S:1} S(h_{\lv})  = \sum_{w \in W^{\lv}} \sum_{\mv \in \Lv} \Phi_{w, \mv}(q^{-1}) e^{\mv}. \ee Since each $\Phi_{w, \mv}(q^{-1}) \in \mathbb{N},$ we may conclude the following from the fact that the coefficient of each $e^{\mv}$ in $S(h_{\lv})$ is finite:  for any fixed $\mv$ there are only finitely many $w \in W^{\lv}$ such that $\Phi_{w, \mv}(q^{-1}) \neq 0$. Note that we could also conclude this fact from the fact that only finitely many $b_{w, \mv}(v^2) \neq 0.$  In any case, the sum \be{S:3} \Phi_{\mv}(q^{-1}) = \sum_{w \in W^{\lv}} \Phi_{w, \mv}(q^{-1}) \ee is finite for any $\mv,$ and we may write \be{S:4} S(h_{\lv}) = \sum_{\mv \in \Lv} \Phi_{\mv}(q^{-1}) e^{\mv}. \ee On the other hand, using (\ref{b}), we may also write \be{S5} S(h_{\lv}) = \sum_{\mv \in \Lv} \Phi_{\mv}(q^{-1}) e^{\mv} = \sum_{\mv \in \Lv} b_{\mv}(q^{-1}) e^{\mv} .\ee  In other words $S(h_{\lv})$ is the specialization of $v^{-2 \la \rho, \lv \ra } \mc{P}^{\lv}_v(e^{\lv})$ at $v^2=q^{-1}.$ Hence, using the results of \S \ref{fact:P:2}, $S(h_{\lv})$ is also then equal to the specialization of \be{prop:1} v^{-2 \la \rho, \lv \ra } \frac{1}{W_{\lv}(v^2)} \mc{P}_{v}(e^{\lv}) = v^{-2 \la \rho, \lv \ra } \frac{\mf{m}}{W_{\lv}(v^2)} \sum_{w \in W} \Delta^w e^{w \lv} \ee where the equality in the above line follows from Cherednik's identity Proposition \ref{proportional}.  We shall write this as \be{prop:2} S(h_{\lv}) = q^{ \la \rho, \lv \ra } \frac{\mf{m}}{W_{\lv}(q^{-1})} \sum_{w \in W} \Delta^w e^{w \lv}, \ee where the right hand side is the specialization of (\ref{prop:1}) at $v=q^{-1/2.}$
	
	Finally it remains to determine $\mf{m}$ (which is independent of $\lv$). This is achieved by evaluating both sides of (\ref{prop:2}) at $\lv=0$ (and $v^2=q^{-1}$.)  By definition of convolution, we must have $S(h_0)=1,$ and so (\ref{prop:2}) reduces to \be{rhs:ct} 1= \frac{\mf{m}}{W(q^{-1})} \sum_{w \in W} \Delta^w.\ee We then have $\mf{m} = H_0^{-1}$ and using (\ref{def:corr}), and the proof of the theorem is concluded.

\begin{appendices}

\section{The Cartan Semigroup} \label{car-app}

This appendix is devoted to the study of the semigroup $G_+$ of \S \ref{intro-Gplus}.  In addition to supplying a proof of Theorem \ref{cartan-thmdefn} and Corollary \ref{semigp-iwa}, we shall give another characterization of $G_+$ in terms of bounded elements with respect to the norms introduced in \S \ref{sec-repth}. We begin by defining the following subset of $G$, which we aim to show is equal to $G_+,$  \be{Gplus} G'_+:= \cup_{\lambda \in \Lambda_+^{\vee}} K \pi^{\lv} K. \ee In order to work effectively with this set one needs to verify that $G'_+$ is in fact a semi-group. In the process of showing this, we shall also see that $G'_+$ is in fact equal to $G_+.$ The techniques which we employ here are based on Garland \cite{gar:car}, with an simplification stemming from the work in \cite[Lemma 3.3]{bgkp} which gives an effective criterion for detecting when certain unipotent elements are integral.

\subsection{The Semigroup of Bounded Elements} Fix the notation as in \S\ref{sec-repth}, and consider the subset of $G$ defined as follows \begin{eqnarray} \label{Gb} G_b:= \{ g \in G \mid  \; \max_{v \in V^{\omega}_{\O} } || g v || < \infty \text{ for every highest weight representation }  V^{\omega} , \omega \in \Lambda_+ \}. \end{eqnarray} Note that if $g \in G_b$ and $v_{\omega}$ is a primitive highest weight vector of $V^{\omega}$ then we also have that \be{max:k} \max_{k \in K } || g k v_{\omega} || < \infty \ee since $k v_{\omega} \in V^\omega_\O.$ This shows that $G_b$ is right $K$-invariant. In fact, $G_b$ is $K$ bi-invariant, as the left $K$-invariance follows from the way we have defined the norms $|| \cdot ||$ on $V^\omega.$ We call $G_b$ the set of \emph{bounded} elements of $G,$ and observe the following simple result.

\begin{nlem} \label{bounded:semigp} The set $G_b$ is a semigroup.  \end{nlem} \begin{proof} Indeed, let $g_1, g_2 \in G_b.$ Fix a highest weight representation $V^{\omega}$ and suppose that there exists an integer $M$ such that $|| g_2 v || \leq q^M$ for every $v \in V^\omega_\O.$  In other words, $g_2 V^\omega_\O \subset \pi^{-M} V^\omega_\O,$ and hence $|| g_1 g_2 v || \leq q^M || g_1 v ||$ for any $v \in V^\omega_\O.$ The lemma is proven. \end{proof}

\subsection{Relation of $G_b$ to the Tits Cone}

Let us next record the relation between $G_b$ and the Tits cone.

\begin{nprop} \label{dom:bound} Let $\lambda^{\vee} \in \Lambda^\vee.$ We have that $\pi^{\lambda^{\vee}} \in G_b$ if and only if $\lambda^{\vee} \in - X.$ \end{nprop}
\begin{proof} Suppose first that $\lambda^{\vee}= - \lambda^{\vee}_+$ for $\lambda^{\vee}_+ \in \Lambda^{\vee}_+.$ Fix a highest weight representation $V^{\omega}$ with weight lattice $P_\omega.$ Then for any $v \in V^\omega_\O$ belonging to the weight space $\mu \in P_\omega$ we have \be{bd:1} || \pi^{\lv} v || \leq q^{- \langle \mu, \lv \rangle }. \ee On the other hand, every weight $\mu \in P_{\omega}$ is of the form $\mu = \omega - \beta$ where $\beta \in Q_+.$ Hence, \be{} \langle \mu, \lv \rangle = - \langle \omega, \lambda^{\vee}_+ \rangle + \langle \beta, \lambda^{\vee}_+  \rangle. \ee Thus $\langle \mu, \lambda^{\vee} \rangle$ is bounded below and so $\pi^{-\lambda^{\vee}_+} \in G_b.$ By $K$-binvariance, we also have that $\pi^{-w \lambda^{\vee}_+} \in G_b$ for any $w \in W.$ Conversely, if $\pi^{\lv} \in G_b$ then the same argument as above shows that $- \langle \beta, \lv \rangle$ is bounded below as $\beta \in Q_+$ varies over the same set as above. From this one can conclude that $\lv$ must be in $-X.$ \end{proof}

From the fact that $G_b$ is $K$-binvariant, we obtain

\begin{ncor} Let $\lv \in \Lv.$ Then the coset $K \pi^{\lv} K \in G_b$ if and only if $\lv \in - X.$ \end{ncor}

\subsection{Relation between $G_b$ and $G'_+$}

Our next goal is to relate $G_b$ and the set $G'_+$ defined in (\ref{Gplus}). To do so, it is convenient to also define \be{reform:semi:2} G'_-:= (G'_+)^{-1}:=  \cup_{\lv \in \Lv} K \pi^{-\lv} K . \ee

\begin{nprop} \label{cartan:reform}  We have an equality of sets (and hence semigroups)  $G'_- = G_b,$ and hence also $G'_+=G_b^{-1}.$  \end{nprop} \begin{proof} The second statement follows immediately from the first. Note that as both $G'_-$ and $G_b$ are $K$-invariant sets, the fact that $G'_- \subset G_b$ follows immediately from Lemma \ref{dom:bound}. It remains to show that $G_b \subset G'_-.$ To do this, it suffices to show that every element of $g \in G_b$ has an expression $g = k_1 \pi^{\mv} k_2.$ Indeed, if this were the case, then by the Lemma \ref{dom:bound} above (and using $K$-binvariance), we necessarily have that $\mv \in -X.$ So the proof of the Proposition can be concluded from the following result, an alternate proof of which can be found in \cite[(2.8)]{gar:car}.

\begin{nlem} \label{unip:cartan} Let $g \in G_b.$ Choose $k \in K$ which maximizes the norm $|| g k v_\omega ||$ for some fixed representation $V_\omega$ with primitive highest weight vector $v_\omega.$  Then if we write $gk= k_1 a u$ in terms of its Iwasawa coordinates, we must have $u \in K.$ \end{nlem}

\begin{proof}[Proof of Lemma \ref{unip:cartan}]  Recall from \cite[Lemma 5.2.1]{bgkp} that we had a decomposition of $U^-$ into \emph{disjoint} subsets $U^-_w$ where $U^-_1 = U^-_{\O}$ and in general, if $u^- \in U^-_w$ then it has an expression \[\begin{array}{lcr} u^- = k \pi^{\xi^{\vee}} U, & \xi^{\vee} \geq 0, & | \xi ^{\vee} | \geq l(w)/2. \end{array}, \] where $| \xi^{\vee} | = \la \rho, \xi^{\vee} \ra.$  From this it follows that \be{b:1} \begin{array}{lcr} U^- \cap K U \subset K \cap U^- & \text{ and } &  U \cap K U^- \subset K \cap U. \end{array} \ee  Now given $g$ and $k$ as in the lemma, assume that $gk = k_1 a u$ with $ u \notin K.$ Then write an opposite Iwasawa decomposition (in terms of $G= U^- A K$) for $u:$ i.e., \be{} u = u^- \pi^{\xi^{\vee}} k_2. \ee By the above, it follows that $\xi^{\vee} \neq 0.$ In fact $\xi^{\vee} < 0$ actually: we also know from \cite[Proposition 3.3.1]{bgkp} that $$K U^- \cap K \pi^{\lambda^{\vee}} U \neq \emptyset$$  implies that $\lambda \leq 0$, just take inverses in this expression. So we have \begin{eqnarray}
 || g k k_2^{-1} v_{\omega} || &=& || k_1 a u k_2^{-1} v_{\omega} || \\
 &=& || a u^{-} \pi^{\xi^{\vee}} v_{\omega} || \\
 & \geq &  q^{- \la \omega, \xi^{\vee} \ra} || a v_{\omega} || = q^{- \la \omega, \xi^{\vee} \ra} || g k v_{\omega} || \end{eqnarray} contradicting the original choice of $k,$ since $\la \omega, \xv \ra < 0.$  Hence $u \in K.$
\end{proof}
The proof of the Proposition is thus also completed. \end{proof}

\subsection{Relating $G'_+$ and $G_+$ (Proof of Theorem \ref{cartan-thmdefn} )}

In the introduction \S \ref{intro-Gplus}, we defined a semi-group $G_+.$ Recall again that we have defined a map $| \eta|: G \rr \zee$ in (\ref{eta:Kpts}). Recall that $\eta$ was defined with respect to the description of $G$ as a semi-direct product $G= \mc{K}^* \rtimes G'$ by projection onto the $\mc{K}^*$ factor and then composing with the valuation map $\mc{K}^* \rr \zee.$ Writing an element $g \in G$ with respect to the Iwasawa decomposition $g = u \pi^{\lv} k$ where $u \in U, \lv \in \Lv, k \in K$ one can easily verify that \be{eta:iwa} |\eta(g)|= | \eta(\pi^{\lv}) | = \la \delta, \lv \ra,  \ee where $\delta$ was the minimal positive imaginary root.

\begin{nprop} \label{intro-semigp} The semigroup $G'_+$ is equal to $G_+,$ the semigroup defined in (\ref{intro-Gplus}).  \end{nprop}

\begin{proof} If $\lv \in \Lambda_+$ then certainly $|\eta(\pi^{\lv}) | > 0$ (in fact this is true for all $\lv \in X$ by our description (\ref{ti:exp}) ). From this it follows easily that $G'_+ \subset G_+.$

As for the opposite inclusion, it is clear that $K \subset G'_+.$ To show that central $\mc{K}^* \subset G$ is contained in $G'_+$ we proceed as follows. If $\sigma^\cc $ is such a central element with $\sigma \in \mc{K}^*,$ $\sigma = \pi^m u$ with $u$ a unit and $m \in \zee,$ then  $u^\cc \in K$ and $m \cc \in \Lambda_+.$ Thus $\sigma^\cc= u^\cc \pi^{m \cc} \in G'_+.$

It remains to show that every $g \in G$ with $| \eta(g) | > 0$ lies in $G'_+.$ For any such $g$ write an Iwasawa decomposition $g = u \pi^{\mv} k$ with $k \in K$, $\mv \in \Lv$ and $u \in U.$  As observed above, $|\eta(g)| = |\eta(\pi^{\mv})| = \la \delta, \mv \ra  > 0.$ Hence $\mv \in X$ by (\ref{ti:exp}). Consider now the element $g^{-1} = k^{-1} \pi^{-\mv} u^{-1}.$ We may conclude that $g^{-1} \in G_b$ from the following result of Garland,

\begin{nthm}\cite[Theorem 1.7]{gar:car}  \label{gar:bd} The set of element $K \pi^{ - \mv} U$ with $\mv \in X$ lie in $G_b.$ \end{nthm}

Thus for each $g \in G$ with $| \eta(g) | > 0$ we have shown that $g^{-1} \in G_b$ and so $g \in G'_+$ using Proposition \ref{cartan:reform}.

\end{proof}

\subsection{Relation to Iwasawa Semigroup}

\begin{nprop} \label{iwa-rel} The set of element $K \pi^{-\mv} U$ with $\mv \in X$ is equal to $G_b.$ \end{nprop}
\begin{proof} From the above quoted result of Garland, Theorem \ref{gar:bd}, elements of the form $K \pi^{-\mv} U$ with $\mv \in X$ lie in $G_b.$ Suppose now that $g \in G_b,$ and suppose it were written in the form $g = k \pi^{\mv} u$ with $k \in K, \mv \in \Lv, u \in U.$ Combining Propositions \ref{cartan:reform} and \ref{intro-semigp}, we have that $G_b^{-1} = G_+.$ So $g^{-1}  = u^{-1} \pi^{- \mv} k^{-1}$ must lie in $G_+.$ This implies that $- \mv \in X$ by (\ref{eta:iwa}) and so $\mv \in - X$ as desired. \end{proof}

\section{The "affine" Root System and the Bruhat pre-order on $\mc{W}$} \label{app-aff}

The goal of this appendix is to introduce a notion of an "affinized" root system attached to a Kac-Moody root system and study some of its basic properties.

\subsection{"Affine" Roots}  Recall that $R_{re}$ was the set of real roots of $\f{g}.$  Let us define four subsets of the set of \emph{"affine" roots} $\ar:= R_{re} \times \zee$ as follows,
\begin{eqnarray}
\ar_+^+ &:=& \{ (a, k) \in R_{re} \times \zee : a > 0, k \geq 0 \} \\
\ar_-^+ &:=& \{ (a, k) \in R_{re} \times \zee : a > 0, k < 0 \} \\
\ar_+^- &:=& \{ (a, k) \in R_{re} \times \zee : a < 0, k > 0 \} \\
\ar_-^- &:=& \{ (a, k) \in R_{re} \times \zee : a < 0, k \leq 0 \} \end{eqnarray}

Note that upper indices shall refer to Kac-Moody parameters, and lower ones to the local field. We also define the set of \emph{positive and negative "affine" roots} as \be{ar:pm} \begin{array}{lcr} \ar^+:= \ar_+^+ \cup \ar_-^{+} & \text{ and } & \ar^-:= \ar_+^- \cup \ar_-^{-}. \end{array} \ee We shall sometimes write write $\alpha= a + k \pi,$ with $a
\in R_{re}$ and $k \in \zee$ to denote the pair $(a, k) \in \mc{R}.$

There is a left action of $\aw$ (see (\ref{aw-X}) ) on $\ar$ defined through the formula \be{left:aw} w \pi^{\lv}.(a +k \, \pi) = w \cdot a + (\langle \lv, a \rangle + k) \pi \ee where $a \in R_{re}, k \in \zee$ and $w \pi^{\lv} \in \aw.$

In the usual setting of Coxeter groups, the length of an element $x \in \mc{W}$ could be defined as the size of the set $x \mc{R}^+ \cap \mc{R}^-$ or $x \mc{R}^- \cap \mc{R}^+.$ In our setting, however these sets are not finite, as $x \ar_+^- \cap \ar^-$ and $x \ar_-^-  \cap  \ar^+$ are not finite in general. On the other hand, if we restrict to $\aw_X$ the following is true.

\begin{nprop} \label{half-length} Let $x \in \aw_X.$ Then the following sets are finite \be{half:fin} \begin{array}{lcr}  x \ar_+^+ \cap \ar^- \text{ and } & x \ar_-^+  \cap & \ar^+ \end{array} \ee
\end{nprop}

\begin{proof} Let us prove the first statement, the proof of the second being similar.  Write $x = w \pi^{\lv}$ for $\lv \in X,$ and let $a+ k \pi \in \ar_+^+$ (so $a>0, k \geq 0).$ If $x . (a +k \pi) \in \ar_-$ then we must have from (\ref{left:aw}) that either, \begin{eqnarray} w a > 0, &\text{ and }& \langle \lv, a \rangle + k < 0 \\
 w a < 0 & \text{ and } &\langle \lv, a \rangle + k \leq 0 \end{eqnarray}  Since $\lv \in X$ for any integer $n$, the number of roots $a > 0$ such that $\langle \lv, a \rangle < n$ is finite in number, and so there are only finitely many $a$ which satisfy either of the above two conditions.  On the other hand, for any such a fixed $a$ there are only finitely many $k$ such that either equation will be satisfied.  The finiteness required follows. \end{proof}

\subsection{Another pre-order on $\mc{W}$}

We shall also consider an action of $\aw$ on the right on $\ar,$  \be{right:aw} (a + n \pi). \pi^{\lv} w := w^{-1}(a) + (n- \langle \lambda, a \rangle) \pi, \ee where $a \in R_{re}, n \in \zee.$ Let us also introduce the following simple elements in $\mc{W},$ attached to $\alpha = a + n \pi \in \mc{R},$ \be{w:alp} w_{\alpha}:= w_{a}(n):= w_{a} \pi^{n a^{\vee}} \ee which satisfy \be{mult:simp} w_\alpha  \pi^{\lv} w = w_{a}(n) \pi^{\lv} w = \pi^{ w_{a} \lambda - n a} w_{a} w. \ee

\begin{de} Given $x \in \aw$ and $\alpha = a + n \pi \in \ar$ we say that $\alpha$ is $x$-\emph{positive} or $x$-\emph{negative} if $\alpha . x \in \ar_{\pm}.$  For $x, y \in \aw,$ we shall say that $y \, \leq_B \, x$  if there exists $\alpha_i:= a_i + n_i \pi \in \ar$ for $i=1, \ldots , k $ such that $y = w_{\alpha_k} \cdots w_{\alpha_1} x,$ and  \be{alpha:pos} \alpha_1 \text{ is } x \text{-negative and} \,  \alpha_j  \text{ is } w_{\alpha_{j-1}} \cdots w_{\alpha_1} x \text{-negative for } j=2, \ldots, k. \ee \end{de}

\emph{Remarks:} We do not know whether the relation $\leq_B$ is an order. It is clear that it is a pre-order, namely if $x \leq_B y$ and $y \leq_B z,$ then $x \leq_B z$ for $x, y, z \in \mc{W}.$ However, we do not know whether $x \leq_B y$ and $y \leq_B x$ implies that $x =y.$ It would also be interesting to understand: a) the structure of the set of elements which are $\leq_B$ to a fixed $x \in \mc{W};$ and b) the relation of $\leq_B$ and $\preceq$ of Definition \ref{aw:domord}.

\subsection{The order $\leq_B$ and Iwahori intersections} Let $x \in \mc{W}_X$ and $y \in \mc{Y}.$ Then we have seen in Proposition \ref{domorder} that $\leq$ arises naturally when one considers the intersection $U y I \cap I x I.$ We also have the following result which relates these same intersections to $\leq_B,$

\begin{nprop} \label{bruhatorder} Let $x \in \aw_X$ and $y \in \aw$ such that $I x I \cap U y I \neq \emptyset.$ Then $y \leq_B x.$ \end{nprop}

\begin{proof} As $\I= U_{\O} U^-_{\pi} A_{\O}$ we first note \be{bo:1} I x I \cap U y I  \neq \emptyset \iff U^-_{\pi} x \I \cap U y \I \neq \emptyset. \ee  Let us write $x= \pi^{\lv} \sigma,$ $\sigma \in W$ and $\lv= w \lv_+$ with $\lv_+ \in \Lv_+$ and $w \in W$ (we may do this since we assumed that $x \in \aw_X.$)

\emph{Step 1: Reduction to a Finite Dimensional Problem.}  The first thing we would like to show is that in analyzing the intersection (\ref{bo:1}) above, we may actually replace $U^-_{\pi}$ by a subgroup which is a product of only finitely many root groups.  To explain this, we break up \be{bo:2} R_{-, re} = R_{-,w} \sqcup R_{-}^w \ee where $R_{-, w}$ and $R_-^w$ are characterized by the conditions \be{bo:3} \begin{array}{lcr} w^{-1} R_{-, w} \subset R_+ & \text{ and } & w^{-1} R_{-}^w \subset R_- \end{array}. \ee Corresponding to the above decomposition, we have a product decomposition
$U^- = U_{-w} U_{-}^w,$ where \be{bo:5} \begin{array}{lcr} U_{-w} := \prod_{-\beta \in R_{-, w} } U_{\beta} & \text{ and } & U_-^w:= \prod_{\beta \in R_-^w } U_{\beta} \end{array}. \ee This in turn implies a decomposition, \be{bo:4} U^-_{\pi} = U_{-w, \pi} U_{-, \pi}^w.\ee

 Now, suppose we are given $\beta \in R_{-}^w$ with $\chi_{\beta}(s) \in U_{-, \pi}^w$ with $s \in \mc{K}$ such that its valuation $\valu(s)= \ell$ (so necessarily $\ell \geq 1.$)  Then \be{bo:5} \sigma^{-1} \pi^{-\lv} \chi_{\beta}(s) \pi^{\lv} \sigma = \chi_{\sigma^{-1} (\beta) } (\pi^{\langle \lv, -\beta \rangle }s ), \ee where \be{bo:6} \valu(\pi^{\langle \lv, -\beta \rangle }s) = \langle \lv, -\beta \rangle + \ell. \ee   Recalling that $\lv = w \lv_+$ we have \be{bo:7} \langle \lv, -\beta \rangle = - \langle \lv_+, w^{-1} \beta \rangle \geq 0 \ee since $w^{-1} \beta \in R_-$ (as $\beta \in R^w_-$).  Hence $\chi_{- \sigma^{-1} (\beta) } (\pi^{\langle - \lv, \beta \rangle} s ) \in \I$ since $-\langle \lv, \beta \rangle + \ell > \ell \geq  1.$ So, we have seen that \be{bo:8} U y \I \cap U^-_{\pi} x I \neq \emptyset \iff U_{-w, \pi} \pi^{\lv} \sigma \I \cap U y \I \neq \emptyset. \ee We shall now study the intersection problem (\ref{bo:8}) where we replace the integral group $U_{-w, \pi}$ with the larger group $U_{-w, \mc{K}} = U_{-w},$ i.e., we are now analyzing the problem of when \be{bo:9} U_{-w} \pi^{\lv} \sigma \I \cap U y \I \neq \emptyset. \ee This is a problem more tractable to a "Gindikin-Karpelevic"-type induction.

\emph{Step 2: Some finite "Gindikin-Karpelevic"-combinatorics.} Before proceeding further, we recall some simple combinatorial facts about the group $U_{-w}.$ It is easy to see that $R_{-, w} = R_{-} \cap w R_+.$ Further, if $w = w_{a_1} \cdots w_{a_r}$ with $a_i \in \Pi$ then \begin{eqnarray} R_{-, w} &=& \{ - a_1, -w_{a_1}(a_2), \cdots, -w_{a_1} \cdots w_{a_{r-1}}(a_r) \}  \\
&=& R_{-, w'} \cup \{ - w_{a_1} \cdots w_{a_{r-1}} (a_r) \} \end{eqnarray} where $w' = w w_{a_r} = w_{a_1} \cdots w_{a_{r-1}}.$ Let $\gamma = w_{a_1} \cdots w_{a_{r-1}}(a_r).$  Then we have,

\begin{nclaim} \label{bo:lem} Suppose that $x_{\gamma} \in U_{\gamma}$ and $u_{w'} \in U_{w'}.$ Then $$x_{\gamma}^{-1} u_{-w'} x_{\gamma} \in U(w') U_{-, w'}$$ where $U(w'):= w' U w'^{-1} \cap U \subset U$ \end{nclaim}
\begin{proof} Note that $w' U w'^{-1} = U(w') U^-(w')$  where we set $U^{\pm}(w'):= w' U w'^{-1} \cap U^{\pm}.$  Now $x_{\gamma} \in U(w')$ since $\gamma= w'(\alpha_r).$ Furthermore, $U^-(w')=U_{-, w'}$ as is easily verified.  \end{proof}

\emph{Step 3: Relation to positivity.} We may write $u_{-w} \in U_{-w, k}$ as \be{bo-3:1} u_{-w} = u_{-w'} u_{- \gamma} \ee where $u_{-w'} \in U_{-w'}$ and $u_{\gamma}= x_{-\gamma}(s) \in U_{- \gamma}$ with $\valu(s)=\ell \in \zee.$ Then there are two possibilities,
\begin{description}
\item[(a)] We have $\sigma^{-1} \pi^{-\lambda} x_{-\gamma}(s) \pi^{\lv} \sigma \in \I.$ Equivalently we may phrase this as saying that $- \gamma + \ell \pi$ is $\pi^{\lv} \sigma$-positive.

\item[(b)] We have $\sigma^{-1} \pi^{-\lambda} x_{-\gamma}(s) \pi^{\lv} \sigma \notin \I.$ Equivalently we may phrase this as saying that $- \gamma + \ell \pi$ is $\pi^{\lv} \sigma$-negative.

\end{description}

If we are in case (b), we may rewrite using (\ref{sl2}) \be{sl2:1} x_{-\gamma}(s) = x_{\gamma}(s^{-1}) w_{\gamma} (-s)^{\gamma^{\vee}}  x_{\gamma}(s^{-1}) \ee and so \begin{eqnarray}
 u_{-w'} x_{-\gamma}(s) \pi^{\lv} \sigma \I &=& u_{-w'} x_{\gamma}(s^{-1}) w_{\gamma} (-s)^{\gamma^{\vee}}x_{\gamma}(s^{-1}) \pi^{\lv} \sigma \I \\
 &=&  u_{-w'} x_{\gamma}(s^{-1}) w_{\gamma} \pi^{ \ell \gamma^{\vee}} \pi^{\lv} \sigma \I \end{eqnarray} where we have used the condition (b) in the last line. But using the Claim \ref{bo:lem} above, the last expression may be written in the form \be{bo-3:2} U(w') \tilde{u}_{-w'}  w_{\gamma}(l) \pi^{\lv} \sigma \I,\ee (note: $x_{\gamma} \in U(w')$) where we recall our notation that $w_\gamma(\ell) = w_{\gamma} \pi^{ \ell \gamma^{\vee}}.$ In summary, we have shown the following,

\begin{nclaim} \label{bo:ind} In the notation above, $u_{-w} \in U_{-w}$ is such that $u_{-w} \pi^{\lv} \sigma \I \in U y \I. $ Then either,
\begin{description}
\item[a)] We have $u_{-w'} \pi^{\lv} \sigma \I \in U y \I$

\item[b)] We have $u_{-w'} w_{\gamma}(\ell) \pi^{\lv} \sigma \I \in U y \I.$
\end{description}

\noindent Furthermore, the case (b) occurs if $- \gamma + l \pi$ is $\pi^{\lv} \sigma$-negative. \end{nclaim}

The proof of Proposition \ref{bruhatorder} follows from an easy induction using the previous Lemma \ref{bo:ind}.

 \end{proof}

\end{appendices}

%

\end{document}